\documentclass[12pt]{amsart}

\usepackage[a4paper, total={6.2in, 9in}]{geometry}

\usepackage{amsfonts,amssymb,amsmath,amscd,amstext}
\usepackage[colorlinks=true,linkcolor=blue,citecolor=blue]{hyperref}
\usepackage{graphicx}
\usepackage{comment}
\usepackage[dvipsnames]{xcolor}
\usepackage{natbib}
\usepackage{graphicx}
\usepackage{verbatim}
\usepackage{ dsfont }
\usepackage{ mathrsfs }
\usepackage{hyperref}
\usepackage{subfiles}
\usepackage{xcolor}

\usepackage{tikz}
\usepackage{tikz-cd}
\usepackage{pgfplots}
\pgfplotsset{compat=1.18}

\pgfkeys{/pgf/trig format=rad}

\newcommand{\al}{\mu^{su}}

\renewcommand{\leq}{\leqslant}
\renewcommand{\geq}{\geqslant}

\newcommand{\rr}{{\mathbb{R}}}
\newcommand{\cc}{{\mathbb{C}}}

\newcommand{\hh}{{\mathbb{H}}}

\newcommand{\nn}{{\mathbb{N}}}

\newcommand{\eps}{\varepsilon}

\renewcommand{\leq}{\leqslant}
\renewcommand{\geq}{\geqslant}

\newcommand{\Psisu}{\Psi_\varepsilon^{su}}

\newcommand{\Nsua}{\mathscr N_{\varepsilon,1}^{su}}
\newcommand{\Nsub}{\mathscr N_{\varepsilon,2}^{su}}
\newcommand{\Nsuc}{\mathscr N_{\varepsilon,3}^{su}}
\newcommand{\Nsuj}{\mathscr N_{\varepsilon,j}^{su}}

\newcommand{\Nea}{  \widehat N_{\varepsilon,1} }
\newcommand{\Neb}{  \widehat N_{\varepsilon,2} }
\newcommand{\Nec}{ \widehat N_{\varepsilon,3} }
\newcommand{\Nej}{ \widehat N_{\varepsilon,j} }
\newcommand{\er}{{\bar \eps}}

\newcommand{\Fep}{F_\varepsilon}

\DeclareMathOperator{\divv}{div}

\DeclareMathOperator{\ric}{Ric}

\newtheorem{theorem}{Theorem}[section]
\newtheorem{proposition}[theorem]{Proposition}
\newtheorem{lemma}[theorem]{Lemma}

\theoremstyle{definition}

\newtheorem{remark}[theorem]{Remark}

\numberwithin{equation}{section}
\theoremstyle{remark}

\title{Fuglede's theorem in the Heisenberg group}

\author[D.~Massenz]{Davide Massenz}
\address[Davide Massenz]{Dipartimento di Matematica “Tullio Levi-Civita” \\
Università di Padova \\ Via Trieste, 63, 35131, Padova, Italy}
\email{massenz@math.unipd.it}

\author[R.~Monti]{Roberto Monti}
\address[Roberto Monti]{Dipartimento di Matematica “Tullio Levi-Civita” \\
Università di Padova \\ Via Trieste, 63, 35131, Padova, Italy}
\email{monti@math.unipd.it}

\author[J.~Pozuelo]{Juli\'an Pozuelo} 
\address[Juli\'an Pozuelo]{Departamento de Matemáticas, Universidad de Alicante\\
Campus de Sant Vicent del Raspeig s/n, 03690 Sant Vicent del Raspeig, Spain
}
\email{julian.pozuelo@ua.es}

\date{\today}

\begin{document}

\newcommand{\Se}{S_\varepsilon}
\newcommand{\So}{S_0}
\newcommand{\ee}{\mathrm{e}}
\newcommand{\se}{\sigma}
\newcommand{\Be}{B_\varepsilon}
\newcommand{\lame}{\lambda_\varepsilon}
\newcommand{\Ve}{\mathscr {V}_\varepsilon}
\newcommand{\Ae}{\mathscr{A}_\varepsilon}
\newcommand{\Le}{\mathscr{L}_\varepsilon}
\newcommand{\Fe}{\mathscr{F}_\varepsilon}
\newcommand{\Je}{J_\varepsilon}
\newcommand{\qe}{q_\varepsilon}
\newcommand{\Phie}{\Phi_\varepsilon}
\newcommand{\Psie}{\Theta_\varepsilon}
\newcommand{\gep}{g_\varepsilon}
\newcommand{\ue}{u_\varepsilon}
\newcommand{\mue}{\mu_\varepsilon}
\newcommand{\Ce}{C_\varepsilon}
\newcommand{\Ne}{{N}_\varepsilon}
\newcommand{\be}{b_\varepsilon} 
\newcommand{\Fd}{\mathscr {F}_ {\varepsilon}^M} 
\newcommand{\F}{\mathscr {F}_ {\varepsilon}} 
\newcommand{\ea}{\eta_\varepsilon} 
\newcommand{\de}{\mathrm{dist}_\varepsilon} 
\newcommand{\nue}{A_\varepsilon} 
\newcommand{\T}{T}

 \renewcommand{\theta}{\vartheta}

\newcommand{\NP}{P_\varepsilon} 
\newcommand{\hep}{h_\varepsilon}

\newcommand{\Nsu}{\mathscr{N}_\varepsilon^{su}}
\newcommand{\Nu}{\mathscr{N}_\varepsilon^{u}}

\newcommand{\Hsu}{H_\varepsilon^{su}}
\newcommand{\hsu}{h_\varepsilon^{su}}

\newcommand{\Bo}{B_0}
\newcommand{\Sr}{S_{0,\varrho}}
\newcommand{\A}{\mathcal A}
\newcommand{\Su}{S_0^u}
\newcommand{\Bu}{B_0^u}
\renewcommand{\rho}{\varrho}
\newcommand{\Ao}{\mathscr {A}_0} 
\newcommand{\ex}{\mathrm{e}}
\newcommand{\nablaSo}{\nabla\! _{S_0} }
\newcommand{\nablaSe}{\nabla\! _{S_\varepsilon} }

 \newcommand{\fa}{\beta_1}
 \newcommand{\fb}{\beta_2}
 \newcommand{\fc}{\beta_3}
 \newcommand{\fd}{\beta_4}

 \newcommand{\Te}{T_\varepsilon}

%%%%% LARGE COMMAND here
%\large
%%%%% LARGE COMMAND here

\maketitle
 
\begin{abstract}
We prove a Fuglede-type quantitative isoperimetric inequality for the
Pansu sphere, the conjectured isoperimetric set in the sub-Riemannian
Heisenberg group $\hh^1$. We show that, among suitable
volume-preserving normal perturbations of the Pansu sphere that
are small in a weighted $C^2$-norm, the difference between the
sub-Riemannian perimeters of the perturbed and unperturbed regions is
bounded below by a positive constant times the square of their
$L^1$-distance. The proof is based on a Riemannian approximation
argument with uniform estimates.
\end{abstract}

%\tableofcontents

\section{Introduction}

The isoperimetric inequality in Euclidean space $\rr^n$ states that,
among sets of prescribed volume, Euclidean balls minimise the perimeter,
with equality only for balls. A strengthened form is the
\emph{quantitative isoperimetric inequality}, which asserts that there
exist constants $C,q>0$ such that, for every ball $B\subset\rr^n$ and
every set $E\subset\rr^n$ of finite perimeter satisfying $|E|=|B|$,
\[
\frac{P(E)-P(B)}{P(B)}
\geq
C\inf_{x\in\rr^n}
\left(
\frac{|E\triangle(B+x)|}{|B|}
\right)^q.
\]
Here, $E\triangle F$ denotes the symmetric difference of two sets
$E$ and $F$.

Early results of this kind include those of
Bonnesen~\cite{Bonnesen1924} for planar convex domains and
Hall~\cite{Hall1992} for quantitative isoperimetric inequalities in
$\rr^n$ with non-sharp exponent, obtained by symmetrisation.
Fuglede~\cite{Fuglede1989} proved a sharp quantitative estimate for
\emph{nearly spherical sets}, that is, sets having the same volume as
the unit ball and whose boundaries are normal graphs of functions $u$
over the unit sphere $\mathbb{S}^{n-1}$, with small $C^1$-norm.
Fuglede's estimate is obtained by a second-variation argument that
combines the spectral gap of the Jacobi operator of the sphere with a
Poincar\'e-type inequality for the mean-zero and barycentre-free
components of $u$. This shows that the perimeter deficit of the
perturbed set is bounded below by a positive constant times the
squared $L^2$-norm of $u$.

More recently, Fusco, Maggi and
Pratelli~\cite{FuscoMaggiPratelli2008} and Figalli, Maggi and
Pratelli~\cite{FigalliMaggiPratelli2010} proved, by two different
methods, the quantitative isoperimetric inequality with the sharp
exponent $q=2$ for arbitrary sets of finite perimeter. The approach
in~\cite{FuscoMaggiPratelli2008} combines a symmetrisation procedure
with Fuglede's result for nearly spherical sets. The general problem is
reduced to this case by replacing $E$ with a set having suitable
properties. By contrast, the argument
in~\cite{FigalliMaggiPratelli2010} relies on optimal mass transport
between $E$ and $B$.

A further approach was developed by Cicalese and
Leonardi~\cite{CicaleseLeonardi2012}. It is based on a selection
principle that reduces the general case to the nearly spherical case
covered by Fuglede's theorem, using the regularity theory of
quasiminimisers of the perimeter; see~\cite{CicaleseLeonardi2013}.

The quantitative isoperimetric inequality has also been extended to
the Gaussian setting by Cianchi, Fusco, Maggi and
Pratelli~\cite{CianchiFuscoMaggiPratelli2011}, and to isoperimetric
problems with homogeneous weights by Cinti, Glaudo, Pratelli, Ros-Oton
and Serra~\cite{CintiGlaudoPratelliRosOtonSerra2022}.

In hyperbolic $n$-space, Bögelein, Duzaar and
Scheven~\cite{BogeleinDuzaarScheven2015} established a sharp
quantitative isoperimetric inequality by combining a suitable
symmetrisation procedure with a Fuglede-type analysis, adapted to the
hyperbolic metric, of nearly spherical perturbations of geodesic balls.
Bögelein, Duzaar and Fusco~\cite{BogeleinDuzaarFusco2017} subsequently
obtained the analogous sharp result on the round sphere
$\mathbb{S}^n$. Again, their argument uses symmetrisation to reduce the
problem to a Fuglede-type second-variation estimate for nearly
spherical caps.

 In the broader class of rank-one symmetric spaces of non-compact type,
Silini~\cite{Silini2025} recently proved a Fuglede-type estimate for
geodesic spheres by exploiting the explicit spectral gap for the
Laplacian on such spheres computed by Bettiol, Lauret and Piccione~\cite{BettiolLauretPiccione2022}.

The problem has also been studied in the general setting
of Riemannian manifolds. Carron~\cite{Carron1996} proved a Fuglede-type
estimate for geometrically stable constant-mean-curvature sets in
complete Riemannian manifolds, with a constant $C$ that may depend on
the volume of the isoperimetric candidate. Chodosh, Engelstein and
Spolaor~\cite{ChodoshEngelsteinSpolaor2023} later showed that a
Fuglede-type estimate need not hold on an arbitrary closed Riemannian
manifold when the kernel of the Jacobi operator is not generated by
ambient isometries. Nevertheless, they proved that a suitably modified
estimate holds for every real-analytic metric, without requiring any
a priori knowledge of the shape of the minimisers.

In this paper, we consider the sub-Riemannian Heisenberg group
$\hh^1$. Our aim is to prove that the Pansu sphere
$\So\subset\hh^1$, the conjectured isoperimetric set, minimises the
sub-Riemannian perimeter in a quantitative sense among small
volume-preserving normal perturbations. This result is the
sub-Riemannian counterpart of Fuglede's theorem for the standard sphere
in $\rr^n$. A non-sharp quantitative isoperimetric inequality in
$\hh^n$ for sets contained in a vertical cylinder was obtained by
Franceschi, Leonardi and
Monti~\cite{FranceschiLeonardiMonti2015}, using subcalibrations.

The Heisenberg group $\hh^1$ is $\rr^3$ endowed with the Lie group
structure whose Lie algebra $\mathfrak h_1$ is spanned by the
left-invariant vector fields
\[
X=\frac{\partial}{\partial x}
  +\tau y\,\frac{\partial}{\partial t},
\qquad
Y=\frac{\partial}{\partial y}
  -\tau x\,\frac{\partial}{\partial t},
\qquad
T=\frac{\partial}{\partial t},
\]
the only non-zero commutator being $
[X,Y]=-2\tau T$. 
Here, $\tau>0$ is a parameter that will be regarded as fixed throughout
the paper.

The bundle spanned by $X$ and $Y$ is called the \emph{horizontal
bundle}, and the Haar measure on $\hh^1$ is the Lebesgue measure
$\mathscr L^3$. If $S=\partial B$ is the boundary of a bounded open set
$B\subset\hh^1$, its horizontal, or sub-Riemannian, surface measure is
defined variationally by
\[
\Ao(S)=\sup\Big\{\int_B(X\varphi_1+Y\varphi_2)\,dxdydt\ :\ \varphi_1,\varphi_2\in C_c^\infty(\hh^1),\ \varphi_1^2+\varphi_2^2\leqslant 1\Big\},
\]
This construction localises to a Borel measure supported on $S$, whose
total mass is $\Ao(S)$.

In the sub-Riemannian Heisenberg group $\hh^1$,
Pansu~\cite{Pansu1982,Pansu1983} conjectured that, up to left
translations, dilations and modifications on null sets, the
isoperimetric regions are copies of the topological ball $B_0$ whose
boundary is the \emph{Pansu sphere}
\begin{equation}\label{Pansu}
S_0
 =
\left\{
(z,t)\in\hh^1:
|z|\leq1,\quad
|t|
 =
\frac{\tau}{2}
\left(
|z|\sqrt{1-|z|^2}+\arccos|z|
\right)
\right\}.
\end{equation}
The Pansu sphere can be obtained by rotating a
Carnot--Carath\'eodory geodesic about the vertical axis.

The existence of isoperimetric regions in $\hh^1$, and more generally
in every Carnot group, was established by Leonardi and
Rigot~\cite{LeonardiRigot2003}; see also \cite{MR4939883}.  
%However, the
%regularity and characterisation of isoperimetric regions remain among
%the principal open problems in the area.
Several partial results on Pansu's conjecture have been obtained,
typically under an a priori restriction on the class of competitors.
Monti~\cite{Monti2008axial} proved that Pansu spheres in $\hh^n$ solve
the isoperimetric problem within the class of rotationally symmetric
sets. Ritoré and Rosales~\cite{MR2435652} proved an Alexandrov-type
theorem in $\hh^1$, thereby establishing Pansu's conjecture under a
$C^2$-regularity assumption on the boundary of an isoperimetric region.
Monti and Rickly~\cite{MontiRickly2009} settled the conjecture within
the class of convex sets. Finally, Ritoré~\cite{Ritore2012} used a
calibration argument to prove that Pansu spheres minimise the perimeter
among all sets of finite perimeter that contain a disc in the
hyperplane $\{t=0\}$ and are contained in the vertical cylinder over
that disc.
 
After this work was completed, Chen, Kao and Lien posted a series of papers on arXiv presenting an AI-based proof of the conjecture in every dimension $n\geq 1$.

The notion of a normal perturbation of the Pansu sphere $\So$ presents
a genuine difficulty: at the poles of $\So$ on the vertical axis $
Z=\{(z,t)\in\hh^1:z=0\},$
the horizontal distribution is tangent to $\So$, and hence the
horizontal normal is not defined.

We define $
t^-:\So\to\rr$ and $ t^+:\So\to\rr\cup\{+\infty\}$,
by $t^\pm(p)=t^\pm(r)$, where $p=(z,t)$ and $r=|z|$:
\begin{equation}
\label{titti_0}
\begin{split}
t^-(r)
&=
-\frac{r}{\sqrt{1-r^2}}
\arctan\left(\frac{\sqrt{1-r^2}}{r}\right),\\
t^+(r)
&=
\frac{r}{\sqrt{1-r^2}}
\left[
\arctan\left(\frac{r}{\sqrt{1-r^2}}\right)
+\frac{\pi}{2}
\right].
\end{split}
\end{equation}
These functions are understood by continuous extension at $r=0$ and
$r=1$; in particular,
\[
t^-(0)=t^+(0)=0,
\qquad
t^-(1)=-1,
\qquad
t^+(1)=+\infty.
\]
For $p\in\So$, let
$
s\mapsto\gamma_p(s)$ 
be the distinguished Carnot--Carath\'eodory geodesic selected by the
Riemannian approximation, with $\gamma_p(0)=p$, such that
$
\mathrm{dist}_0(\gamma_p(s),p)
 =
\mathrm{dist}_0(\gamma_p(s);\So)$, 
for all sufficiently small $s\in(-s_0,s_0)$. This identity continues
to hold precisely for
$
s\in(t^-(p),t^+(p));$   whenever $t^\pm(p)$ is finite,
$
\gamma_p(t^\pm(p))\in
Z$, see Section \ref{tris*}. 

This observation motivates our definition of a \emph{normal graph}
over $\So$. Let
$
u:\{p=(z,t)\in\So:|z|\geq\varrho\}\to\rr$, for some $\varrho \in (0,1)$,
be continuous and satisfy
$
u(p)\in\{t^-(p),t^+(p)\}$ when $|z|=\varrho$,
and
$
t^-(p)<u(p)<t^+(p)$ when $|z|>\varrho$. The normal graph of $u$ is
\[
\So^u
 =
\bigl\{
\gamma_p(u(p))\in\hh^1:
p=(z,t)\in\So,\ |z|\geq\varrho
\bigr\}.
\]
It is a well-defined topological sphere enclosing a domain $\Bo^u$.

Let $u$ be as above, with associated parameter
$\varrho\in(0,1)$. We say that $u$ is a \emph{piecewise-$C^2$ function} if
$ u\in C^2(\So\cap\{|z|\geq\varrho\})$.
For $p=(z,t)\in\So$ with $|z|\geq\varrho$, we define the weighted
second-order seminorm
\begin{equation}
\label{w2}
|u(p)|_{2,w}
 =
\frac{|u(p)|}{|z|}
+|\nabla u(p)|
+|z|\,|\nabla^2u(p)|,
\end{equation}
where $\nabla u$ and $\nabla^2u$ denote, respectively, the Euclidean
tangential gradient and Hessian of $u$ on $\So$.

%The corresponding
%first-order seminorm is
%\[
%|u(p)|_{1,w}
% =
%\frac{|u(p)|}{|z|}
%+|\nabla u(p)|.
%\]

\begin{theorem}
\label{F0}
Given   $\gamma\in(0,2]$ and $L>0$, there exist two constants $\delta=\delta(\gamma,L) >0$ and $C=C(\gamma,L)>0$ 
with the following property.
Let  $u:S_0\rightarrow\rr$ be  a continuous and  piecewise-$C^2$ function
satisfying,  for  some $\rho\geq0$$:$
\begin{enumerate}
    \item[(1)]  $u(p)\in\{t^-(p), t^+(p)\} $ for   $p=(z,t)\in S_0$ with $|z|\leq \varrho;$ 
    \item[(2)]  $t^-(p)<u(p)<t^+(p)$  for   $p=(z,t)\in S_0$ with $|z|> \varrho;$     
     \item[(3)]  $\rho\leq \delta$ and  $ | u(p)|_{2,w} \leq \delta$ for   $p=(z,t)\in S_0$ with $|z|> \varrho;$   
    \item[(4)] $\displaystyle u(p)^2\leq  L |z|^\gamma   \int_{\So} u^ 2 d\Ao$, for   $p=(z,t) \in \So;$
    \item[(5)]     $ \displaystyle  |\nabla  u(p)|^2\leq L  \int_{S_0}\|\nablaSo u\| ^ 2 d\Ao$,   for   $p=(z,t)\in S_0$ with $|z|> \varrho;$
    \item[(6)]     $\mathscr{L}^3(B_0^u)=\mathscr{L}^3(B_0).$       
     \end{enumerate}
 Then the following inequality holds:
\begin{equation}
\label{QI}
\mathscr{A}_0(S_0^u)-\mathscr{A}_0(S_0)\geq C\big(\mathscr{L}^3( B_0^u\triangle B_0)\big)^2.
\end{equation}
\end{theorem}
 
 Assumptions (1) and (2) ensure that the normal graph $\So^u$ is well defined. Assumptions (4) and (5) control the singular sub-Riemannian geometry at the poles of $\So$. In (5), $|\nabla u|$ denotes the Euclidean length of the tangential gradient of $u$, while $\nablaSo u$ is the horizontal tangential gradient of $u$. We comment on the $C^2$-smallness assumption in (3) later.

 The proof of Theorem \ref{F0} relies on a Riemannian approximation argument. 
We renormalise the vertical vector field $T$ by means of a parameter $\varepsilon >0$, setting
\[
   \Te = \varepsilon  \frac{\partial}{\partial t},
\]
and consider on $\hh^1$ the left-invariant Riemannian metric $g_\varepsilon$ for which the basis $X,Y,\Te \in \mathfrak{h}_{1}$ is orthonormal. We denote by $\Ae$ and $\Ve$ the corresponding Riemannian surface and volume measures.

Let $f_\varepsilon :[0,1]\rightarrow\rr$ be the function
\begin{equation}
\label{PansuProfile}
 f_\varepsilon(r) =\int_r^{1} \frac{s\sqrt{\varepsilon ^2 + \tau^2 s^2}} {\sqrt{1-s^2}}ds,\quad r\in [0,1],
\end{equation} 
and define the set 
\begin{equation}\label{Tomter}
\Se = \big\{(z,t)\in\hh^1: |t|=f_\varepsilon(|z|)\big\},\quad \varepsilon\geq 0.
\end{equation}
$\Se$ is an embedded topological sphere in $\hh^1$ with constant mean curvature $H=2$ with respect to the Riemannian metric $\gep$.
It was first computed by Tomter in \cite{Tomter1993} and is conjectured to be an isoperimetric surface in the Riemannian manifold $(\hh^1,\gep)$, among sets of the same volume and up to left translations.
This conjecture also remains open. When $\varepsilon=0$, the sphere $\Se$ reduces to the Pansu sphere in \eqref{Pansu}.

 The quantitative isoperimetric inequality \eqref{QI} is derived from the following Riemannian version:
 \begin{equation}
\label{QIe}
\Ae(\Se^u)-\Ae(\Se) -  H \big (\Ve (\Be^u) -\Ve (\Be) \big) \geq  \frac{C}{\Ae(\Se)} (\Ve(B_\eps^u\triangle B_\eps))^2.
\end{equation}
The point is to establish such an inequality with a constant $C>0$ \emph{independent of $\eps>0$}.
 
 To this aim, there are three problems to tackle.
 The first problem concerns a uniform spectral gap for the first positive eigenvalue of the Jacobi operator of $\Se$: this eigenvalue should remain
 uniformly bounded away from zero for small $\eps>0$. Although we were unable to establish such a uniform spectral gap, we could nevertheless
 prove the uniform strict stability of $\Se$, starting from the stability of the Pansu sphere $\So$ and imposing condition (4).
 This is Theorem \ref{8.1} in Section \ref{QUA}.

 The second problem is the presence of singularities of order $1/\eps^2$ in the geometric terms
$\operatorname{Ric}_\varepsilon(\Nu)$ and $\|\hep^u\|^{2}$, which appear in the second variation of the area $\Ae$.
In the case of the sphere $\Se$, an explicit computation shows that the singular terms cancel in the sum
$\operatorname{Ric}_\varepsilon(\Ne)+\|h_\eps\|^{2}$; see Remark \ref{singsing}.
As observed by Verzellesi and the third author in \cite{MR5029815}, this cancellation is a more general phenomenon, and we establish it also for the graphs $\Se^u$; see Theorem \ref{CANC}.
To exploit this cancellation, however, we need to work with a geometric variant of the second variation of $\Ae$, involving second derivatives of the graph function $u$.
This delicate point is discussed at the end of Section \ref{EXP}.
The $C^2$-smallness assumption in condition (3), rather than the $C^1$-smallness required in the standard Fuglede theorem, is related to this issue.

  The third problem is the presence, in the second variation of the area, of limiting singularities of order $1/|z|^2$ at points $(z,t)\in\Se$ near the poles.
These arise from the singular sub-Riemannian structure of the limiting Pansu sphere $\So$, whose north and south poles are characteristic points.
Assumptions (4) and (5), which involve the parameters $\gamma\in (0,2]$ and  $L>0$, control these singularities.

In Theorem \ref{F0}, no assumption is imposed on the barycentre of $\Bo^u$, nor is $u$ required to belong to the kernel of the Jacobi operator for the Riemannian problem with $\eps>0$.
The reason is that, whereas the normal graphs $\So^u$ are defined only for $|z|\geq\rho$ --- yet, moving the characteristic set of $\So$ --- the lifted normal graphs $\Se^u$ are defined for all $|z|\leq 1$ and, by assumption (1), are fixed at the poles.

 From a technical point of view, the core of the paper lies in Sections \ref{TEK1} and \ref{TEK2}, where we derive a collection of estimates for the normal $\Nu$ to the graphs $\Se^u$.
Based on explicit formulae for the geodesic flow starting from $\Se$, the computations involved are demanding and yet apparently unavoidable.
These estimates are used in Section \ref{UC}, where we prove the uniform continuity of the second variation of the isoperimetric functional
$\Fe=\Ae-H\Ve$ for $\eps>0$, with estimates independent of $\eps$.
Once Theorem \ref{unifcontSecondVar} has been established, deriving inequality \eqref{QIe} is not difficult.
The details of the proof of Theorem \ref{F0} are given in Section \ref{FUG}.

\medskip
\noindent\emph{AI disclaimer.} The mathematics in this article was developed without the use of artificial intelligence.
 
\medskip
\noindent\emph{Acknowledgements.}
 During the preparation of this article, J.~Pozuelo was supported by the research grant PID2023-151060NB-I00 funded by MCIN/AEI/10.13039/501100011033 and by ERDF/EU, and Ramón Areces Foundation grants XXXV. The authors thank S.~Verzellesi for several fruitful discussions.

 \section{Notation and Riemannian approximation}
 \label{tris*}
 
 In this section, we clarify several facts concerning the convergence of the Tomter spheres $\Se$ to the Pansu sphere $\So$ as $\eps\to 0$. We also study the limit of the normal flow starting from $\Se$, thereby obtaining the sub-Riemannian normal flow starting from $\So$. This leads to the appropriate definition of normal graphs over the Pansu sphere.

 \subsection{Notation}

We first fix some notation used throughout the paper.
For vector fields in $\rr^3=\hh^1$, we use the following coordinate representations.
Given functions $a,b,c:\rr^3\to\rr$ and $w=a+ib:\rr^3\to\cc$, at each fixed point $(z,t)\in\hh^1$ we make the identification
\begin{equation}
\label{TG}
\begin{split}
(w,c)_*
&=
a\frac{\partial}{\partial x}
+b\frac{\partial}{\partial y}
+c\frac{\partial}{\partial t}
\\
&=
aX+bY
+\frac{1}{\varepsilon}
\bigl(c+\tau\operatorname{Im}(w\bar z)\bigr)\Te
\\
&=
\Bigl[
w,\frac{1}{\varepsilon}
\bigl(c+\tau\operatorname{Im}(w\bar z)\bigr)
\Bigr]_*.
\end{split}
\end{equation}
The notation
\[
[w,c]_*=\operatorname{Re}(w)X+\operatorname{Im}(w)Y+c\Te
\]
is slightly misleading, since it does not explicitly display the base point $(z,t)\in\hh^1$.

For vector fields $U$ and $V$, we use the following notation for the scalar product and norm:
\[
\langle U,V\rangle=\gep(U,V)
\qquad\text{and}\qquad
\|U\|^2=\gep(U,U),
\]
thus suppressing the dependence on $\eps>0$.

On the Lie algebra $\mathfrak{h}_1$, we define the wedge product $\wedge$ with respect to the oriented orthonormal basis $X,Y,\Te$.
In the coordinates $[\ldots]_*$, it is given explicitly by
\begin{equation}
\label{WED}
[z,t]_*\wedge[z',t']_*
=
\bigl[i(tz'-t'z),\operatorname{Im}(z'\bar z)\bigr]_*.
\end{equation}
The length of the resulting vector with respect to the metric $g_\varepsilon$ is
\[
\bigl\|[z,t]_*\wedge[z',t']_*\bigr\|
=
\left(
|tz'-t'z|^2+\bigl(\operatorname{Im}(z'\bar z)\bigr)^2
\right)^{1/2}.
\]

For a vector field $U=[w,c]_*$, we use the notation
\begin{equation}
\label{barU}
\overline{U}
=
\overline{[w,c]_*}
=
[w,-c]_*
\end{equation}
for the vector field obtained by reversing the sign of the $\Te$-component.

We denote by $\nabla$ the Levi-Civita connection on $\hh^1$ associated with the metric $\gep$, and by $D$ the Tanaka-Webster connection.
In both cases, we suppress the dependence on $\eps>0$ from the notation.
The connection $D$ preserves the frame $X,Y,\Te$ and, for any   $V=[v,d]_*$, we have:
\begin{equation}
\label{VU}
D_UV=[Uv,Ud]_*=(U\mathop{\rm Re}v)X+(U\mathop{\rm Im}v)Y+(Ud)\Te.
\end{equation}
The connection $D$  is thus compatible with the metric $\gep$:
\begin{equation}
\label{VU2}
U\langle V,V'\rangle
=\langle D_UV,V'\rangle+\langle V,D_UV'\rangle.
\end{equation}
The Levi-Civita connection $\nabla$ and the Tanaka-Webster connection  $D$ are related by the identity
\begin{equation}
\label{LC}
\nabla_UV
=D_UV-\frac{\tau}{\eps}\overline U\wedge V
=D_UV-\sigma\overline U\wedge V.
\end{equation}
The parameter
\begin{equation}
\label{sigma}
\sigma=\frac{\tau}{\eps}
\end{equation}
will occur frequently throughout the paper.

\subsection{Convergence of surface measures}

We denote by $\Ae$ the Riemannian surface measure on $\hh^1$ induced by the metric $g_\varepsilon$.  
If $I,J\subset\rr$ are intervals and $F:I\times J\to\hh^1$ is a $C^1$ parametrisation
of a surface
\[
S=\{F(s,\theta)\in\hh^1:(s,\theta)\in I\times J\},
\]
then its measure is given by the area formula
\[
\Ae(S)=\int_{I\times J}
 \|\partial_sF\wedge\partial_\theta F\|\,ds\,d\theta,
\]
where the wedge product $\wedge$ is defined in \eqref{WED}.
For each fixed surface $S$, we have
\[
\varepsilon\Ae(S)\to\Ao(S)
\qquad\text{as }\varepsilon\to0;
\]
see Lemma \ref{Lex1} below for a special case.

We now  introduce a useful parametrisation of the spheres $\Se$.
Let $\kappa:[0,\pi]\times[0,2\pi]\to\cc$ be the map
\[
\kappa(s,\theta)
 =\frac12\ex^{i\theta}\bigl(\ex^{2is}-1\bigr)
 =i\sin(s)\ex^{i(\theta+s)}.
\]
When $\theta=0$, the curve $s\mapsto\kappa(s,0)$ parametrises, in the
anticlockwise direction, the circle with centre $-1/2\in\cc$ and radius
$1/2$, starting at $0$ when $s=0$. Notice that
$
|\kappa(s,\theta)|=\sin s$.

For each $\varepsilon\geq0$, we define the map
$
F_\varepsilon:[0,\pi]\times[0,2\pi]\to\hh^1=\cc\times\rr
$
by
\begin{equation}
\label{HR}
F_\varepsilon(s,\theta)
 =\bigl(\kappa(s,\theta),
 \mathop{\rm sgn}(s-\pi/2)f_\varepsilon(\sin s)\bigr),
\end{equation}
where $f_\varepsilon:[0,1]\to\rr$ is the profile function of $\Se$
introduced in \eqref{PansuProfile}. For every $\varepsilon\geq0$, we have
$
F_\varepsilon\bigl([0,\pi]\times[0,2\pi]\bigr)=\Se.
$

When $\varepsilon=0$, Hurtado and Rosales \cite{HurtadoRosales2015}
proved that the following area formula holds for suitable functions
$u:\So\to\rr$:
\begin{equation}
\label{area_0}
\int_{\So}u\,d\Ao
 =\int_0^\pi\int_0^{2\pi}
 u\bigl(F_0(s,\theta)\bigr)\sin^2(s)\,d\theta\,ds.
\end{equation}

\noindent
When $\varepsilon>0$, the following area formula holds for every
sufficiently regular function $u:\Se\to\rr$:
\begin{equation}
\label{area_e}
\int_{\Se}u\,d\Ae
 =\int_0^\pi\int_0^{2\pi}
 u\bigl(F_\varepsilon(s,\theta)\bigr)
 \|\partial_sF_\varepsilon\wedge\partial_\theta F_\varepsilon\|
 \,d\theta\,ds.
\end{equation}

In the following lemma, we compute the area element
$
\|\partial_sF_\varepsilon\wedge\partial_\theta F_\varepsilon\|
$
appearing in \eqref{area_e}.

\begin{lemma}
\label{Lex1}
For every $\varepsilon>0$, $s\in[0,\pi]$, and
$\theta\in[0,2\pi]$, we have
\begin{equation}
\label{A_e}
\|\partial_sF_\varepsilon\wedge\partial_\theta F_\varepsilon\|
 =\frac{1}{\varepsilon}
 \left(
 \sin^4(s)\bigl|\ex^{is}-i\varepsilon k_\varepsilon(s)\bigr|^2
 +\varepsilon^2\sin^2(s)\cos^2(s)
 \right)^{1/2},
\end{equation}
where $0\leq k_\varepsilon(s)\leq1$. In particular, for every $\delta \in (0,\pi/2)$ 
\begin{equation}
\label{UNIF}
\lim_{\varepsilon\to0}
\sup_{\substack{s\in[\delta,\pi-\delta]\\ \theta\in[0,2\pi]}}
\left|
 \varepsilon
 \|\partial_sF_\varepsilon\wedge\partial_\theta F_\varepsilon\|
 -\sin^2(s)
\right|=0.
\end{equation} 
\end{lemma}

\begin{proof}
We begin by computing the tangent vectors $\partial_sF_\varepsilon$ and
$\partial_\theta F_\varepsilon$ to $\Se$. In the notation introduced in
\eqref{TG}, they are
\[
\partial_sF_\varepsilon
 =\Bigl(i\ex^{i(\theta+2s)},-|\cos s|f_\varepsilon'(\sin s)\Bigr)_*
\quad\text{and}\quad
\partial_\theta F_\varepsilon
 =\Bigl(-\sin(s)\ex^{i(\theta+s)},0\Bigr)_*.
\]
To express these vectors in the coordinates $[\ldots]_*,$ we introduce the
auxiliary function $h_\varepsilon:[0,1]\to\rr$ defined, in view of
\eqref{PansuProfile}, by
\begin{equation}
\label{ge}
h_\varepsilon(r)
 =\frac{f_\varepsilon(r)-f_0(r)}{\varepsilon}
 =\int_r^1\frac{t}{\sqrt{1-t^2}}\,
 \frac{\varepsilon}
 {\sqrt{\varepsilon^2+\tau^2t^2}+\tau t}\,dt.
\end{equation}
In particular,
\begin{equation}
\label{g_e}
-|\cos s|h_\varepsilon'(\sin s)
 =\sin(s)k_\varepsilon(s),
\qquad
k_\varepsilon(s)
 =\frac{\varepsilon}
 {\sqrt{\varepsilon^2+\tau^2\sin^2s}+|\tau|\sin s}.
\end{equation}
This defines the function $k_\varepsilon$ appearing in \eqref{A_e}.
Notice that $0\leq k_\varepsilon(s)\leq1$.

Using the identity
$f_\varepsilon'=f_0'+\varepsilon h_\varepsilon'$ and formula \eqref{TG},
we see that only the term involving $h_\varepsilon'$ contributes to the
$\Te$-component of $\partial_sF_\varepsilon$, since the vector
$\partial_sF_0$ is horizontal. It then follows from \eqref{g_e} that
\begin{equation}
\label{fefe}
\partial_sF_\varepsilon
 =\Bigl[i\ex^{i(\theta+2s)},\sin(s)k_\varepsilon(s)\Bigr]_*.
\end{equation}
The computation of $\partial_\theta F_\varepsilon$ is simpler. Formula
\eqref{TG} gives
\[
\partial_\theta F_\varepsilon
 =\left[-\sin(s)\ex^{i(\theta+s)},
 \frac{1}{\varepsilon}\sin^2s\right]_*.
\]

By \eqref{WED}, their wedge product is
\[
\partial_sF_\varepsilon\wedge\partial_\theta F_\varepsilon
 =\left[
 \frac{1}{\varepsilon}\sin^2(s)\ex^{i(\theta+s)}
 \bigl(\ex^{is}-i\varepsilon k_\varepsilon(s)\bigr),
 \sin(s)\cos(s)
 \right]_*.
\]
Formula \eqref{A_e} follows. Finally, since
$0\leq k_\varepsilon(s)\leq1$, the uniform convergence in \eqref{UNIF}
follows directly from \eqref{A_e}.
\end{proof}

\begin{remark}
\label{tris}
It follows from \eqref{A_e} that, for every $0<\varrho\leq1$,
\begin{equation}
\label{poker-bis}
\eps\Ae\bigl(\Se\cap\{|z|\leq\varrho\}\bigr)
\leq
4\pi\varrho^2\sqrt{\eps^2+\varrho^2}.
\end{equation}
On the other hand, there exists an absolute constant $c_0>0$ such
that for $0<\eps<1/2$
\begin{equation}
\label{poker-ter}
c_0\varrho^3
\leq
\Ao\bigl(\So\cap\{|z|\leq\varrho\}\bigr)
\leq 2
\eps\Ae\bigl(\Se\cap\{|z|\leq\varrho\}\bigr).
\end{equation}

Now fix $\varrho\in(0,1)$. It follows from \eqref{UNIF} that there
exists $\eps_0>0$ such that, for every $0<\eps<\eps_0$,
\begin{equation}
\label{poker}
\frac{1}{2}\eps
\Ae\bigl(\Se\cap\{|z|>\varrho\}\bigr)
\leq
\Ao\bigl(\So\cap\{|z|>\varrho\}\bigr)
\leq
2\eps
\Ae\bigl(\Se\cap\{|z|>\varrho\}\bigr).
\end{equation}
These estimates will be used in Section \ref{UC}.
\end{remark}

\subsection{Normal flow and normal graphs}

Let $\Ne$ be the outer unit normal to $\Se$. A computation based on
\eqref{PansuProfile} gives, in the case of the northern hemisphere,
\begin{equation}
\label{nux}
\Ne(z,t)
 =\bigl[z(1-\tau p_\varepsilon i),\varepsilon p_\varepsilon\bigr]_*.
\end{equation}
In the case of the southern hemisphere, we instead have $
\Ne(z,t)
 =\bigl[z(1+\tau p_\varepsilon i),-\varepsilon p_\varepsilon\bigr]_*$.
 Henceforth, all computations are
carried out in the northern hemisphere. We extend $\Ne$ to a vector
field in a neighbourhood of $\Se$, still denoted by $\Ne$, by solving
the geodesic equation
\[
\nabla_{\Ne}\Ne=0.
\]

Let $\Psi_\varepsilon:\Se\times\rr\to\hh^1$ be the flow of $\Ne$ such
that
\[
\Psi_\varepsilon\bigl((z,t),0\bigr)=(z,t)
\qquad\text{for every }(z,t)\in\Se.
\]
The flow is defined for every $s\in\rr$ because the geodesics are
complete. Moreover, there exists $s_0>0$ such that
$
\Psi_\varepsilon:\Se\times(-s_0,s_0)\to\hh^1
$
is a diffeomorphism onto its image.

The flow can be computed explicitly. We omit the computation here.
Writing
$\Psi_\varepsilon
 =(\Psi_{\varepsilon,1},\Psi_{\varepsilon,2})\in\cc\times\rr,
$
we obtain
\begin{equation}
\label{Psie1}
\Psi_{\varepsilon,1}\bigl((z,t),s\bigr)
 =z\Bigl(1+s(\tau p_\varepsilon+i)\fa(s_\varepsilon)\Bigr),
\end{equation}
where, with $r=|z|$, we set
\begin{equation}
\label{LUX}
p_\varepsilon
 =\left(\frac{1-r^2}{\varepsilon^2+\tau^2r^2}\right)^{1/2},
\qquad
\fa(s_\varepsilon)
 =\frac{\ex^{-is_\varepsilon}-1}{s_\varepsilon},
\qquad
s_\varepsilon=2\tau p_\varepsilon s.
\end{equation}
The second component is
\begin{equation}
\label{Psie2}
\begin{split}
\Psi_{\varepsilon,2}\bigl((z,t),s\bigr)
= t+s\varepsilon^2p_\varepsilon
 +\tau s^2(1-\varepsilon^2p_\varepsilon^2)\fb(s_\varepsilon)
 +\tau sr^2\Bigl(
 \fc(s_\varepsilon)+\tau p_\varepsilon\fd(s_\varepsilon)
 \Bigr),
\end{split}
\end{equation}
where
\begin{equation}
\label{LAX}
\fb(s_\varepsilon)
 =\frac{s_\varepsilon-\sin(s_\varepsilon)}{s_\varepsilon^2},
\qquad
\fc(s_\varepsilon)
 =\frac{1-\cos(s_\varepsilon)}{s_\varepsilon},
\qquad
\fd(s_\varepsilon)
 =\frac{\sin(s_\varepsilon)}{s_\varepsilon}.
\end{equation}
After defining their values at the origin by continuity, the functions
$\fa,\fb,\fc$, and $\fd$ are analytic.

Using the normal flow $\Psi_\varepsilon$, we can define normal graphs over
$\Se$. For every continuous function $u:\Se\to\rr$ satisfying
$\|u\|_\infty<s_0$, we define the graph flow
$\Psi_\varepsilon^u:\Se\times[0,1]\to\hh^1$ by setting
\begin{equation}
\label{pss}
\Psi_\varepsilon^u(p,s)
 =\Psi_\varepsilon\bigl(p,su(p)\bigr).
\end{equation}
We call the set
\[
\Se^u
 =\bigl\{\Psi_\varepsilon^u(p,1)\in\hh^1:p\in\Se\bigr\}
\]
the normal graph of $u$ over $\Se$. We also denote by
$\Be^u\subset\hh^1$ the bounded domain such that
$\partial\Be^u=\Se^u$.

The flow $\Psi_\varepsilon$ has a well-behaved limit $\Psi_0$ as
$\varepsilon\to0$. Writing its components as
$
\Psi_0=(\Psi_{0,1},\Psi_{0,2})\in\cc\times\rr,
$
we obtain the following formulae: 
\begin{equation}
\begin{split}
\Psi_{0,1}\bigl((z,t),s\bigr)
 &=
 z\frac{\tau p_0-i}{2\tau p_0}
 \left(\ex^{i(2\theta_0-s_0)}+1\right),\\
\Psi_{0,2}\bigl((z,t),s\bigr)
 &=
 t+\tau s^2\frac{s_0-\sin(s_0)}{s_0^2}
 +\tau sr^2
 \frac{1-\cos(s_0)+\tau p_0\sin(s_0)}{s_0},
\end{split}
\end{equation}
where $(z,t)\in\So$, $r=|z|\in(0,1]$, and
\[
p_0=\frac{\sqrt{1-r^2}}{\tau r},
\qquad
\ex^{i2\theta_0}
 =\frac{\tau p_0+i}{\tau p_0-i},
\qquad\text{and}\qquad
s_0=2\tau p_0s.
\]

For every $(z,t)\in\So$ with $|z|\neq0$, the
Carnot--Carath\'eodory geodesic passing through $(z,t)$ in the direction
of the horizontal normal to $\So$ is not unique, unlike in the
Riemannian case. The limiting procedure selects the unique
distance-minimising geodesic from $\So$. More precisely, for every
$p=(z,t)\in\So$ with $z\neq0$, there exists $\delta>0$ such that, for
every $s\in(-\delta,\delta)$,
\[
|s|
 =\mathrm{dist}_0\bigl(\Psi_0(p,s),p\bigr)
 =\mathrm{dist}_0\bigl(\Psi_0(p,s);\So\bigr),
\]
where $\mathrm{dist}_0$ denotes the Carnot--Carath\'eodory distance on
$\hh^1$. The value of $\delta$ tends to zero as $p$ approaches either
pole.

When the flow $\Psi_0$ reaches the cut locus of the distance function
$\mathrm{dist}_0(\,\cdot\,;\So)$, namely the vertical axis $Z$, it
ceases to be injective. We now     determine the set on which $\Psi_0$ is a
diffeomorphism.

For $p=(z,t)\in\So\setminus Z$, with $r=|z|$, let
$t^-(p)=t^-(r)$ and $t^+(p)=t^+(r)$ be, respectively, the first negative
and first positive solutions $s$ of the equation
$
\Psi_{0,1}(p,s)=0.
$
Solving
$
\ex^{i(2\theta_0-s_0)}+1=0,
$
we obtain the formulae in \eqref{titti_0}. Notice that
\[
t^-(0)=t^+(0)=0,
\qquad
t^-(1)=-1,
\qquad
t^+(1)=+\infty.
\]
The domain of injectivity of $\Psi_0$ is therefore
\[
\Sigma
 =\bigl\{(p,s)\in\So\times\rr:
 t^-(p)<s<t^+(p)\bigr\}.
\]

\begin{theorem}
\label{INJ}
The map
$
\Psi_0:\Sigma\to\hh^1\setminus Z
$
is a diffeomorphism.
\end{theorem}

The proof of this theorem is given in Appendix 1.

Normal graphs over $\So$ can be defined consistently for the following
class of functions. Let $u:\So\to\rr$ be a continuous function
satisfying the following properties for some $0\leq\varrho<1$:
\begin{itemize}
\item[(i)] $u(p)=t^+(p)$ or $u(p)=t^-(p)$ for every
$p=(z,t)\in\So$ with $|z|\leq\varrho$;

\item[(ii)] $t^-(p)<u(p)<t^+(p)$ for every
$p=(z,t)\in\So$ with $|z|>\varrho$.
\end{itemize}
The \emph{sub-Riemannian normal graph} of $u$ is the surface
\[
\Su
 =\bigl\{\Psi_0\bigl(p,u(p)\bigr)\in\hh^1:
 p=(z,t)\in\So,\ |z|\geq\varrho\bigr\}.
\]
By (ii) and Theorem \ref{INJ}, the graph map
$\Psi_0^u:
\bigl(\So\cap\{|z|\geq\varrho\}\bigr)\times[0,1]\to\hh^1
$
is well defined by
\[
\Psi_0^u(p,s)=\Psi_0\bigl(p,su(p)\bigr),
\qquad
p\in\So\cap\{|z|\geq\varrho\},
\quad s\in[0,1].
\]
Although $\So\cap\{|z|\geq\varrho\}$ is not a topological sphere when
$\varrho>0$, the surface $\Su$ is a topological sphere, since
$
\Psi_0\bigl(p,u(p)\bigr)\in Z$ if 
 $|z|=\varrho$.
We denote by $\Bu\subset\hh^1$ the bounded domain whose boundary is
$
\partial\Bu=\Su$.

By (i), $u$ coincides with either $t^+$ or $t^-$ on each connected
component of $\So\cap\{|z|\leq\varrho\}$, and
$\Psi_0^u(\,\cdot\,,1)$ maps each such component into the vertical axis:
\[
\Psi_0\bigl(p,u(p)\bigr)
 \in Z:=\bigl\{(0,t)\in\hh^1:t\in\rr\bigr\},
\qquad |z|\leq\varrho.
\]
The extension of $u$ prescribed in (i) on
$\So\cap\{|z|<\varrho\}$ is not strictly necessary for the definition
of the graph $\Su$. However, we shall need to lift $u$ to $\Se$ in
order to define normal graphs over $\Se$ without holes. Condition (i)
provides a canonical extension for this purpose.
A function $u:\So\to\rr$ satisfying (i) and (ii) can be lifted to a
function on $\Se$, still denoted by $u$, by setting
$
u(z,t)=u(z,t')
$
whenever $(z,t)\in\Se$ and $(z,t')\in\So$ have the same horizontal
coordinate and satisfy $tt'\geq0$. The graphs of $u$ over $\So$ and
$\Se$ are illustrated in the following figure.

\begin{figure}[h]
\centering
\begin{minipage}{0.42\textwidth}
    \centering
    \resizebox{\linewidth}{!}{%
        \begin{tikzpicture}
\begin{axis}[
    axis lines=middle,
    xlabel={$r$}, ylabel={$t$},
    trig format=rad, 
    xmin=-1.2, xmax=1.2,
    ymin=0, ymax=1.9,
    xtick=\empty,
    ytick=\empty,
    width=12cm, height=8cm
]

% --- Funzione profilo ---
  \addplot[thick, samples=500, domain=-1:1]
    {(acos(abs(x)) + abs(x)*sqrt(1 - x^2))}
    node[pos=0.3, above left] {$S_0$};

% --- Curva definita da punti, smooth ---
  \addplot[fill, fill opacity=0.1,
    thick, red, smooth, tension=0.9
  ] coordinates {
    (-1.08,0)
      (-1.08,0.7)
      (-0.5,1.2)
      (0,pi/2+0.2)
      (0.5,1.2)
      (1.08,0.5)
      (0.91,0)
  }
  node[opacity=1, pos=0.2, above left] {$S_{0}^u$};

%--- punti rho ---
 \fill[black] (0.2,0) circle (1pt) node[above right] {$\varrho$};
  \fill[black] (-0.2,0) circle (1pt) node[above left] {$-\varrho$};
%--- segmenti ---
\draw[ dotted] (0.2,0) -- (0.2,pi/2) node[opacity=1, red, pos=0.5, right] {$B_{0}^u$}
node[opacity=1, black, pos=0.96, left=7.5mm ] {\scriptsize{$u=t^+$}};
\draw[ dotted] (-0.2,0) -- (-0.2,pi/2);
\draw[red, thick] (0,1.57) -- (0,1.77);

%--- geodetiche
 \addplot[dashed, samples=500, domain=-0.2:0.2]
    {sqrt(0.2^2-x^2)+pi/2-0.001};
\fill[black] (0,pi/2+0.2) circle (1pt);
\addplot[dashed, samples=500, domain=-0.1:0.1]
    {sqrt(0.1^2-x^2)+pi/2-0.001};
\fill[black] (0,pi/2+0.2) circle (1pt);
\addplot[dashed, samples=500, domain=-0.15:0.15]
    {sqrt(0.15^2-x^2)+pi/2-0.001};
\fill[black] (0,pi/2+0.2) circle (1pt);
\addplot[dashed, samples=500, domain=-0.05:0.05]
    {sqrt(0.05^2-x^2)+pi/2-0.001};
\fill[black] (0,pi/2+0.2) circle (1pt);

\end{axis}
\end{tikzpicture}%
    }
\end{minipage}
\hfill
\begin{minipage}{0.1\textwidth}
    \centering
    $\longrightarrow$
\end{minipage}
\hfill
\begin{minipage}{0.42\textwidth}
    \centering
    \resizebox{\linewidth}{!}{%
        \begin{tikzpicture}
\begin{axis}[
    axis lines=middle,
    xlabel={$r$}, ylabel={$t$},
    trig format=rad, 
    xmin=-1.2, xmax=1.2,
    ymin=0, ymax=2,
    xtick=\empty,
    ytick=\empty,
    width=12cm, height=8cm
]

% --- Funzione profilo ---
  \addplot[thick, samples=500, domain=-1:1]
    {((0.5^2+4)*atan(2*sqrt((1-x^2)/(0.5^2+4*x^2))) + 2*sqrt((0.5^2+4*x^2)*(1 - x^2))))/4}
    node[pos=0.3, above left] {$S_{\varepsilon}$};

% --- Curva definita da punti, smooth ---
  \addplot[    thick, red, smooth, tension=0.9
  ] coordinates {
     (-1.08,0)
      (-1.08,0.7)
      (-0.6,1.2)
      (-0.3,1.65)
      (-0.15,1.82)
      (-0.09,1.86)
  }
  node[opacity=1, pos=0.4, above left] {$S_{\varepsilon}^u$};

\addplot[
    thick, red, smooth, tension=0.9
  ] coordinates {
      (0.91,0)
      (1.08,0.5)
      (0.6,1.2)
      (0.3,1.65)
      (0.15,1.82)
      (0.09,1.86)
  };
%--- punti r_0 ---
 \fill[black] (0.2,0) circle (1pt) node[above right] {$\varrho$};
  \fill[black] (-0.2,0) circle (1pt) node[above left] {$-\varrho$};
%--- segmenti ---
\draw[ dotted] (0.2,0) -- (0.2,1.66)
node[opacity=1, red, pos=0.5, right] {$B_{\varepsilon}^u$}
node[opacity=1, black, pos=0.96, left=7.5mm ] {\scriptsize{$u=t^+$}};
\draw[ dotted] (-0.2,0) -- (-0.2,1.66);

\path[fill=red, fill opacity=0.1]
  plot[smooth, tension=0.9] coordinates {
    (-1.08,0)
    (-1.08,0.7)
    (-0.6,1.2)
    (-0.3,1.65)
    (-0.15,1.82)
    (-0.09,1.86)
  }
  plot[domain=0.8382852:0, variable=\t, samples=60]
    ({-0.2716889+0.2716889*cos(\t)}, {1.658+0.2716889*sin(\t)})
  plot[domain=3.1415927:2.3033074, variable=\t, samples=60]
    ({0.2716889+0.2716889*cos(\t)}, {1.658+0.2716889*sin(\t)})
  plot[smooth, tension=0.9] coordinates {
    (0.09,1.86)
    (0.15,1.82)
    (0.3,1.65)
    (0.6,1.2)
    (1.08,0.5)
    (0.91,0)
  }
  -- cycle;

% --- Riempimento (solo fill, nessun segmento disegnato) della zona
%     sotto la spezzata (-1.08,0)--(-0.09,1.86)--(0,1.658)--(0.09,1.86)--(0.91,0) ---
\fill[red, fill opacity=0.1]
  (-1.08,0) -- (-0.09,1.86) -- (0,1.658) -- (0.09,1.86) -- (0.91,0) -- cycle;

% --- Archi rossi visibili che formano la cuspide ---
\addplot[red, thick, domain=3.1415927:2.3033074, variable=\t, samples=60, forget plot]
    ({0.2716889+0.2716889*cos(\t)}, {1.658+0.2716889*sin(\t)});
\addplot[red, thick, domain=0.8382852:0, variable=\t, samples=60, forget plot]
    ({-0.2716889+0.2716889*cos(\t)}, {1.658+0.2716889*sin(\t)});

%--- geodetiche
 \addplot[ dashed, black, samples=500, domain=-0.2:-0.09]
    {1.1*sqrt(0.2^2-x^2)+1.66-0.001};
 \addplot[ dashed, black, samples=500, domain=-0.15:-0.05]
    {1.1*sqrt(0.15^2-x^2)+1.66-0.001};
\addplot[ dashed, black, samples=500, domain=-0.1:-0.03]
    {1.1*sqrt(0.1^2-x^2)+1.66-0.001};
\addplot[ dashed, black, samples=500, domain=-0.05:-0.01]
    {1.1*sqrt(0.05^2-x^2)+1.66-0.001};
\addplot[ dashed, black, samples=500, domain=0.09:0.2]
    {1.1*sqrt(0.2^2-x^2)+1.66-0.001};
\addplot[ dashed, black, samples=500, domain=0.05:0.15]
    {1.1*sqrt(0.15^2-x^2)+1.66-0.001};
\addplot[ dashed, black, samples=500, domain=0.03:0.1]
    {1.1*sqrt(0.1^2-x^2)+1.66-0.001};
\addplot[ dashed, black, samples=500, domain=0.01:0.05]
    {1.1*sqrt(0.05^2-x^2)+1.66-0.001};
\fill[black] (-0.09,1.86) circle (1pt);
\fill[black] (0.09,1.86) circle (1pt);

\end{axis}
\end{tikzpicture}%
    }
\end{minipage}
\caption{Passage from $\eps=0$ to  $\eps>0$.}
\end{figure}

\section{Uniform strict stability of $\Se$ for small $\varepsilon>0$}
\label{QUA}

In this section, we establish the strict stability of Tomter's spheres,
uniformly with respect to the small parameter $\varepsilon>0$.

For every sufficiently regular function $u:\Se\to\rr$, we define the
isoperimetric functional
\begin{equation}
\label{IF}
\Fe(u)=\Ae(\Se^u)-H\Ve(\Be^u),
\end{equation}
where $H=2$ is the constant mean curvature of $\Se$, independent of
$\varepsilon>0$, while $\Se^u$ is the normal graph of $u$ and $\Be^u$
is the domain bounded by $\Se^u$. The first variation of $\Fe$ at zero
vanishes:
\[
\begin{split}
\left.\frac{d}{ds}\Fe(su)\right|_{s=0}
&=
\left.\frac{d}{ds}\Ae(\Se^{su})\right|_{s=0}
-H\left.\frac{d}{ds}\Ve(\Be^{su})\right|_{s=0}\\
&=
\int_{\Se}uH\,d\Ae
-H\int_{\Se}u\,d\Ae
=0.
\end{split}
\]

For $u\in C^2(\Se)$, the second variation of $\Fe$ is given by
\[
\begin{split}
\left.\frac{d^2}{ds^2}\Fe(su)\right|_{s=0}
&=
\left.\frac{d^2}{ds^2}\Ae(\Se^{su})\right|_{s=0}
-H\left.\frac{d^2}{ds^2}\Ve(\Be^{su})\right|_{s=0}\\
&=
-\int_{\Se}
u\left\{
\Delta_\varepsilon u+
\left(\operatorname{Ric}_\varepsilon(\Ne)+\|\hep\|^2\right)u
\right\}\,d\Ae\\
&=
\int_{\Se}u\Le u\,d\Ae,
\end{split}
\]
where
$\Le:C^2(\Se)\to C(\Se)$,
\[
\Le u
 =-\Delta_\varepsilon u
 -\left(\operatorname{Ric}_\varepsilon(\Ne)+\|\hep\|^2\right)u,
\]
is the Jacobi operator of $\Se$. Here, $\Delta_\varepsilon$ is the
Laplace--Beltrami operator on $\Se$,
$\operatorname{Ric}_\varepsilon(\Ne)$ is the Ricci curvature in the
direction of the normal $\Ne$, and $\|\hep\|^2$ is the squared norm of
the second fundamental form of $\Se$.

After integration by parts, the second variation takes the form
\[
\left.\frac{d^2}{ds^2}\Fe(su)\right|_{s=0}
=
\int_{\Se}
\left\{
\|\nablaSe u\|^2
-\left(\operatorname{Ric}_\varepsilon(\Ne)+\|\hep\|^2\right)u^2
\right\}\,d\Ae,
\]
where $\nablaSe u$ is the tangential gradient of $u$ and
$\|\nablaSe u\|$ is its length with respect to the metric $\gep$.
This expression is well defined for $u\in\mathrm{Lip}(\Se)$. We refer
the reader to \cite{Ritore2023} for these facts. Further details are provided in Section \ref{EXP} below.

The goal of this section is to prove the following theorem. Given a
function $u:\So\to\rr$, we extend it to $\Se$ --- retaining the notation
$u$ --- by setting
\[
u(z,t')=u(z,t)
\]
whenever $(z,t)\in\So$ and $(z,t')\in\Se$ satisfy $tt'\geq0$.

\begin{theorem}
\label{8.1}
For every $\gamma>0$ and $L>0$, there exist constants
$\varepsilon_0=\varepsilon_0(\gamma ,L)>0$ and $m=m(\gamma ,L)>0$ with the following
property. If $u\in\mathrm{Lip}(\So)$ satisfies
\begin{equation}
\label{Integral_estimate}
u(p)^2
 \leq L|z|^\gamma\int_{\So}u^2\,d\Ao,
\qquad p=(z,t)\in\So,
\end{equation}
then, for every $\varepsilon\in(0,\varepsilon_0)$,
\begin{equation}
\label{USe}
\int_{\Se}
\left\{
\|\nablaSe u\|^2
-\left(
 \operatorname{Ric}_\varepsilon(\Ne)+\|\hep\|^2
 \right)u^2
\right\}\,d\Ae
\geq\frac{m}{8}\int_{\Se}u^2\,d\Ae.
\end{equation}
\end{theorem}

\begin{remark}
\label{singsing}
The geometric quantity
\[
Q_\varepsilon
 =\operatorname{Ric}_\varepsilon(\Ne)+\|\hep\|^2
\]
can be computed explicitly. By formula (3.4) for $\hep$ in \cite{FranceschiMontefalconeMonti2019} its squared norm at a point
$(z,t)\in\Se$ is
\begin{equation}
\label{ACCA}
\|\hep(z,t)\|^2
 =2\left(
 1+\frac{\se^4(1+\se^2)}{(1+\se^2r^2)^2}r^4
 \right),
\qquad r=|z|\leq1.
\end{equation}
Notice that, for $z\neq0$,
\[
\|\hep(z,t)\|^2\to\infty
\qquad\text{as }\varepsilon\to0
\quad\text{(equivalently, as $\sigma\to\infty$)}.
\]

For a vector field $U:\hh^1\to T\hh^1$, the Ricci tensor is given by, see \cite{MR5029815}: 
\begin{equation}
\label{ricciexpress}
\ric_\varepsilon(U)
 =-2\se^2\left(
 \|U\|^2-2\langle U,\Te\rangle^2
 \right).
\end{equation}
Substituting \eqref{nux} into this formula, we obtain
\begin{equation}
\label{RICCO}
\begin{split}
\operatorname{Ric}_\varepsilon\bigl(\Ne(z,t)\bigr)
 &=-2\se^2\bigl(1-2\varepsilon^2p_\varepsilon^2\bigr)\\
 &=2\se^2
 \frac{1-\se^2r^2-2r^2}{1+\se^2r^2}.
\end{split}
\end{equation}
Notice that, for $z\neq0$,
\[
\operatorname{Ric}_\varepsilon\bigl(\Ne(z,t)\bigr)\to-\infty
\qquad\text{as }\sigma\to\infty.
\]

Adding \eqref{ACCA} and \eqref{RICCO}, we finally obtain
\begin{equation}
\label{QUQQO}
Q_\varepsilon(z,t)
 =\frac{2(1+\se^2)}{(1+\se^2r^2)^2}.
\end{equation}
Thus, the singular terms cancel in \eqref{QUQQO}, and
$Q_\varepsilon$ concentrates at the poles as $\varepsilon\to0$.
\end{remark}

\begin{proposition}
For every $\varrho\in(0,1]$, we have
\begin{equation}
\label{tuono}
\lim_{\varepsilon\to0}
\sup_{\substack{(z,t)\in\Se\\ |z|\geq\varrho}}
Q_\varepsilon(z,t)=0.
\end{equation}
Moreover, for every $\varepsilon>0$,
\begin{equation}
\label{lampo}
\int_{\Se}Q_\varepsilon(z,t)\,d\Ae(z,t)=8\pi.
\end{equation}
\end{proposition}

\begin{proof}
Formula \eqref{tuono} follows   from \eqref{QUQQO}. The proof of
\eqref{lampo} is less immediate. By symmetry we have
\[
\int_{\Se}Q_\varepsilon(z,t)\,d\Ae(z,t)=2\int_{\Se^+}Q_\varepsilon(z,t)\,d\Ae(z,t),
\]
where $\Se^+=\Se\cap\{t\geq0\}$.
Parametrizing $S_\varepsilon^+$ by $\varphi:[0,1]\times[0,2\pi]\to S_\varepsilon^+$ as
$
\varphi(r,\theta)=\big(r\cos\theta,r\sin\theta,f_\varepsilon(r)\big)$,   at a point $\varphi(r,\theta)$ it holds
\begin{equation*}
\begin{split}
\|\partial_r \varphi \wedge \partial _\theta \varphi \|
%=\|(\cos\theta,\sin\theta, f'_\varepsilon(r))\wedge (-r\sin\theta,r\cos\theta,0)\|\\
&=\|[\cos\theta,\sin\theta,\varepsilon^{-1} f'_\varepsilon(r)]_* \wedge [-r\sin\theta,r\cos\theta,\sigma r^2]_*\|.
\end{split}
\end{equation*} 
From \eqref{PansuProfile} we compute  
\[
\frac{f'_\varepsilon(r)}{\varepsilon}=-r\sqrt{\frac{1+\sigma^2r^2}{1-r^2}},
\]
and using the formula  \eqref{WED}  we get
\[
\|\partial_r \varphi \wedge \partial _\theta \varphi \|^2=r^2\Big(\frac{r^2(1+\sigma^2r^2)}{1-r^2}+\sigma^2r^2+1\Big)=r^2 \frac{1+\sigma^2r^2}{1-r^2},
\]
and thus
\begin{align*}
\int_{\Se^+}Q_\varepsilon(z,t)\,d\Ae(z,t)&=\int_0^{2\pi}\!\!\int_0^1\frac{2(1+\sigma^2)}{(1+\sigma^2 r^2)^2}\sqrt{\frac{1+\sigma^2 r^2}{1-r^2}}\;r\,dr\,d\theta
\\&=4\pi(1+\sigma^2)\int_0^1\frac{r\,dr}{(1+\sigma^2 r^2)^{3/2}\sqrt{1-r^2}} .    
\end{align*}
Using the change of variable $w=\sqrt{1-r^2}$, the integral is
\begin{equation*}
\begin{split}
\int_0^1\frac{dw}{(1+\sigma^2-\sigma^2w^2)^{3/2}}=\left[\frac{w}{(1+\sigma^2)\sqrt{1+\sigma^2-\sigma^2w^2}}\right]_0^1=\frac1{1+\sigma^2},
\end{split}
\end{equation*}
so that we obtain the claim.
\end{proof}

The proof of \eqref{USe} relies on a Poincar\'e inequality on $\So$.
Consider the Hurtado--Rosales parametrisation of $\So$,
\[
(s,\theta)\mapsto F_0(s,\theta)\in\So,
\]
introduced in \eqref{HR} with $\varepsilon=0$. The vector field
\[
(s,\theta)\mapsto\partial_sF_0(s,\theta)\in T\So
\]
is horizontal and has unit length. Therefore, for every
$u\in\mathrm{Lip}(\So)$, we define the norm of its horizontal
tangential gradient by
\begin{equation}
\label{horiz}
\|\nablaSo u(F_0(s,\theta))\|
 =
\left|\frac{d}{ds}u(F_0(s,\theta))\right|.
\end{equation}

\begin{lemma}
\label{POI}
For every $\gamma>0$ and $L>0$, there exists a constant $m=m(\gamma,L)>0$ with the following
property. If $u\in\mathrm{Lip}(\So)$ satisfies
\eqref{Integral_estimate}, then
\begin{equation}
\label{PON}
m\int_{\So}u(p)^2\,d\Ao
 \leq
\int_{\So}\|\nablaSo u(p)\|^2\,d\Ao.
\end{equation}
\end{lemma}

\begin{proof}
We claim that the following infimum is strictly positive:
\[
m
 =\inf\left\{
 \frac{\displaystyle
 \int_{\So}\|\nablaSo u(p)\|^2\,d\Ao}
 {\displaystyle
 \int_{\So}u(p)^2\,d\Ao}
 :
 \ u\in\mathrm{Lip}(\So),\
 u\text{ satisfies \eqref{Integral_estimate}},\
 u\neq0
 \right\}.
\]
Arguing by contradiction, suppose that $m=0$, and let
$(u_n)_{n\in\nn}$ be a minimising sequence. Since condition
\eqref{Integral_estimate} is homogeneous, we may normalise the sequence
so that
\[
\int_{\So}u_n(p)^2\,d\Ao=1
\qquad\text{for every }n\in\nn,
\]
and
\[
\lim_{n\to\infty}
\int_{\So}\|\nablaSo u_n(p)\|^2\,d\Ao=0.
\]

Choose $\rho>0$ such that
\begin{equation}
\label{W1}
\rho<\Big(\frac{1}{4L\Ao(\So)} \Big)^{1/\gamma}.
\end{equation}
It follows from \eqref{Integral_estimate} that 
\begin{equation}
\label{Q1}
\begin{split}
\int_{\So\cap\{|z|\leq\rho\}}u_n(p)^2\,d\Ao
&\leq
L\rho^\gamma \Ao\bigl(\So\cap\{|z|\leq\rho\}\bigr)
\int_{\So}u_n(p)^2\,d\Ao\\
&\leq L\rho^\gamma \Ao(\So)
<\frac{1}{4}.
\end{split}
\end{equation}

To estimate the integral over $\So\cap\{|z|>\rho\}$, we use the
parametrisation
$
F_0:[0,\pi]\times[0,2\pi]\to\So.
$
By \eqref{area_0} and \eqref{horiz}, for every
$u\in\mathrm{Lip}(\So)$, setting
$
v(s,\theta)=u\bigl(F_0(s,\theta)\bigr),
$
we have
\begin{equation}
\label{W2}
\begin{split}
\int_{\So}u(p)^2\,d\Ao
&=
\int_0^{2\pi}\int_0^\pi
v(s,\theta)^2\sin^2(s)\,ds\,d\theta,\\
\int_{\So}\|\nablaSo u(p)\|^2\,d\Ao
&=
\int_0^{2\pi}\int_0^\pi
|\partial_s v(s,\theta)|^2\sin^2(s)\,ds\,d\theta.
\end{split}
\end{equation}

Fix $\delta\in(0,\pi/2)$ such that
\[
F_0(s,\theta)\in\So\cap\{|z|>\rho\}
\quad\Longleftrightarrow\quad
s\in(\delta,\pi-\delta);
\]
that is, $\delta=\arcsin\rho$. Setting
$
v_n(s,\theta)=u_n\bigl(F_0(s,\theta)\bigr),
$
we have, for $s\in(\delta,\pi-\delta)$,
\[
v_n(s,\theta)
 =
v_n(\delta,\theta)
+\int_\delta^s\partial_\sigma v_n(\sigma,\theta)\,d\sigma.
\]
Since $v_n(\delta,\theta)^2\leq L\rho^\gamma$, it follows that
\[
v_n(s,\theta)^2
 \leq
2L\rho^\gamma
+2C_\delta
\int_\delta^{\pi-\delta}
|\partial_\sigma v_n(\sigma,\theta)|^2
\sin^2(\sigma)\,d\sigma,
\]
where
\[
C_\delta
 =\int_\delta^{\pi-\delta}
 \frac{1}{\sin^2(\sigma)}\,d\sigma.
\]

Integrating this inequality and using \eqref{W1} and \eqref{W2}, we
obtain
\begin{equation}
\label{Q2}
\begin{split}
\int_{\So\cap\{|z|>\rho\}}u_n(p)^2\,d\Ao
&=
\int_\delta^{\pi-\delta}\int_0^{2\pi}
v_n(s,\theta)^2\sin^2(s)\,d\theta\,ds\\
&\leq
2L\rho^\gamma\Ao(\So)
+2\pi C_\delta
\int_0^{2\pi}\int_\delta^{\pi-\delta}
|\partial_\sigma v_n(\sigma,\theta)|^2
\sin^2(\sigma)\,d\sigma\,d\theta\\
&<
\frac{1}{2}
+2\pi C_\delta
\int_{\So}\|\nablaSo u_n(p)\|^2\,d\Ao.
\end{split}
\end{equation}
Choose $n\in\nn$ sufficiently large that
\[
2\pi C_\delta
\int_{\So}\|\nablaSo u_n(p)\|^2\,d\Ao
<\frac{1}{4}.
\]
Combining \eqref{Q1} and \eqref{Q2}, we obtain the contradiction
\[
\int_{\So}u_n(p)^2\,d\Ao<1.
\]
This completes the proof.
\end{proof}

\begin{lemma}
\label{L5}
There exists $\varepsilon_0>0$ such that, for every
$\varepsilon\in(0,\varepsilon_0)$ and every
$u\in\mathrm{Lip}(\So)$,
\begin{equation}
\label{energy}
\varepsilon\int_{\Se}\|\nablaSe u\|^2\,d\Ae
 \geq
\frac{1}{4}\int_{\So}\|\nablaSo u\|^2\,d\Ao.
\end{equation}
\end{lemma}

\begin{proof}
The length of the vector field
$\partial_sF_\varepsilon\in T\Se$ in \eqref{fefe} satisfies
\[
\begin{split}
\|\partial_sF_\varepsilon\|^2
&=1+\sin^2(s)k_\varepsilon(s)^2\\
&=
1+\left(
\frac{\varepsilon\sin s}
{\sqrt{\varepsilon^2+\tau^2\sin^2s}+|\tau|\sin s}
\right)^2
\leq2.
\end{split}
\]
Using the identity
$
u\bigl(F_\varepsilon(s,\theta)\bigr)
 =u\bigl(F_0(s,\theta)\bigr)
$
and \eqref{horiz}, we deduce that, for almost every $(s,\theta)$,
\begin{equation}
\label{M1}
\begin{split}
\|\nablaSe u(F_\varepsilon(s,\theta))\|^2
&\geq
\frac{1}{\|\partial_sF_\varepsilon\|^2}
\left|
\frac{d}{ds}u(F_\varepsilon(s,\theta))
\right|^2\\
&\geq
\frac{1}{2}
\left|
\frac{d}{ds}u(F_0(s,\theta))
\right|^2\\
&=
\frac{1}{2}\|\nablaSo u(F_0(s,\theta))\|^2.
\end{split}
\end{equation}
It follows from \eqref{A_e} and the bound
$0\leq k_\varepsilon(s)\leq1$ that there exists
$\varepsilon_0>0$ such that, for every
$\varepsilon\in(0,\varepsilon_0)$, $s\in[0,\pi]$, and
$\theta\in[0,2\pi]$,
\begin{equation}
\label{M2}
\varepsilon
\|\partial_sF_\varepsilon\wedge\partial_\theta F_\varepsilon\|
 \geq\frac{1}{2}\sin^2(s).
\end{equation}
Indeed,
$
\bigl|\ex^{is}-i\varepsilon k_\varepsilon(s)\bigr|
 \geq1-\varepsilon k_\varepsilon(s)
 \geq1-\varepsilon.
$
The area formulae \eqref{area_e} and \eqref{area_0}, together with
\eqref{M1} and \eqref{M2}, now give
\[
\begin{split}
\varepsilon\int_{\Se}\|\nablaSe u\|^2\,d\Ae
&=
\int_0^\pi\int_0^{2\pi}
\|\nablaSe u(F_\varepsilon(s,\theta))\|^2
\varepsilon
\|\partial_sF_\varepsilon\wedge\partial_\theta F_\varepsilon\|
\,d\theta\,ds\\
&\geq
\frac{1}{4}
\int_0^\pi\int_0^{2\pi}
\|\nablaSo u(F_0(s,\theta))\|^2
\sin^2(s)\,d\theta\,ds\\
&=
\frac{1}{4}\int_{\So}\|\nablaSo u\|^2\,d\Ao.
\end{split}
\]
\end{proof}

\begin{lemma}
\label{leox}
For every $L>0$ and $\eta>0$, there exists
$\varepsilon_0=\varepsilon_0(\eta,L)>0$ with the following property.
If $u\in\mathrm{Lip}(\So)$ satisfies \eqref{Integral_estimate}, then,
for every $\varepsilon\in(0,\varepsilon_0)$,
\begin{equation}
\label{FOX}
\left|
\varepsilon\int_{\Se}u^2\,d\Ae
-\int_{\So}u^2\,d\Ao
\right|
\leq
\eta\int_{\So}u^2\,d\Ao.
\end{equation}
\end{lemma}

\begin{proof}
Choose $\varrho\in(0,1)$ such that
\begin{equation}
\label{PAR}
L
\left(
\Ao\bigl(\So\cap\{|z|\leq\varrho\}\bigr)
+4\pi\sqrt{5}\arcsin(\varrho)
\right)
\leq\frac{\eta}{2}.
\end{equation}
By \eqref{Integral_estimate},
\[
\int_{\So\cap\{|z|\leq\varrho\}}u^2\,d\Ao
\leq
L\varrho^\gamma
\Ao\bigl(\So\cap\{|z|\leq\varrho\}\bigr)
\int_{\So}u^2\,d\Ao,
\]
where $\varrho^\gamma\leq 1$.
On the other hand, it follows from \eqref{A_e} that, for every
$\varepsilon\in(0,1]$, $s\in[0,\pi]$, and
$\theta\in[0,2\pi]$,
\[
\varepsilon
\|\partial_sF_\varepsilon\wedge\partial_\theta F_\varepsilon\|
\leq\sqrt{5}.
\]
Consequently,
\[
\varepsilon
\int_{\Se\cap\{|z|\leq\varrho\}}u^2\,d\Ae
\leq
4\pi\sqrt{5}\arcsin(\varrho)L\varrho^\gamma
\int_{\So}u^2\,d\Ao.
\]
It now follows from \eqref{PAR} that
\[
\left|
\varepsilon\int_{\Se\cap\{|z|\leq\varrho\}}u^2\,d\Ae
-\int_{\So\cap\{|z|\leq\varrho\}}u^2\,d\Ao
\right|
\leq
\frac{\eta}{2}\int_{\So}u^2\,d\Ao.
\]

Set $\delta=\arcsin(\varrho)$. On the complementary region, we have
\[
\begin{split}
\left|
\varepsilon\int_{\Se\cap\{|z|>\varrho\}}u^2\,d\Ae
-\int_{\So\cap\{|z|>\varrho\}}u^2\,d\Ao
\right|
\leq{}&
\int_\delta^{\pi-\delta}\int_0^{2\pi}
u\bigl(F_0(s,\theta)\bigr)^2\\
&\quad{}\times K_\varepsilon(s)\sin^2(s)\,d\theta\,ds,
\end{split}
\]
where, by \eqref{A_e},
\[
K_\varepsilon(s)
 =
\left|
\left(
\bigl|\ex^{is}-i\varepsilon k_\varepsilon(s)\bigr|^2
+\frac{\varepsilon^2\cos^2(s)}{\sin^2(s)}
\right)^{1/2}
-1
\right|.
\]
There exists
$\varepsilon_0=\varepsilon_0(\varrho,\eta)>0$ such that, for every
$\varepsilon\in(0,\varepsilon_0)$ and
$s\in[\delta,\pi-\delta]$,
\[
K_\varepsilon(s)\leq\frac{\eta}{2}.
\]
Therefore, by \eqref{area_0},
\[
\left|
\varepsilon\int_{\Se\cap\{|z|>\varrho\}}u^2\,d\Ae
-\int_{\So\cap\{|z|>\varrho\}}u^2\,d\Ao
\right|
\leq
\frac{\eta}{2}\int_{\So}u^2\,d\Ao.
\]
Combining the estimates on the two regions proves the claim.
\end{proof}

\begin{lemma}
\label{Lex3}
For every $\gamma>0$, $L>0$, and $\eta>0$, there exists
$\varepsilon_0=\varepsilon_0(\eta;\gamma,L)>0$ with the following property.
If $u\in\mathrm{Lip}(\So)$ satisfies \eqref{Integral_estimate}, then,
for every $\varepsilon\in(0,\varepsilon_0)$,
\begin{equation}
\label{Lex33}
\int_{\Se}
\left(
\operatorname{Ric}_\varepsilon(\Ne)+\|\hep\|^2
\right)u^2\,d\Ae
\leq
\eta\int_{\Se}u^2\,d\Ae.
\end{equation}
\end{lemma}

\begin{proof}
Choose $\varrho\in(0,1)$ such that
\begin{equation}
\label{riomare}
8\pi L\varrho^\gamma \leq\frac{\eta}{2}.
\end{equation}
We first consider the region where $|z|\leq\varrho$. By
\eqref{Integral_estimate}, \eqref{lampo}, and \eqref{riomare}, we have
\begin{equation}
\label{BAB}
\begin{split}
\int_{\Se\cap\{|z|\leq\varrho\}}
\left(
\operatorname{Ric}_\varepsilon(\Ne)+\|\hep\|^2
\right)u^2\,d\Ae
&\leq
8\pi L\varrho^\gamma \int_{\So}u^2\,d\Ao\\
&\leq
\frac{\eta}{2}\int_{\So}u^2\,d\Ao\\
&\leq
\frac{\eta}{2}\int_{\Se}u^2\,d\Ae,
\end{split}
\end{equation}
where, after decreasing $\varepsilon_0$ if necessary, the last
inequality follows from Lemma \ref{leox}.

We now  estimate the integral over
$\Se\cap\{|z|>\varrho\}$. By \eqref{QUQQO}, there exists
$\varepsilon_0=\varepsilon_0(\varrho,\eta)>0$ such that, for every
$\varepsilon\in(0,\varepsilon_0)$ and every $(z,t)\in\Se$ with
$|z|\geq\varrho$,
\[
\operatorname{Ric}_\varepsilon(\Ne)+\|\hep\|^2
\leq\frac{\eta}{2}.
\]
It follows that
\begin{equation}
\label{BIB}
\int_{\Se\cap\{|z|\geq\varrho\}}
\left(
\operatorname{Ric}_\varepsilon(\Ne)+\|\hep\|^2
\right)u^2\,d\Ae
\leq
\frac{\eta}{2}\int_{\Se}u^2\,d\Ae.
\end{equation}
Combining \eqref{BAB} and \eqref{BIB} proves the claim.
\end{proof}

\medskip

\begin{proof}[Proof of Theorem \ref{8.1}]
Fix $\gamma>0$ and $L>0$, and let $m=m(\gamma,L)>0$ be the constant given by Lemma
\ref{POI}. Choose $\eta>0$ such that
\begin{equation}
\label{fulmine}
\frac{m}{4(1+\eta)}-\eta\geq\frac{m}{8}.
\end{equation}
Let $\varepsilon_0>0$ be sufficiently small that Lemmas \ref{L5},
\ref{leox}, and \ref{Lex3} apply, and fix
$\varepsilon\in(0,\varepsilon_0)$. Let
$u\in\mathrm{Lip}(\So)$ satisfy \eqref{Integral_estimate}. By
\eqref{energy} and \eqref{Lex33},
\[
\begin{split}
\int_{\Se}
\left\{
\|\nablaSe u\|^2
-\left(
\operatorname{Ric}_\varepsilon(\Ne)+\|\hep\|^2
\right)u^2
\right\}\,d\Ae
&\geq
\frac{1}{4\varepsilon}
\int_{\So}\|\nablaSo u\|^2\,d\Ao
-\eta\int_{\Se}u^2\,d\Ae\\
&\geq
\frac{m}{4\varepsilon}
\int_{\So}u^2\,d\Ao
-\eta\int_{\Se}u^2\,d\Ae,
\end{split}
\]
where the second inequality follows from the Poincar\'e inequality
\eqref{PON}. By \eqref{FOX},
\[
\varepsilon\int_{\Se}u^2\,d\Ae
\leq
(1+\eta)\int_{\So}u^2\,d\Ao.
\]
Consequently,
\[
\begin{split}
\int_{\Se}
\left\{
\|\nablaSe u\|^2
-\left(
\operatorname{Ric}_\varepsilon(\Ne)+\|\hep\|^2
\right)u^2
\right\}\,d\Ae
&\geq
\left(
\frac{m}{4(1+\eta)}-\eta
\right)
\int_{\Se}u^2\,d\Ae\\
&\geq
\frac{m}{8}\int_{\Se}u^2\,d\Ae,
\end{split}
\]
where the last inequality follows from \eqref{fulmine}.
\end{proof}

\section{Expansion of the isoperimetric functional}
\label{EXP}

In this section, we compute the Taylor expansion of the isoperimetric
functional up to second order, with the remainder in integral form, and
outline the estimates carried out in the following sections. The
second-order term can be expressed in different ways, which exhibit
different behaviour as the parameter $\varepsilon>0$ varies.

Fix $\varepsilon>0$. For a function $u\in\mathrm{Lip}(\Se)$, let
$\Se^u$ be its normal graph and $\Be^u$ the bounded domain such that
$\partial\Be^u=\Se^u$. We consider the isoperimetric functional
$\F:\mathrm{Lip}(\Se)\to\rr$ defined by
\[
\F(u)=\Ae(\Se^u)-H\Ve(\Be^u),
\]
where $H=2$ is the constant mean curvature of $\Se$.

For fixed $u$ and $s\in[0,1]$, set
\[
\Fep(s)=\F(su).
\]
The function $\Fep$ has the following Taylor expansion:
\begin{equation}
\label{G1}
\Fep(1)
 =\Fep(0)+\Fep'(0)+\frac{1}{2}\Fep''(0)
 +\int_0^1(1-s)\bigl(\Fep''(s)-\Fep''(0)\bigr)\,ds.
\end{equation}
The computations of the various terms are well known. For the sake of clarity, we recall them in the following results.

We assume that $u:\Se\to\rr$ is (continuous and) the lift to $\Se$ of a piecewise $C^2$ function on $\So$. We denote by
$\Nsu$ the outer unit normal to the graph $\Se^{su}$, by $\hsu$ its
second fundamental form, and by $\Hsu$ its mean curvature. Consider the
geodesic flow
$
\Psi_\varepsilon^s=\Psi_\varepsilon(\,\cdot\,,s)
$
introduced in \eqref{Psie1}--\eqref{Psie2}. Its generating vector field
\[
\Ne=\frac{\partial}{\partial s}\Psi_\varepsilon^s
\]
coincides on $\Se$ with its outer unit normal. To perform the
computations for fixed $\eps>0$, we extend $u$ constantly along the
flow $\Psi_\varepsilon^s$ and define $\bar u:\Se^{su}\to\rr$ by
\[
\bar u=\langle\Ne,\Nsu\rangle u.
\]

\begin{theorem}[First and second variations of area]
\label{pacco}
Let $\varepsilon>0$, and let $u:\Se\to\rr$ be a function of class $C^2$. Then, for every $s\in[0,1]$,
\begin{equation}
\label{fvareas}
\frac{d}{ds}\Ae(\Se^{su})
 =\int_{\Se^{su}}\Hsu\bar u\,d\Ae,
\end{equation}
and
\begin{equation}
\label{svareas}
\frac{d^2}{ds^2}\Ae(\Se^{su})
 =
\int_{\Se^{su}}
\left\{
\|\nabla_{\Se^{su}}\bar u\|^2
+\left(
(\Hsu)^2-\operatorname{Ric}_\varepsilon(\Nsu)-\|\hsu\|^2
\right)\bar u^2
\right\}\,d\Ae.
\end{equation}
\end{theorem}

\begin{proof}
Let
$
N_\varepsilon^\perp
 =\langle\Ne,\Nsu\rangle\Nsu
$
be the component of $\Ne$ orthogonal to $T\Se^{su}$, and let
$
N_\varepsilon^\top=\Ne-N_\varepsilon^\perp
$
be its tangential component. By Stokes' theorem,
\begin{equation}
%\label{IP}
\begin{split}
\frac{d}{ds}\Ae(\Se^{su})
&=
\int_{\Se^{su}}
\operatorname{div}_{\Se^{su}}(u\Ne)\,d\Ae\\
&=
\int_{\Se^{su}}
\operatorname{div}_{\Se^{su}}
\bigl(uN_\varepsilon^\top+uN_\varepsilon^\perp\bigr)\,d\Ae\\
&=
\int_{\Se^{su}}
\operatorname{div}_{\Se^{su}}(uN_\varepsilon^\perp)\,d\Ae\\
&=
\int_{\Se^{su}}
\operatorname{div}_{\Se^{su}}(\bar u\Nsu)\,d\Ae\\
&=
\int_{\Se^{su}}
\bar u\,\operatorname{div}_{\Se^{su}}(\Nsu)\,d\Ae.
\end{split}
\end{equation}
This proves \eqref{fvareas}.

Differentiating once more with respect to $s$, which amounts to
differentiating in the direction $u\Ne$, we obtain
\[
\begin{split}
\frac{d^2}{ds^2}\Ae(\Se^{su})
=
\int_{\Se^{su}}
\Bigl\{&
\bar u\,\operatorname{div}_{\Se^{su}}(\Nsu)
       \operatorname{div}_{\Se^{su}}(u\Ne)\\
&+
u\bigl(N_\varepsilon^\top+N_\varepsilon^\perp\bigr)
\left[
\operatorname{div}_{\Se^{su}}(\bar u\Nsu)
\right]
\Bigr\}\,d\Ae.
\end{split}
\]
Integration by parts gives
\begin{equation}
\label{IP}
\begin{split}
\int_{\Se^{su}}
uN_\varepsilon^\top
\left[
\operatorname{div}_{\Se^{su}}(\bar u\Nsu)
\right]\,d\Ae
=
-\int_{\Se^{su}}
\bar u\,
\operatorname{div}_{\Se^{su}}(uN_\varepsilon^\top)
\operatorname{div}_{\Se^{su}}(\Nsu)\,d\Ae.
\end{split}
\end{equation}
On the other hand, the contribution from the normal component is
\[
\begin{split}
\int_{\Se^{su}}
uN_\varepsilon^\perp
\left[
\operatorname{div}_{\Se^{su}}(\bar u\Nsu)
\right]\,d\Ae
&=
\int_{\Se^{su}}
\bar u\Nsu
\left[
\operatorname{div}_{\Se^{su}}(\bar u\Nsu)
\right]\,d\Ae\\
&=
\int_{\Se^{su}}
\left\{
\|\nabla_{\Se^{su}}\bar u\|^2
-Q_\varepsilon^{su}\bar u^2
\right\}\,d\Ae,
\end{split}
\]
where $
Q_\varepsilon^{su}
 =\operatorname{Ric}_\varepsilon(\Nsu)+\|\hsu\|^2,
$  see \cite{Ritore2023}.

Combining these identities, we obtain
\[
\begin{split}
\frac{d^2}{ds^2}\Ae(\Se^{su})
={}&
\int_{\Se^{su}}
\bar u\,\operatorname{div}_{\Se^{su}}(\Nsu)
\left\{
\operatorname{div}_{\Se^{su}}(u\Ne)
-\operatorname{div}_{\Se^{su}}(uN_\varepsilon^\top)
\right\}\,d\Ae\\
&+
\int_{\Se^{su}}
\left\{
\|\nabla_{\Se^{su}}\bar u\|^2
-Q_\varepsilon^{su}\bar u^2
\right\}\,d\Ae\\
={}&
\int_{\Se^{su}}
\left\{
\|\nabla_{\Se^{su}}\bar u\|^2
-Q_\varepsilon^{su}\bar u^2
+(\Hsu)^2\bar u^2
\right\}\,d\Ae.
\end{split}
\]
This proves \eqref{svareas}.
\end{proof}

%For future reference, we observe that there exists and absolute constant $C>0$ such that
%\begin{equation}
%\label{picco}
%\Big| \int_{T_{\eps} ^{su}}
%u ^2 \langle \Ne,\mathscr{N}^{su}_{\eps,+}   \rangle 
%\operatorname{div}_{\Se^{su}}(\mathscr{N}^{su}_{\eps,+})
%\langle N_\varepsilon^\top, \mathscr{M}^{su}_{\eps,+}\rangle 
% d\mathscr{L}_\eps \Big| \leq C\varrho^2 \max_{T_{\eps} ^{su}}  |H_{\eps,+} ^{su}| \mathscr{L}_\eps(T_{\eps} ^{su})
%\end{equation}

\begin{theorem}[First and second variations of volume]\label{pacco2}
Let $\varepsilon>0$, and let $u:\Se\to\rr$ be a  function  of class $C^2$. Then, for every $s\in[0,1]$,
\begin{equation}
\label{VV1}
\frac{d}{ds}\Ve(\Be^{su})
 =\int_{\Se^{su}}\bar u\,d\Ae,
\end{equation}
and
\begin{equation}
\label{VV2}
\frac{d^2}{ds^2}\Ve(\Be^{su})
 =
\int_{\Se^{su}}
\left\{\bar u \Nsu(\bar u)
+\bar u^2\Hsu
\right\}\,d\Ae.
\end{equation} 
\end{theorem}

\begin{proof}
We have
\[
\begin{split}
\frac{d}{ds}\Ve(\Be^{su})
&=
\left.
\frac{d}{dr}
\int_{\Be^{(s+r)u}}d\Ve
\right|_{r=0}\\
&=
\left.
\frac{d}{dr}
\int_{\Be^{su}}
\operatorname{Jac}\bigl(\Psi_\varepsilon(\,\cdot\,,ru)\bigr)\,d\Ve
\right|_{r=0}\\
&=
\int_{\Be^{su}}
\operatorname{div}_\varepsilon(u\Ne)\,d\Ve\\
&=
\int_{\Se^{su}}
u\langle\Ne,\Nsu\rangle\,d\Ae\\
&=
\int_{\Se^{su}}\bar u\,d\Ae.
\end{split}
\]
Here we have used the standard differentiation formulae for the flow,
together with the divergence theorem.

Formula \eqref{VV2} follows by differentiating \eqref{VV1} and using
the geodesic equation
\[
\nabla_{\Ne}\Ne=0,
\]
the identity $\Ne u=0$, and the first-variation formula
\eqref{fvareas}: 
\begin{equation}\label{IP2}
\begin{split}
\frac{d^2}{ds^2} \Ve  (B_{\eps}^{su})&=\int_{\Se^{su}}  \big\{
u\Ne  \bar u +\bar u  \divv_{\Se^{su}}(u\Ne ) \big\} d\Ae
\\ 
&
=\int_{\Se^{su}}  \big\{
u\Ne \bar   u  +{\bar u}^2 \divv_{\Se^{su}}(\Nsu )-u N_\eps^\top \bar u \big\} 
d\Ae\\&=\int_{\Se^{su}}\big\{
\bar u \Nsu(\bar u)
+\bar u^2\Hsu\big\}
\,d\Ae,
\end{split}
\end{equation}
where we used the decomposition $\Ne=\Ne^\top+\Ne^\perp=\Ne^\top+\langle\Ne,\Nsu\rangle\Nsu$ into its tangential and orthogonal parts with respect to $\Se^{su}$.
\end{proof}

\begin{remark}

We need a version of Theorem \ref{pacco} and Theorem \ref{pacco2}, and in particular of formula
\eqref{svareas}, for the case in which $u$ is the lift to $\Se$ of a
continuous function on $\So$ that is piecewise $C^2$. More precisely,
$
u\in C^2\bigl(\So\cap\{|z|\geq\varrho\}\bigr)
$
and
$
u(p)\in\{t^-(p),t^+(p)\}$
 when $|z|\leq\varrho$,
as in assumption (1) of Theorem \ref{8.1}. In this case, the integration
by parts in \eqref{IP} produces an additional boundary term.

Let $\varrho\in(0,1)$ be the parameter associated with $u$. For
$s\in[0,1]$, we set
\[
\begin{split}
S_{\eps,+}^{su}
&=
\bigl\{\Psi_\varepsilon^u(p,s)\in\hh^1:
p=(z,t)\in\Se,\ |z|\geq\varrho\bigr\},
\\
S_{\eps,-}^{su}
&=
\bigl\{\Psi_\varepsilon^u(p,s)\in\hh^1:
p=(z,t)\in\Se,\ |z|\leq\varrho\bigr\}.
\end{split}
\]
On the region $S_{\eps,+}^{su}$, the integration-by-parts identity used
in \eqref{IP} is replaced by
\[
\begin{split}
&\int_{S_{\eps,+}^{su}}
uN_\varepsilon^\top
\left[
\operatorname{div}_{\Se^{su}}(\bar u\Nsu)
\right]\,d\Ae
\\
&\qquad =
-\int_{S_{\eps,+}^{su}}
\bar u\,
\operatorname{div}_{\Se^{su}}(uN_\varepsilon^\top)
\operatorname{div}_{\Se^{su}}(\Nsu)\,d\Ae
+ E_{\eps,+}^{su},
\end{split}
\]
where, by divergence theorem,
\begin{equation}\label{folle}
\begin{split}
E_{\eps,+}^{su}
&=
\int_{S_{\eps,+}^{su}}
\operatorname{div}_{\Se^{su}}
\left(
\bar u\,
\operatorname{div}_{\Se^{su}}(\Nsu)\,
uN_\varepsilon^\top
\right)\,d\Ae
\\
&=
\int_{T_{\eps}^{su}}
u^2
\langle\Ne,\mathscr{N}_{\eps,+}^{su}\rangle
\operatorname{div}_{\Se^{su}}(\mathscr{N}_{\eps,+}^{su})
\langle N_\varepsilon^\top,\mathscr{M}_{\eps,+}^{su}\rangle
\,d\mathscr{L}_\eps.
\end{split}
\end{equation}
Here, $T_{\eps}^{su}$ is the relative boundary of
$S_{\eps,+}^{su}$ in $S_{\eps}^{su}$;
$\mathscr{N}_{\eps,+}^{su}$ is the continuous extension to
$T_{\eps}^{su}$ of the unit normal $\mathscr{N}_{\eps}^{su}$ defined on
$S_{\eps,+}^{su}$;
$\mathscr{M}_{\eps,+}^{su}$ is the outward unit conormal to
$T_{\eps}^{su}$ in $S_{\eps,+}^{su}$; and
$\mathscr{L}_\eps$ is the length measure induced by the metric $\gep$.

Analogously, the integration-by-parts identity used
in \eqref{IP2} is replaced by
$$
\int_{S_{\eps,+}^{su}}
\bar u 
\operatorname{div}_{\Se^{su}}(u \Ne^\top)d\Ae
=-\int_{S_{\eps,+}^{su}}u\Ne^\top\bar u d\Ae+\widetilde E_{\eps,+}^{su},
$$
where
\begin{equation}\label{folle2}
    \begin{split}
        \widetilde E_{\eps,+}^{su}
        =\int_{S_{\eps,+}^{su}} \operatorname{div}_{S_{\eps}^{su}}(\bar u u \Ne^\top)
        \, d\Ae
        =\int_{T_\eps^{su}}u^2
\langle\Ne,\mathscr{N}_{\eps,+}^{su}\rangle
\langle N_\varepsilon^\top,\mathscr{M}_{\eps,+}^{su}\rangle
\,d\mathscr{L}_\eps.
    \end{split}
\end{equation}

An analogous formula holds on $S_{\eps,-}^{su}$. Consequently, under
the piecewise-$C^2$ assumption, the right-hand side of \eqref{svareas}
contains the additional term
\begin{equation} \label{pix}
E_\eps^{su}=E_{\eps,+}^{su}+E_{\eps,-}^{su},
\end{equation}
and the right-hand side of \eqref{VV2} contains the additional term
\begin{equation} \label{pox}
\widetilde E_\eps^{su}=\widetilde E_{\eps,+}^{su}+\widetilde E_{\eps,-}^{su}.
\end{equation}
\end{remark}

We now describe the estimates needed in applying the expansion
\eqref{G1}. At $s=0$, the first- and second-order terms in \eqref{G1}
are
\[
\Fep'(0)
 =\int_{\Se}Hu\,d\Ae
 -H\int_{\Se}u\,d\Ae
 =0
\]
and
%, using the shorthand notation
%\[
%Q_\varepsilon
 %=\operatorname{Ric}_\varepsilon(\Ne)+\|\hep\|^2,
%\]
\begin{equation}
\label{G_eq:second}
\Fep''(0)
 =\int_{\Se}
 \left\{
 \|\nablaSe u\|^2-Q_\varepsilon u^2
 \right\}\,d\Ae.
\end{equation}
Under the assumptions of Theorem \ref{8.1}, the latter term satisfies
the uniform lower bound
\[
\Fep''(0)
 \geq\frac{m}{8}\|u\|_{L^2(\Se)}^2
\]
for every sufficiently small $\varepsilon>0$, as stated in
\eqref{USe}.

The difference appearing in the integral remainder in \eqref{G1} can
be computed using \eqref{svareas} and \eqref{VV2}. For every
$s\in[0,1]$, it is given by
\[
\begin{split}
\Fep''(s)-\Fep''(0)
={}&
\int_{\Se^{su}}
\left\{
\|\nabla_{\Se^{su}}\bar u\|^2
+\left((\Hsu)^2-Q_\varepsilon^{su}\right)\bar u^2
\right\}\,d\Ae\\
&-
H\int_{\Se^{su}}
\left\{ \bar u \Nsu(\bar u)
+\bar u^2\Hsu
\right\}\,d\Ae\\
&-
\int_{\Se}
\left\{
\|\nablaSe u\|^2-Q_\varepsilon u^2
\right\}\,d\Ae,
\end{split}
\]
where
$
Q_\varepsilon^{su}
 =\operatorname{Ric}_\varepsilon(\Nsu)+\|\hsu\|^2.
$

For fixed $\eps>0$, the standard argument proceeds by showing that,
for every $\eta>0$, there exists
$\delta=\delta(\varepsilon,\eta)>0$ such that
\begin{equation}
\label{FOP}
\left|\Fep''(s)-\Fep''(0)\right|
\leq
\eta\int_{\Se}
\left\{
\|\nablaSe u\|^2+u^2
\right\}\,d\Ae
\end{equation}
for every $s\in[0,1]$, provided that
$\|u\|_{C^1(\Se)}<\delta$.

This estimate can be proved using the following first- and second-variation
formulae:
\begin{equation}
\frac{d}{ds}\Ae(\Se^{su})
 =
\int_{\Se^{su}}
\operatorname{div}_{\Se^{su}}(u\Ne)\,d\Ae,
\end{equation}
and
\begin{equation}
\label{ZX}
\frac{d^2}{ds^2}\Ae(\Se^{su})
 =
\int_{\Se^{su}}
\left\{
u\Ne\left[\operatorname{div}_{\Se^{su}}(u\Ne)\right]
+\left(\operatorname{div}_{\Se^{su}}(u\Ne)\right)^2
\right\}\,d\Ae.
\end{equation}
In our setting, the integrand in \eqref{ZX} contains terms with
coefficients $1/\eps^2$ and $1/\eps$, as well as non-singular terms.
The term with coefficient $1/\eps^2$ is the following (we omit the
computation):
\begin{equation}
\frac{\tau^2}{\eps^2}u^2\big\{3(1-\langle N_\eps , \Nsu \rangle^2)-4\langle \Nsu \wedge N_\eps , T_\eps \rangle^2\big\}.
\end{equation}
When $s=0$, we have $\Nsu=\Ne$, and hence $\{\ldots\}=0$. The
difficulty is that smallness of $\|u\|_{C^1}$ does not ensure that
\[
\frac{\tau^2}{\eps^2}\{\ldots\}
\]
is small uniformly in $\eps>0$.

The form of \eqref{ZX} differs from that of the geometric formula
\eqref{svareas} because of the integration by parts used in
\eqref{IP}:
\[
\begin{split}
\frac{d}{ds}\Ae(\Se^{su})
&=
\int_{\Se^{su}}
\operatorname{div}_{\Se^{su}}(u\Ne)\,d\Ae\\
&=
\int_{\Se^{su}}
\bar u\,\operatorname{div}_{\Se^{su}}(\Nsu)\,d\Ae.
\end{split}
\]
After this integration by parts, the geometric terms of order
$1/\eps^2$ cancel; see Theorem \ref{CANC} below. However, an additional
derivative of $u$ appears through
$\operatorname{div}_{\Se^{su}}(\Nsu)$. This is why assumption (3) of
Theorem \ref{F0} requires the smallness of a $C^2$-type norm of $u$ in
order to prove \eqref{FOP}.

The geometric terms in $\Fep''(s)-\Fep''(0)$ are singular with respect
to both $\varepsilon$ and $|z|$ as $\varepsilon\to0$ and
$(z,t)\in\Se$ approaches either pole, respectively. Assumptions (4)
and (5) of Theorem \ref{F0} control these singularities.

In Sections \ref{TEK1} and \ref{TEK2}, we prove preliminary estimates
for the normal $\Nsu$. These estimates are used in Section \ref{UC},
where we establish an estimate of the form \eqref{FOP}, with $\eta$
small, uniformly in $\varepsilon>0$; see Theorem
\ref{unifcontSecondVar}.

\section{Estimates for the normal $\Nsu$}
\label{TEK1}

We prove a first estimate needed to establish \eqref{FOP}. More
precisely, we show that, under a suitable smallness assumption on $u$
and uniformly in $\eps>0$, the normal $\Nsu$ to the graph $\Se^{su}$
is close to $\Ne$. Here, $\Ne$ is the vector field generating the flow
\[
\Psi_\varepsilon^s(p)=\Psi_\varepsilon(p,s),
\qquad p=(z,t)\in\Se,
\]
introduced in \eqref{Psie1} and \eqref{Psie2}.
In the coordinates $[\cdots]_*$ and using the shorthand notation
$s_\varepsilon=2\tau p_\varepsilon s$, where $p_\eps$ is as in \eqref{LUX}, the vector field $\Ne$ is given
by (northern hemisphere):
\begin{equation}
\label{nix}
\Ne\bigl(\Psi_\varepsilon^{su}(p)\bigr)
 =
\left[
z(1-\tau p_\varepsilon i)\ex^{-is_\varepsilon u},
\varepsilon p_\varepsilon
\right]_*.
\end{equation}

Henceforth, we extend a given function $u:\Se\to\rr$ to a function
$u:\rr^3=\hh^1\to\rr$, and use the notation
\[
|\nabla u|
 =(u_x^2+u_y^2+u_t^2)^{1/2}
\]
for the Euclidean norm of its gradient. We also use the following
weighted $C^1$-seminorm:
\begin{equation}
\label{w1}
|u(p)|_{1,w}
 =\frac{|u(p)|}{|z|}+|\nabla u(p)|,
\qquad
p=(z,t)\in\hh^1\setminus Z.
\end{equation}
Most of the estimates are expressed in terms of the recurring quantity
\begin{equation}
\label{ER}
\er
 =\frac{\eps}{\sqrt{\eps^2+\tau^2r^2}},
\qquad r=|z|.
\end{equation}

The goal of this section is to prove the following theorem.

\begin{theorem}
\label{Socrates}
Let $u:\hh^1\to\rr$ be a function such that
$|u(p)|_{1,w}\leq L$ or every $p=(z,t)\in\hh^1\setminus Z$,
for some $L>0$. Then there exists $C_L>0$ such that, for every
$p\in\Se\setminus Z$, $s\in[0,1]$, and $\varepsilon>0$,
\begin{equation}
\label{EQUA1}
\left\|
\Nsu\bigl(\Psi_\varepsilon^{su}(p)\bigr)
-\Ne\bigl(\Psi_\varepsilon^{su}(p)\bigr)
\right\|
\leq C_L|u(p)|_{1,w},
\end{equation}
and
\begin{equation}
\label{EQUA2}
\left|
\left\langle
\Nsu\bigl(\Psi_\varepsilon^{su}(p)\bigr)
-\Ne\bigl(\Psi_\varepsilon^{su}(p)\bigr),
\Te
\right\rangle
\right|
\leq
\er C_L|u(p)|_{1,w}.
\end{equation}
\end{theorem}

The proof follows from Lemmas \ref{lello1}--\ref{lello3} below.
We denote by
\[
\widehat N_\varepsilon\bigl(\Psi_\varepsilon^{su}(p)\bigr)
 \in T_{\Psi_\varepsilon^{su}(p)}\hh^1
\]
the vector which, at $\Psi_\varepsilon^{su}(p)$, has the same
$[\cdots]_*$ coordinates as the normal $\Ne$ at $p\in\Se$. Thus, we set
\begin{equation}
\label{nullo}
\widehat N_\varepsilon\bigl(\Psi_\varepsilon^{su}(p)\bigr)
 =
\left[
z(1-\tau p_\varepsilon i),
\varepsilon p_\varepsilon
\right]_*.
\end{equation}
First, we estimate the difference
$\Ne-\widehat N_\varepsilon$.

\begin{lemma}
\label{lello1}
 Let $u:\hh^1\to\rr$ be a function such that
$|u(p)|_{1,w}\leq L$ or every $p=(z,t)\in\hh^1\setminus Z$,
for some $L>0$. Then there exists a constant $C_L>0$ such that, for
every $p\in\Se\setminus Z$, $s\in[0,1]$, and $\varepsilon\geq0$,
\begin{equation}
\label{EQUA1ter}
\big\|
\Ne\bigl(\Psi_\varepsilon^{su}(p)\bigr)
-\widehat N_\varepsilon\bigl(\Psi_\varepsilon^{su}(p)\bigr)
\big\|
\leq C_L\tau p_\varepsilon|u(p)|.
\end{equation}
\end{lemma}

\begin{proof}
Comparing \eqref{nix} and \eqref{nullo}, and using
$
\bigl|z(1-\tau p_\varepsilon i)\bigr|\leq1$,
we obtain
\[
\begin{split}
\big\|
\Ne\bigl(\Psi_\varepsilon^{su}(p)\bigr)
-\widehat N_\varepsilon\bigl(\Psi_\varepsilon^{su}(p)\bigr)
\big\|
&=
\left|
z(1-\tau p_\varepsilon i)
\bigl(\ex^{-is_\varepsilon u(p)}-1\bigr)
\right|\\
&\leq
\left|
\sum_{n=1}^{\infty}
\frac{\bigl(-i2\tau p_\varepsilon su(p)\bigr)^n}{n!}
\right|\\
&\leq
\ex^{2\tau p_\varepsilon|u(p)|}-1.
\end{split}
\]
Since
$
|\tau p_\varepsilon u(p)|
\leq\tau p_\varepsilon L|z|
\leq L,
$
the claim follows.
\end{proof}

We fix the following orthonormal frame for the tangent space $T_p\Se$
at a point $p=(z,t)$ with $z\neq0$:
\begin{equation}
\label{NOTA1}
\begin{split}
X_1
&=
\frac{1}{r}
\sqrt{\frac{\varepsilon^2+\tau^2r^2}{\varepsilon^2+\tau^2}}
\bigl[z(\tau p_\varepsilon+i),0\bigr]_*
\\
&=[w_1,c_1]_*
 =\bigl(w_1,\tau\mathop{\rm Im}(z\bar w_1)\bigr)_*,\\
X_2
&=
\frac{\varepsilon}{r}
\sqrt{\frac{1-r^2}{\varepsilon^2+\tau^2}}
\left[
z(1-\tau p_\varepsilon i),
-\frac{(\varepsilon^2+\tau^2)r^2}
{\varepsilon\sqrt{(1-r^2)(\varepsilon^2+\tau^2r^2)}}
\right]_*\\
&=[w_2,c_2]_*
 =\bigl(w_2,\varepsilon c_2+\tau\mathop{\rm Im}(z\bar w_2)\bigr)_*.
\end{split}
\end{equation}
The functions $w_1,w_2,c_1$, and $c_2$ are defined by the preceding
identities; explicitly,
\begin{equation}
\label{NOTA1bis}
\begin{split}
w_1
&=
\frac{1}{r}
\sqrt{\frac{\varepsilon^2+\tau^2r^2}{\varepsilon^2+\tau^2}}
z(\tau p_\varepsilon+i),\\
c_1&=0,\\
w_2
&=
\frac{\varepsilon}{r}
\sqrt{\frac{1-r^2}{\varepsilon^2+\tau^2}}
z(1-\tau p_\varepsilon i),\\
c_2
&=
-\frac{r\sqrt{\varepsilon^2+\tau^2}}
{\sqrt{\varepsilon^2+\tau^2r^2}}.
\end{split}
\end{equation}
Notice that
\begin{equation}
\label{NOTA}
|w_1|=1,
\qquad
w_2=-iw_1|w_2|,
\qquad
|w_2|=\varepsilon p_\varepsilon.
\end{equation}
In particular, $X_1$ is horizontal, and
\[
\|X_1\|=\|X_2\|=1,
\qquad
\langle X_1,X_2\rangle=0.
\]

\begin{remark}
It follows from \eqref{NOTA1} that the tangential derivatives of $u$
satisfy the following estimates:
\begin{equation}
\label{STILO1}
\begin{split}
|X_1u|
&\leq
\sqrt{|w_1|^2+\tau^2\mathop{\rm Im}(z\bar w_1)^2}\,
|\nabla u|\\
&\leq
 C |\nabla u|,\\
|X_2u|
&\leq
\sqrt{
|w_2|^2+
\bigl(\varepsilon c_2+\tau\mathop{\rm Im}(z\bar w_2)\bigr)^2
}\,|\nabla u|\\
&\leq C\er|\nabla u|,
\end{split}
\end{equation}
 for a constant $C>0$ independent of $\eps$ and $p\in S_\eps$. Furthermore,
\begin{equation}
\label{STILO}
\begin{split}
& X_1 r 
  = \frac 1 r \mathrm{Re}( w_1 \bar z) = \tau p_\varepsilon \sqrt{\frac{\varepsilon^2+\tau^2 r^2 }{\varepsilon^2 +\tau^2} }
 = \frac{\tau }{\sqrt{\varepsilon ^2+\tau^2} } \sqrt{1-r^2}
 \\
 &
 %  \quad \textrm{and}\quad 
 X_2 r = \frac 1 r \mathrm{Re}( w_2 \bar z) 
 %= \varepsilon  \sqrt{\frac{1-r^2 }{\varepsilon^2 +\tau^2} }
  = \frac{\varepsilon }{\sqrt{\varepsilon ^2+\tau^2} } \sqrt{1-r^2},
  \\
  & X_1 p_\varepsilon=-\tau\frac{ r\sqrt{\eps^2+\tau^2} }{(\sqrt{\varepsilon^2 +\tau^2 r^2})^3 }\quad \textrm{and}\quad
  X_2 p_\varepsilon =-\eps \frac{ r\sqrt{\eps^2+\tau^2} }{(\sqrt{\varepsilon^2 +\tau^2 r^2})^3 }.
\end{split}
\end{equation}
 The computations for $p_\varepsilon$ are obtained in a similar way as for $r$. Hence, we have the following estimates for the first-order derivatives of $p_\varepsilon$; for the second-order, the computations are analogous. For $i,j=1,2$, 
 \begin{equation}
\label{EL2}
|X_ip_\eps|
 \leq C\eps^{i-3}\er^2,
\qquad
|X_iX_jp_\eps|
 \leq C\eps^{i+j-5}\er^3.
\end{equation}
\end{remark}

We extend the frame $X_1,X_2\in T_p\Se$ to a tangent frame on the
surfaces
$
\Se^{su}=\Psi_\varepsilon^{su}(\Se).
$
For every $s\in[0,1]$, we define
\begin{equation}
\label{triang}
X_1^{su}\bigl(\Psi_\varepsilon^{su}(p)\bigr)
 =\Psi_{\varepsilon,*}^{su}X_1(p),
\qquad
X_2^{su}\bigl(\Psi_\varepsilon^{su}(p)\bigr)
 =\Psi_{\varepsilon,*}^{su}X_2(p),
\end{equation}
where $*$ denotes the differential of the map
$\Psi_\varepsilon^{su}$. When no confusion can arise, we simply write
$X_1^{su}=X_1$ and $X_2^{su}=X_2$.

Using the formulae for $\Psi_\eps$ in \eqref{Psie1} and
\eqref{Psie2}, we compute, for $k=1,2$,
\[
X_k\Psi_{\varepsilon,1}^{su}
 =w_k+v_k,
\qquad
X_k\Psi_{\varepsilon,2}^{su}
 =\varepsilon c_k+\tau\mathop{\rm Im}(z\bar w_k)+d_k,
\]
where
\begin{equation}
\label{LOP}
\begin{split}
v_k
={}&
sw_ku(\tau p_\varepsilon+i)\fa
+sz\Bigl[
i(X_ku)\fa
+\tau X_k(up_\varepsilon)
\bigl(
\fa+2su(\tau p_\varepsilon+i)\fa'
\bigr)
\Bigr],\\
d_k
={}&
s\varepsilon^2X_k(up_\varepsilon)
+\tau s^2X_k\bigl(u^2(1-\varepsilon^2p_\varepsilon^2)\bigr)\fb\\
&+
2\tau^2s^3u^2(1-\varepsilon^2p_\varepsilon^2)
\fb'X_k(up_\varepsilon)\\
&+
s\tau X_k(r^2u)
\bigl(\fc+\tau p_\varepsilon\fd\bigr)\\
&+
2\tau^2s^2r^2u
\bigl(\fc'+\tau p_\varepsilon\fd'\bigr)
X_k(up_\varepsilon)\\
&+
s\tau^2r^2u(X_kp_\varepsilon)\fd.
\end{split}
\end{equation}
Henceforth, we use the shorthand notation
$
\beta_j=\beta_j(s_\varepsilon u)$ and $ 
\beta_j'=\beta_j'(s_\varepsilon u)$.
Notice that if $|u|\leq L|z|$ and $s\in[0,1]$, then
$
|s_\varepsilon u|\leq2L$.
It therefore follows from \eqref{LUX} and \eqref{LAX} that
\begin{equation}
\label{LOX}
\sum_{j=1}^4
\left(
|\beta_j|+|\beta_j'|
\right)
\leq C_L,
\end{equation}
where $C_L>0$ depends only on $L$.

In the following lemma, we estimate $v_k$ and $d_k$. When the point $p\in\Se$ is clear from the context, we abbreviate
\[
|u|_{1,w}=|u(p)|_{1,w},
\qquad
|u|_{2,w}=|u(p)|_{2,w}.
\]

\begin{lemma}
\label{lello2} 
 Let $u:\hh^1\to\rr$  be a function such that $|u|\leq L |z|$. 
Then there exists a constant $C_L>0$ such that, for every
$s\in[0,1]$, $(z,t)\in\Se$ with $r=|z|>0$, and $\varepsilon>0$,
\begin{equation}
\label{ZED}
\begin{split}
|v_1|
&\leq C_L|u|_{1,w},\\
|v_2|
&\leq \er C_L|u|_{1,w},\\
|d_1|
&\leq
C_L\sqrt{\varepsilon^2+\tau^2r^2}\,|u|_{1,w},\\
|d_2|
&\leq
\varepsilon C_L\bigl(|u|+|\nabla u|\bigr).
\end{split}
\end{equation}
\end{lemma}

\begin{proof}
It follows from \eqref{LOP} and \eqref{LOX} that
\begin{equation}
\label{est-vk}
\begin{split}
|v_k|
&\leq
C_L\Bigl\{
|w_k||u||\tau p_\varepsilon+i|
+r\bigl[
|X_ku|
+\tau|X_k(up_\varepsilon)|
\bigl(1+|u||\tau p_\varepsilon+i|\bigr)
\bigr]
\Bigr\}\\
&\leq
C_L\Bigl\{
|w_k||u||\tau p_\varepsilon+i|
+r|X_ku|
+\tau r|X_k(up_\varepsilon)|
\Bigr\}\\
&\leq
C_L\Bigl\{
|w_k||u||\tau p_\varepsilon+i|
+|X_ku|
+\tau r|u||X_kp_\varepsilon|
\Bigr\},
\end{split}
\end{equation}
and
\begin{equation}
\label{est-dk}
\begin{split}
|d_k|
\leq C_L\Bigl\{&
\varepsilon^2|u||X_kp_\varepsilon|
+\varepsilon^2p_\varepsilon|X_ku|
+\tau|u||X_ku|\\
&+\tau\varepsilon^2u^2p_\varepsilon|X_kp_\varepsilon|\\
&+\tau^2u^2
\bigl(
|u||X_kp_\varepsilon|
+p_\varepsilon|X_ku|
\bigr)\\
&+\tau^2rp_\varepsilon
\bigl(
|u||X_kr|+r|X_ku|
\bigr)(1+|u|)\\
&+\tau^2r^2|u|
\bigl(1+\tau^2p_\varepsilon^2|u|\bigr)
\bigl(
|u||X_kp_\varepsilon|
+p_\varepsilon|X_ku|
\bigr)\\
&+\tau^2r^2|u||X_kp_\varepsilon|
\Bigr\}\\
\leq C_L\Bigl\{&
\varepsilon^2|u||X_kp_\varepsilon|
+\varepsilon|X_ku|
+\tau|u||X_ku|\\
&+\varepsilon^2|u||X_kp_\varepsilon|
+\tau^2|u|^3|X_kp_\varepsilon|
+\tau|u||X_ku|\\
&+\tau
\bigl(
|u||X_kr|+r|X_ku|
\bigr)(1+|u|)\\
&+\tau^2r|u|
\bigl(
|u||X_kp_\varepsilon|
+p_\varepsilon|X_ku|
\bigr)\\
&+\tau^2r^2|u||X_kp_\varepsilon|
\Bigr\}\\
\leq C_L\Bigl\{&
(\varepsilon^2+\tau^2r^2)
\bigl(
|u||X_kp_\varepsilon|
+p_\varepsilon|X_ku|
\bigr)+\tau
\bigl(
|u||X_kr|+r|X_ku|
\bigr)
\Bigr\}.
\end{split}
\end{equation}
The estimates in \eqref{ZED} now follow from
\eqref{est-vk} and \eqref{est-dk}, together with
\eqref{STILO1} and \eqref{STILO}.
\end{proof}

At the point
$
\Psi_\varepsilon^{su}
 =\Psi_\varepsilon^{su}(p)\in\Se^{su},
$
the vector fields $X_k^{su}$ have the following representations in the
coordinates $[\cdots]_*$:
\begin{equation}
\label{STELLA}
\begin{split}
X_k^{su}(\Psi_\varepsilon^{su})
&=
\bigl(
w_k+v_k,\,
\varepsilon c_k+\tau\mathop{\rm Im}(z\bar w_k)+d_k
\bigr)_*\\
&=
\left[
w_k+v_k,\,
c_k+\frac{d_k}{\varepsilon}
+\frac{\tau}{\varepsilon}
\mathop{\rm Im}\left(
z\bar w_k+\bar\Psi_{\varepsilon,1}^{su}(w_k+v_k)
\right)
\right]_*\\
&=
[w_k,c_k]_*+[v_k,e_k]_*,
\end{split}
\end{equation}
where
\begin{equation}
\label{STOP}
e_k
 =
\frac{d_k}{\varepsilon}
+\frac{\tau}{\varepsilon}
\mathop{\rm Im}\left(
z\bar w_k+\bar\Psi_{\varepsilon,1}^{su}(w_k+v_k)
\right).
\end{equation}

\begin{remark}
It follows from \eqref{Psie1} that
\begin{equation}
\label{STAP}
\begin{split}
\mathop{\rm Im}\left(
z\bar w_k+\bar\Psi_{\varepsilon,1}^{su}(w_k+v_k)
\right)
&=
\mathop{\rm Im}\left(
z\bar w_k
+\bar z
\bigl(su(\tau p_\eps-i)\overline{\beta_1}+1\bigr)
(w_k+v_k)
\right)\\
&=
\mathop{\rm Im}\left(
\bar zv_k
+\bar zsu(\tau p_\eps-i)\overline{\beta_1}
(w_k+v_k)
\right).
\end{split}
\end{equation}
Consequently,
\begin{equation}
\label{bix}
\begin{split}
\left|
\mathop{\rm Im}\left(
z\bar w_k+\bar\Psi_{\varepsilon,1}^{su}(w_k+v_k)
\right)
\right|
\leq{}&
r|v_k|
+
r|u||\tau p_\eps-i||\beta_1|
\bigl(|w_k|+|v_k|\bigr).
\end{split}
\end{equation}
It now follows from \eqref{ZED} that
\[
\begin{split}
\left|
\mathop{\rm Im}\left(
z\bar w_1+\bar\Psi_{\varepsilon,1}^{su}(w_1+v_1)
\right)
\right|
&\leq C_Lr|u|_{1,w},\\
\left|
\mathop{\rm Im}\left(
z\bar w_2+\bar\Psi_{\varepsilon,1}^{su}(w_2+v_2)
\right)
\right|
&\leq\er C_Lr|u|_{1,w}.
\end{split}
\]
We therefore obtain the following estimates for $e_1$ and $e_2$:
\begin{equation}
\label{MORON}
|e_1|
 \leq\frac{C_L}{\er}|u|_{1,w},
\qquad
|e_2|
 \leq C_L|u|_{1,w}.
\end{equation}
\end{remark}

In the next step, we compare $\Nsu$ with $\widehat N_\eps$.

\begin{lemma}
\label{lello3} 
Let  $u:\hh^1\to\rr$ be a function such   that $|u|_{1,w} \leq L $,  for some $L>0$.
Then there exists $C_L>0$ such that, for every
$p\in\Se\setminus Z$, $s\in[0,1]$, and $\varepsilon>0$,
\begin{equation}
\label{EQUA1bis}
\big\|
\Nsu\bigl(\Psi_\varepsilon^{su}(p)\bigr)
-\widehat N_\eps\bigl(\Psi_\varepsilon^{su}(p)\bigr)
\big\|
\leq C_L|u(p)|_{1,w},
\end{equation}
and
\begin{equation}
\label{EQUA2bis}
\big|
\big\langle
\Nsu\bigl(\Psi_\varepsilon^{su}(p)\bigr)
-\widehat N_\eps\bigl(\Psi_\varepsilon^{su}(p)\bigr),
\Te
\big\rangle
\big|
\leq C_L\er|u(p)|_{1,w}.
\end{equation}
\end{lemma}

\begin{proof}
In the notation introduced in \eqref{STELLA}, we have
\begin{equation}
\label{STALLA}
\begin{split}
X_1^{su}\wedge X_2^{su}
&=
[w_1+v_1,e_1]_*
 \wedge[w_2+v_2,c_2+e_2]_*\\
&=
[w_1,0]_*\wedge[w_2,c_2]_*
+[w_1,c_1]_*\wedge[v_2,e_2]_*\\
&\quad
+[v_1,e_1]_*\wedge[w_2,c_2]_*
+[v_1,e_1]_*\wedge[v_2,e_2]_*.
\end{split}
\end{equation}
At every point $p\in\Se$,
\[
X_1\wedge X_2
 =[w_1,0]_*\wedge[w_2,c_2]_*
 =[z(1-i\tau p_\varepsilon),\varepsilon p_\varepsilon]_*
 =\Ne.
\]
Therefore, by the definition of $\widehat N_\eps$, at
$\Psi_\varepsilon^{su}(p)$ we have
\[
[w_1,0]_*\wedge[w_2,c_2]_*
 =\widehat N_\eps.
\]
Consequently,
\[
X_1^{su}\wedge X_2^{su}
 =\widehat N_\eps+R_\varepsilon,
\]
where
\begin{equation}
\label{REO}
\begin{split}
R_\varepsilon
={}&
[w_1,c_1]_*\wedge[v_2,e_2]_*
+[v_1,e_1]_*\wedge[w_2,c_2]_*
+[v_1,e_1]_*\wedge[v_2,e_2]_*.
\end{split}
\end{equation}

The wedge products can be computed using \eqref{WED}. To estimate the
resulting terms, we observe that \eqref{ZED} and \eqref{MORON} give
\[
\begin{split}
&|\mathop{\rm Im}(\bar w_1v_2)|
+|\mathop{\rm Im}(\bar v_1w_2)|
+|\mathop{\rm Im}(\bar v_1v_2)| \leq C_L\er|u|_{1,w},
\end{split}
\]
and
\[
|e_2v_1|
+|e_1w_2-c_2v_1|
+|e_1v_2-e_2v_1|
\leq C_L|u|_{1,w}.
\]
It follows that
\[
\|R_\varepsilon\|
 \leq C_L|u|_{1,w},
\qquad
|\langle R_\varepsilon,\Te\rangle|
 \leq C_L\er|u|_{1,w}.
\]
Finally, 
normalising $X_1^{su}\wedge X_2^{su}$   yields \eqref{EQUA1bis} and
\eqref{EQUA2bis}.
\end{proof}

Combining Lemmas \ref{lello1}--\ref{lello3}, we prove Theorem \ref{Socrates}.

 \section{Estimates for the derivatives of $\Nsu$}
\label{TEK2}

In the second variation of area, the term $\|\hsu\|^2$ involves
derivatives of the normal $\Nsu$. We therefore need precise estimates
for these derivatives. For the tangent space $T\Se^{su}$, we use the
basis
\[
X_1=X_1^{su},
\qquad
X_2=X_2^{su},
\]
introduced in \eqref{triang}. The outer unit normal $\Nsu$ to
$\Se^{su}$ is therefore
\begin{equation}
\label{view}
\Nsu
 =\frac{X_1\wedge X_2}{\|X_1\wedge X_2\|}.
\end{equation}
When needed, we retain the superscript $su$, as in
\[
X_1^{su}\wedge X_2^{su}
 =(X_1\wedge X_2)(\Psisu),
\]
to emphasise the composition of $X_1\wedge X_2$ with $\Psisu$.

We use the following notation for the $[\ldots]_*$ coordinates:
\[
\Nsu=[\Nsua,\Nsub,\Nsuc]_*,
\qquad
\widehat N_\varepsilon=[\Nea,\Neb,\Nec]_*.
\]
Henceforth, as in \eqref{nullo}, $\widehat N_\varepsilon$ denotes the
vector field which, at $\Psi_\varepsilon^{su}(p)$, has the same
$[\cdots]_*$ coordinates as $\Ne$ at $p\in\Se$.

The goal of this section is to prove the following theorem.

\begin{theorem}
\label{derivnormal}
Let $u:\hh^1\to\rr$ be a continuous function that is piecewise of class $C^2$ and satisfies $
|u|_{1,w}\leq L$,  for some $L>0$. 
Then there exists $C_L>0$ such that, for every
$p\in\Se\setminus Z$, $s\in[0,1]$, and $\eps>0$,
\begin{equation}
\label{Da1}
\begin{split}
|X_i\Nsuj-X_i\Nej|
&\leq
\frac{C_L}{r}\er^{i-1}|u(p)|_{2,w},
\qquad i,j=1,2,\\
|X_i\Nsuc-X_i\Nec|
&\leq
\frac{C_L}{r}\er^i|u(p)|_{2,w},
\qquad i=1,2.
\end{split}
\end{equation}
The quantities on the left-hand sides are evaluated at
$\Psisu(p)\in\Se^{su}$.
\end{theorem}

In view of \eqref{view}, we first compute the derivatives of the
$[\ldots]_*$ coordinates of $X_1$ and $X_2$, which, by
\eqref{STELLA}, are given by
\[
X_k\bigl(\Psisu(p)\bigr)
 =(\Psisu)_*X_k(p)
 =[w_k,c_k]_*+[v_k,e_k]_*,
\qquad k=1,2.
\]
Here, $w_k$ and $c_k$ are given in \eqref{NOTA1bis}, while
$v_k,d_k$, and $e_k$ are given in \eqref{LOP} and \eqref{STOP}.
In particular,
\begin{align*}
v_k
&=
sw_ku(\tau p_\varepsilon+i)\beta_1
+sz\Bigl[
i(X_ku)\beta_1
+\tau X_k(up_\varepsilon)
\bigl(
\beta_1+2su(\tau p_\varepsilon+i)\beta_1'
\bigr)
\Bigr],\\
e_k
&=
\frac{d_k}{\eps}
+\frac{\tau}{\eps}
\operatorname{Im}\left(
z\overline{w_k}
+\overline{\Psi_{\varepsilon,1}^{su}}(w_k+v_k)
\right).
\end{align*}

A direct computation gives, for $i,k=1,2$,  
\begin{align}
\label{EL1}
|X_iw_k|
&\leq
\frac{C}{r}\er^{i+k-2},\\
\label{EL1bis}
|X_1c_2|
&\leq
\frac{\tau}{\eps}\er^3,
\qquad
|X_2c_2|
\leq\er^3.
\end{align}
The estimates for $v_k$ and $e_k$ are as follows.

\begin{lemma}
\label{derivX_i}
Let $u:\hh^1\to\rr$ be a continuous function that is piecewise of
class $C^2$ and satisfies $
|u|_{1,w}\leq L$,  
for some $L>0$. 
Then there exists a constant $C_L>0$ such that, for every
$p\in\Se\setminus Z$, $s\in[0,1]$, $\eps>0$, and $i,k=1,2$,
\begin{align}
\label{EL1.2}
|X_iv_k|
&\leq
\frac{C_L}{\eps}\er^{i+k-1}|u|_{2,w},\\
\label{EL1.3}
|X_ie_k|
&\leq
\frac{C_L}{\eps}\er^{i+k-2}|u|_{2,w}
+C_L\eps^{i+k-2}|\nabla^2u|.
\end{align}
\end{lemma}

The proof of this lemma is given in Appendix 2. We now proceed to the
main estimates.

\begin{proof}[Proof of Theorem \ref{derivnormal}]
By the definition of $X_i=X_i^{su}$ in \eqref{triang}, for
$i,j=1,2$ we have
\begin{equation}
\label{NOMOS}
\begin{split}
\left.X_i^{su}(\Nsuj)\right|_{\Psisu}
&=
X_i\bigl(\Nsuj(\Psisu)\bigr)\\
&=
X_i\left(
\frac{(X_1^{su}\wedge X_2^{su})_j}
{\|X_1^{su}\wedge X_2^{su}\|}
\right)\\
&=
\frac{
X_i(X_1^{su}\wedge X_2^{su})_j
-\Nsuj
\left\langle
\Nsu,D_{X_i}(X_1^{su}\wedge X_2^{su})
\right\rangle
}
{\|X_1^{su}\wedge X_2^{su}\|}.
\end{split}
\end{equation}
Here, $D_{X_i}(X_1^{su}\wedge X_2^{su})$ is defined as in
\eqref{VU}. As in \eqref{REO},
\[
X_1^{su}\wedge X_2^{su}
 =\widehat N_\eps+R_\eps,
\]
where
\[
\begin{split}
R_\varepsilon
={}&
[w_1,0]_*\wedge[v_2,e_2]_*
+[v_1,e_1]_*\wedge[w_2,c_2]_*+[v_1,e_1]_*\wedge[v_2,e_2]_*.
\end{split}
\]
Using the bilinearity of the wedge product and formula \eqref{WED}, we
obtain, for $j=1,2$,
\begin{equation}
\label{MAI0}
\begin{split}
\left|
X_i\bigl((X_1^{su}\wedge X_2^{su})_j-\Nej\bigr)
\right|
\leq{}&
|X_iw_1||e_2|
+|X_iv_1|(|c_2|+|e_2|)\\
&+|X_ie_2|(|w_1|+|v_1|)
+|v_1||X_ic_2|\\
&+|X_iw_2||e_1|
+|X_iv_2||e_1|\\
&+|X_ie_1|(|w_2|+|v_2|).
\end{split}
\end{equation}
Hence, by \eqref{EL1}--\eqref{EL1.3}, for $j=1,2$,
\begin{equation}
\label{MAI}
\begin{split}
\left|
X_i\bigl((X_1^{su}\wedge X_2^{su})_j-\Nej\bigr)
\right|
&\leq
\frac{C_L}{r}\er^{i-1}|u|_{1,w}
+\frac{C_L}{\sqrt{\eps^2+\tau^2r^2}}
 \er^{i-1}|u|_{2,w}\\
&\quad
+\frac{C_L}{r}\er^{i-1}|u|_{2,w}
+C_L\eps^{i-1}|\nabla^2u|\\
&\quad
+\frac{C_L\eps^{i-1}}
{\sqrt{\eps^2+\tau^2r^2}}|u|_{1,w}
+\frac{C_L}{r}\er^{i-1}|u|_{1,w}\\
&\quad
+\frac{C_L\eps}{\eps^2+\tau^2r^2}
 \er^{i-2}|u|_{2,w}
+\frac{C_L}{\eps}\er^i|u|_{2,w}\\
&\leq
\frac{C_L}{r}\er^{i-1}|u|_{2,w}.
\end{split}
\end{equation}

For $j=3$, on the other hand,
\begin{equation}
\label{MIA}
\begin{split}
\left|
X_i\bigl((X_1^{su}\wedge X_2^{su})_3-\Nec\bigr)
\right|
&\leq
2\Bigl(
|X_iw_1||v_2|
+|X_iv_1|(|w_2|+|v_2|)\\
&\qquad
+|X_iw_2||v_1|
+|X_iv_2|(|w_1|+|v_1|)
\Bigr)\\
&\leq
\frac{C_L}{r}\er^i|u|_{1,w}
+\frac{C_L}{\sqrt{\eps^2+\tau^2r^2}}
 \er^i|u|_{2,w}\\
&\qquad
+\frac{C_L}{r}\er^i|u|_{1,w}
+\frac{C_L\eps}{\eps^2+\tau^2r^2}
 \er^{i-1}|u|_{2,w}\\
&\leq
\frac{C_L}{r}\er^i|u|_{2,w}.
\end{split}
\end{equation}

Finally, setting $
\al=\|X_1^{su}\wedge X_2^{su}\|$, it follows from \eqref{NOMOS} that
\[
\begin{split}
\al|X_i\Nsuj-X_i\Nej|
\leq{}&
\left|
X_i(X_1^{su}\wedge X_2^{su})_j-X_i\Nej
\right|
+|1-\al||X_i\Nej|\\
&+
|\Nsuj-\Nej|
\|D_{X_i}(X_1^{su}\wedge X_2^{su})\|\\
&+
|\Nej|\,
\|\Nsu-\widehat N_\eps\|
\|D_{X_i}(X_1^{su}\wedge X_2^{su})\|\\
&+
|\Nej|
\big\|
D_{X_i}(X_1^{su}\wedge X_2^{su})
-D_{X_i}\widehat N_\eps
\big\|\\
&+
\big|
\big\langle
\widehat N_\eps,D_{X_i}\widehat N_\eps
\big\rangle
\big|.
\end{split}
\]
The last term vanishes. Indeed, by \eqref{VU} and \eqref{VU2},
\[
0
 =
X_i\langle\widehat N_\eps,\widehat N_\eps\rangle
 =
2\big\langle
\widehat N_\eps,D_{X_i}\widehat N_\eps
\big\rangle.
\]

Using \eqref{EQUA1bis} and \eqref{EQUA2bis} from Lemma
\ref{lello3}, together with \eqref{MAI}, \eqref{MIA}, and
\begin{equation}
\label{misura}
|1-\al|\leq C_L|u|_{1,w},
\end{equation}
we obtain the estimates in \eqref{Da1}.
\end{proof}

\begin{lemma}
\label{derivNeBIS}
There exists a constant $C>0$ such that, for every
$p\in\Se\setminus Z$, $0<\eps\leq1$, and $i,j=1,2$,
\begin{equation}
\label{EL23}
\left|
\langle D_{X_i}\Ne,X_j\rangle
\right|
\leq
\frac{C}{r}\er^{i+j-2}.
\end{equation}
\end{lemma}

\begin{proof}
We use formula \eqref{nux} for the $[\ldots]_*$ coordinates of $\Ne$.
The derivatives of its horizontal component satisfy
\[
\left|
X_i\bigl(z(1-\tau p_\varepsilon i)\bigr)
\right|
\leq
|w_i||1-\tau p_\varepsilon i|
+\tau r|X_ip_\eps|.
\]
Using
$
|w_i|\leq\er^{i-1}
$
and \eqref{EL2}, we obtain
\[
\begin{split}
\left|
X_i\bigl(z(1-\tau p_\varepsilon i)\bigr)
\right|
&\leq
\er^{i-1}
\frac{\sqrt{\eps^2+\tau^2}}
{\sqrt{\eps^2+\tau^2r^2}}
+C\tau r\eps^{i-3}\er^2\\
&\leq
C\frac{\er^{i-1}}
{\sqrt{\eps^2+\tau^2r^2}}.
\end{split}
\]
For the third coordinate, $N_{\eps,3}=\eps p_\eps$, \eqref{EL2} gives
$
|X_i(\eps p_\eps)|
\leq C\eps^{i-2}\er^2.
$

Consequently,
\[
\begin{split}
\left|
\langle D_{X_i}\Ne,X_j\rangle
\right|
&\leq
C\frac{\er^{i-1}}
{\sqrt{\eps^2+\tau^2r^2}}|w_j|
+C\eps^{i-2}\er^2|c_j|\\
&\leq
C\frac{\er^{i+j-2}}
{\sqrt{\eps^2+\tau^2r^2}}
+C\eps^{i-2}\er^2(j-1)\\
&\leq
C\frac{\er^{i+j-2}}{r}.
\end{split}
\]
Here we have used $c_1=0$, $|c_2|\leq 1$, and $\sqrt{\eps^2+\tau^2r^2}\geq\tau r$.
\end{proof}

\begin{remark} With estimates similar to the ones in Lemma \ref{derivNeBIS}, but at the point $\Psi_{\eps}^{su}$, we also have for $i,j=1,2$,
\begin{align}
    \label{EL21}
    | X_i^{su}  N_{\eps,j} (\Psi_{\eps}^{su})| \leq \frac{C_L}{\eps} \er^ i   \quad and  
    \quad 
        | X_i^{su}  \Nej (\Psi_{\eps}^{su})| \leq \frac{ C}{\eps} \er^ i   ;
        \end{align} 
        \begin{align}
    \label{EL22}
    | X_i^{su}  N_{\eps,3} (\Psi_{\eps}^{su})| \leq C \eps^{i-2} \er^ 2 \quad and  
    \quad 
        | X_i^{su}  \Nec (\Psi_{\eps}^{su})| \leq  C \eps^{i-2} \er^ 2 .
        \end{align}
\end{remark}

\section{Uniform continuity of the second variation}
\label{UC}

In this section, we prove the key estimate for the second variation of
the isoperimetric functional $\Fe$ defined in \eqref{IF}, namely the
continuity estimate \eqref{FOP}, uniformly in $\eps>0$.

Given a function $u:\Se\to\rr$, set
\[
\Fep(s)=\Fep(s;u)=\Fe(su).
\]
It follows from \eqref{svareas} and \eqref{VV2} that, for every
$s\in[0,1]$,
\[
\begin{split}
\Fep''(s)
={}&
\int_{\Se^{su}}
\left\{
\|\nabla_{\Se^{su}}\bar u\|^2
+\left((\Hsu)^2-Q_\varepsilon^{su}\right)\bar u^2
\right\}\,d\Ae\\
&-
H\int_{\Se^{su}}
\left\{ \bar u \Nsu(\bar u) 
+\bar u^2\Hsu
\right\}\,d\Ae,
\end{split}
\]
where
$
\bar u=\langle\Nsu,\Ne\rangle u$.

Our goal is to prove the following theorem.

\begin{theorem}
\label{unifcontSecondVar} 
For every $\eta\in(0,1)$, $\gamma>0$,  and $L>0$, there
exist constants
$
\delta=\delta(\eta,\gamma,L)>0$
 and $\eps_0=\eps_0(\eta,\gamma, L)>0
$ 
with the following property. Let $u:\So\to\rr$ be a continuous,
piecewise $C^2$ function satisfying the following conditions for some
$\varrho\geq0$$:$
\begin{enumerate}
\item[(1)]
$u(p)\in\{t^-(p),t^+(p)\}$ for every $p=(z,t)\in\So$ with
$|z|\leq\varrho;$

\item[(2)]
$t^-(p)<u(p)<t^+(p)$ for every $p=(z,t)\in\So$ with
$|z|>\varrho;$

\item[(3)]
$\varrho\leq\delta$  and $|u|_{2,w}\leq\delta$ when
$|z|>\varrho;$

\item[(4)]
$\displaystyle 
u(p)^2
 \leq L|z|^\gamma \int_{\So}u^2\,d\Ao$  for every $p=(z,t)\in\So;$

\item[(5)] 
$\displaystyle  |\nabla u(p)|^2
 \leq L\int_{\So}\|\nablaSo u\|^2\,d\Ao$  for   $p=(z,t)\in\So$ with $|z|> \varrho.$

\end{enumerate}
Then, for every $0<\varepsilon\leq\eps_0$ and $s\in[0,1]$,
\begin{equation}
\label{FEP}
\left|\Fep''(s;u)-\Fep''(0;u)\right|
\leq
\eta\int_{\Se}
\left\{
\|\nablaSe u\|^2+u^2
\right\}\,d\Ae.
\end{equation}
\end{theorem}

Above, the function $u:\So\to\rr$ is lifted to $\Se$ by keeping it
constant for fixed $z$. We decompose
\[
\Fep''(s;u)-\Fep''(0;u)
 =E_1(s)+E_2(s)-E_3(s)-HE_4(s) +E^{su}_\eps - H \widetilde E^{su}_\eps,
\]
where
\begin{align*}
E_1(s)
&=
\int_{\Se^{su}}
\|\nabla_{\Se^{su}}\bar u\|^2\,d\Ae
-\int_{\Se}\|\nablaSe u\|^2\,d\Ae,\\
E_2(s)
&=
\int_{\Se^{su}}
\bar u^2\Hsu(\Hsu-H)\,d\Ae,\\
E_3(s)
&=
\int_{\Se^{su}}Q_\varepsilon^{su}\bar u^2\,d\Ae
-\int_{\Se}Q_\varepsilon u^2\,d\Ae,\\
E_4(s)
&=
\int_{\Se^{su}} \bar u \Nsu(\bar u) \,d\Ae,
\end{align*}
and $E^{su}_\eps$, $\widetilde E^{su}_\eps$ are as in \eqref{pix} and \eqref{pox}, respectively.
We estimate these terms separately.

We first introduce notation for the metric induced by $\gep$ on
$\Se^{su}$. Let
$
g_{ij}=\langle X_i^{su},X_j^{su}\rangle
$
be its components at $\Psi_\eps^{su}$, and let $g^{ij}$ be the components of the inverse
matrix. We claim that, for $i,j=1,2$, 
\begin{equation}
\label{mela}
g^{ij}=\delta^{ij}+r^{ij},
\qquad
|r^{ij}|
 \leq C_L\er^{2-i-j}|u|_{1,w}.
\end{equation}

We now verify \eqref{mela}. By \eqref{STELLA}, together with
\eqref{NOTA}, \eqref{ZED}, \eqref{MORON}, and \eqref{misura}, we have
$
\det(g_{ij})
 =\|X_1^{su}\wedge X_2^{su}\|^2
 =(\al)^2$.
Consequently,
\begin{equation}
\label{gigi}
\begin{split}
g^{11}
&=
\frac{g_{22}}{\det(g_{ij})}
 =
\frac{\langle X_2^{su},X_2^{su}\rangle}{(\al)^2}\\
&=
\frac{|w_2+v_2|^2+|c_2+e_2|^2}{(\al)^2}
 =1+r^{11},
\qquad
|r^{11}|\leq C_L|u|_{1,w},\\
g^{12}=g^{21}
&=
-\frac{g_{12}}{\det(g_{ij})}\\
&=
-\frac{
\mathop{\rm Re}
\bigl(w_1\bar v_2+v_1\bar w_2+v_1\bar v_2\bigr)
+e_1(c_2+e_2)
}{(\al)^2}\\
&=r^{12},
\hspace{5.2cm}
\qquad
|r^{12}|
 \leq\frac{C_L}{\er}|u|_{1,w},\\
g^{22}
&=
\frac{g_{11}}{\det(g_{ij})}
 =
\frac{\langle X_1^{su},X_1^{su}\rangle}{(\al)^2}\\
&=
\frac{|w_1+v_1|^2+e_1^2}{(\al)^2}
 =1+r^{22},
\qquad
\hspace{1.3cm}
|r^{22}|
 \leq\frac{C_L}{\er^2}|u|_{1,w}.
\end{split}
\end{equation}
In the second identity, we used
$
\mathop{\rm Re}(w_1\bar w_2)=0$.

\subsection{Estimate for the energy}

In this subsection, we prove the uniform continuity of the energy.

\begin{theorem}
\label{unifcontgradient}
Under the assumptions of Theorem \ref{unifcontSecondVar}, for every
$0<\varepsilon\leq\eps_0$ and $s\in[0,1]$,
\begin{equation}
\label{SE1} \int_{\Se^{su}}
\|\nabla_{\Se^{su}}\bar u\|^2\,d\Ae
-\int_{\Se}\|\nablaSe u\|^2\,d\Ae \leq
\eta\int_{\Se}
\left(
\|\nablaSe u\|^2+u^2
\right)\,d\Ae.
\end{equation}
\end{theorem}

\begin{proof}
Since $|\langle\Nsu,\Ne\rangle|\leq1$, we obtain
\begin{equation}
\label{fico}
\begin{split}
\|\nabla_{\Se^{su}}\bar u\|^2
\leq
\|\nabla_{\Se^{su}}u\|^2
&+
2|u|\,
\|\nabla_{\Se^{su}}\langle\Nsu,\Ne\rangle\|
\|\nabla_{\Se^{su}}u\|\\
&+
u^2
\|\nabla_{\Se^{su}}\langle\Nsu,\Ne\rangle\|^2.
\end{split}
\end{equation}

By the metric compatibility property \eqref{VU2}, at
$\Psi_\eps^{su}(p)$ we have
\[
\begin{split}
X_i\langle\Nsu,\Ne\rangle
&=
\langle D_{X_i}\Nsu,\Ne\rangle
+\langle\Nsu,D_{X_i}\Ne\rangle\\
&=
\langle D_{X_i}\Nsu-D_{X_i}\widehat N_\eps,\Ne\rangle\\
&\quad+
\langle D_{X_i}\widehat N_\eps,
        \Ne-\widehat N_\eps\rangle
+\langle\Nsu-\Ne,D_{X_i}\Ne\rangle.
\end{split}
\]
Notice that the vertical component of
$\Ne-\widehat N_\eps$ vanishes. It follows from \eqref{Da1},
\eqref{EL21}, \eqref{EQUA1ter}, \eqref{EQUA1}, and \eqref{EL22} that
\begin{equation}
\label{PUNK1}
\begin{split}
|X_i\langle\Nsu,\Ne\rangle|
\leq{}&
\|D_{X_i}\Nsu-D_{X_i}\widehat N_\eps\|\\
&+
\big|
\big\langle
D_{X_i}\widehat N_\eps,
\Ne-\widehat N_\eps
\big\rangle
\big|\\
&+
\big|
\big\langle
\Nsu-\Ne,D_{X_i}\Ne
\big\rangle
\big|\\
\leq{}&
\frac{C_L}{r}\er^{i-1}|u|_{2,w}
+C_L
\frac{\er^{i-1}}{\sqrt{\eps^2+\tau^2r^2}}
\tau p_\eps|u|\\
&+
C_L|u|_{1,w}
\frac{\er^{i-1}}{\sqrt{\eps^2+\tau^2r^2}}\\
&+
C_L\er|u|_{1,w}
\frac{\eps^i}{\eps^2+\tau^2r^2}\\
\leq{}&
\frac{C_L}{r}\er^{i-1}|u|_{2,w}.
\end{split}
\end{equation}
Combining \eqref{mela} and \eqref{PUNK1}, we obtain
\begin{equation}
\label{mora}
\begin{split}
\|\nabla_{\Se^{su}}\langle\Nsu,\Ne\rangle\|^2
&\leq
\frac{C_L}{r^2}
\sum_{i,j=1}^2
\er^{2-i-j}\er^{i-1}\er^{j-1}
|u|_{2,w}^2\\
&\leq
\frac{C_L}{r^2}|u|_{2,w}^2.
\end{split}
\end{equation}

Next, we estimate the difference
$
\|\nabla_{\Se^{su}}u(\Psi_\eps^{su}(p))\|^2
-\|\nablaSe u(p)\|^2.
$
Using
$
X_i^{su}u(\Psi_\eps^{su}(p))=X_iu(p),
$
together with \eqref{mela} and \eqref{STILO1}, we obtain
\begin{equation}
\label{pera}
\begin{split}
\|\nabla_{\Se^{su}}u(\Psi_\eps^{su}(p))\|^2
&=
\sum_{i,j=1}^2
g^{ij}X_i^{su}u\,X_j^{su}u\\
&=
\|\nablaSe u(p)\|^2
+\sum_{i,j=1}^2
(g^{ij}-\delta^{ij})X_iu(p)X_ju(p)\\
&\leq
\|\nablaSe u(p)\|^2
+C_L|u|_{1,w}
\sum_{i,j=1}^2
\er^{2-i-j}\er^{i-1}\er^{j-1}|\nabla u(p)|^2\\
&\leq
\|\nablaSe u(p)\|^2
+C_L|u|_{1,w}|\nabla u(p)|^2.
\end{split}
\end{equation}

Combining \eqref{fico}, \eqref{mora}, and \eqref{pera}, we obtain
\begin{equation}
\label{lampone}
\begin{split}
\|\nabla_{\Se^{su}}\bar u(\Psi_\eps^{su}(p))\|^2
\leq{}&
\|\nablaSe u(p)\|^2
+C_L|u|_{1,w}|\nabla u(p)|^2\\
&+
\frac{C_L|u(p)|}{r}|u|_{2,w}
\left(
\|\nablaSe u(p)\|^2
+C_L|u|_{1,w}|\nabla u(p)|^2
\right)^{1/2}\\
&+
\frac{C_Lu(p)^2}{r^2}|u|_{2,w}^2\\
\leq{}&
\|\nablaSe u(p)\|^2
+C_L|u|_{2,w}
\left(
|\nabla u(p)|^2+\frac{u(p)^2}{r^2}
\right).
\end{split}
\end{equation}
In the last inequality, we used $|u|_{2,w}\leq C$.

We make the change of variables
$
q=\Psi_\eps^{su}(p)\in\Se^{su}$,
$p\in\Se$. 
By \eqref{misura},
\begin{equation}
\label{mis2}
\al
 =\|X_1^{su}\wedge X_2^{su}\|
 =1+m^{su},
\qquad
|m^{su}|\leq C_L|u|_{1,w}.
\end{equation}
Therefore,
\begin{equation}
\label{caco}
\begin{split}
\int_{\Se^{su}}
\|\nabla_{\Se^{su}}\bar u\|^2\,d\Ae
&=
\int_{\Se}
\|\nabla_{\Se^{su}}\bar u(\Psi_\eps^{su}(p))\|^2
\al\,d\Ae\\
&\leq
\int_{\Se}\|\nablaSe u\|^2\al\,d\Ae\\
&\quad+
C_L\int_{\Se}
|u|_{2,w}
\left(
|\nabla u|^2+\frac{u^2}{r^2}
\right)\al\,d\Ae\\
&\leq
\int_{\Se}\|\nablaSe u\|^2\,d\Ae\\
&\quad+
C_L\int_{\Se}
|u|_{2,w}
\left(
|\nabla u|^2+\frac{u^2}{r^2}
\right)\,d\Ae.
\end{split}
\end{equation}
In the last inequality, we used \eqref{mis2}, the bounds
$|u|_{1,w}\leq|u|_{2,w}\leq C$, and hence the uniform boundedness of
$\al$.

We estimate the last integral separately on the regions
$|z|\leq\varrho$ and $|z|>\varrho$. Since
$|u|_{2,w}\leq C$ when $|z|\leq\varrho$, it follows from
\eqref{poker-bis} that
\begin{equation}
\label{noce_base}
\begin{split}
C_L\int_{\Se\cap\{|z|\leq\varrho\}}
|u|_{2,w}
\left(
|\nabla u|^2+\frac{u^2}{r^2}
\right)\,d\Ae
&\leq
C_L\Ae\bigl(\Se\cap\{|z|\leq\varrho\}\bigr)\\
&\leq
\frac{C_L\varrho^2}{\eps}
\sqrt{\eps^2+\varrho^2}.
\end{split}
\end{equation}

At a point $p\in\So$ with $|z|=\varrho$, condition (1) and the
behaviour of $t^\pm$ imply that $|u(p)|$ is comparable to $\varrho$.
Using condition (4), we therefore obtain
\begin{equation}
\label{noce_due}
\varrho^2
\leq
\varrho^{2-\gamma}
\leq
C_L\int_{\So}u^2\,d\Ao.
\end{equation}
Consequently, by \eqref{noce_base}, \eqref{noce_due}, and
\eqref{poker},
\begin{equation}
\label{noce}
\begin{split}
C_L\int_{\Se\cap\{|z|\leq\varrho\}}
|u|_{2,w}
\left(
|\nabla u|^2+\frac{u^2}{r^2}
\right)\,d\Ae
&\leq
C_L\Ae\bigl(\Se\cap\{|z|\leq\varrho\}\bigr)\\
&\leq
\frac{LC_L}{\eps}
\sqrt{\eps^2+\varrho^2}
\int_{\So\cap\{|z|>\varrho\}}u^2\,d\Ao\\
&\leq
2LC_L\sqrt{\eps^2+\varrho^2}
\int_{\Se}u^2\,d\Ae\\
&\leq
\frac{\eta}{2}\int_{\Se}u^2\,d\Ae.
\end{split}
\end{equation}
The last inequality holds for $\eps\leq\eps_0$ and
$\varrho\leq\delta$, provided that
\begin{equation}
\label{prugna}
4LC_L\sqrt{\eps_0^2+\delta^2}\leq\eta.
\end{equation}
This gives a first condition on $\eps_0>0$ and $\delta>0$.

We now estimate the integral over the region $|z|\geq\varrho$,
beginning with the energy term. Choose
\begin{equation}
\label{mirto}
0<\delta
\leq
\frac{\eta}{16LC_L\Ao(\So)},
\end{equation}
and suppose that $|u|_{2,w}\leq\delta$ when $|z|>\varrho$. By
condition (5), the comparison of the surface measures, \eqref{energy},
and \eqref{mirto},
\begin{equation}
\label{mandorlo1}
\begin{split}
C_L\int_{\Se\cap\{|z|\geq\varrho\}}
|u|_{2,w}|\nabla u|^2\,d\Ae
&\leq
LC_L\delta\Ae(\Se)
\int_{\So}\|\nablaSo u\|^2\,d\Ao\\
&\leq
\frac{2LC_L\delta}{\eps}\Ao(\So)
\int_{\So}\|\nablaSo u\|^2\,d\Ao\\
&\leq
8LC_L\delta\Ao(\So)
\int_{\Se}\|\nablaSe u\|^2\,d\Ae\\
&\leq
\frac{\eta}{2}
\int_{\Se}\|\nablaSe u\|^2\,d\Ae.
\end{split}
\end{equation}

The term involving $u^2$ is estimated as follows. By condition (4) and
the inequality
\[
\int_{\Se}|z|^{\gamma-2}\,d\Ae
\leq\frac{C_\gamma}{\eps},
\qquad \gamma>0,
\]
we obtain
\begin{equation}
\label{mandorlo2}
\begin{split}
C_L\int_{\Se\cap\{|z|\geq\varrho\}}
|u|_{2,w}\frac{u^2}{r^2}\,d\Ae
&\leq
LC_L\delta
\left(
\int_{\Se}|z|^{\gamma-2}\,d\Ae
\right)
\left(
\int_{\So}u^2\,d\Ao
\right)\\
&\leq
\frac{LC_LC_\gamma\delta}{\eps}
\int_{\So}u^2\,d\Ao\\
&\leq
2LC_LC_\gamma\delta
\int_{\Se}u^2\,d\Ae\\
&\leq
\frac{\eta}{2}\int_{\Se}u^2\,d\Ae,
\end{split}
\end{equation}
provided that
\[
0<\delta\leq\frac{\eta}{4LC_LC_\gamma}.
\]
This gives an additional condition on $\delta$.

The estimate \eqref{SE1} now follows from \eqref{caco},
\eqref{noce}, \eqref{mandorlo1}, and \eqref{mandorlo2}.
\end{proof}

   \subsection{Estimate for the geometric term $Q_\varepsilon^{su}$}

In this subsection, we prove the part of estimate \eqref{FOP}
concerning the term
\[
Q_\varepsilon^{su}
 =\operatorname{Ric}_\eps(\Nsu)+\|\hsu\|^2.
\]

\begin{theorem}
\label{unifcontQeps}
Under the assumptions of Theorem \ref{unifcontSecondVar}, for every
$0<\varepsilon\leq\eps_0$ and $s\in[0,1]$,
\begin{equation}
\label{ortica}
\Big|
\int_{\Se^{su}}Q_\varepsilon^{su}\bar u^2\,d\Ae-
\int_{\Se}Q_\varepsilon u^2\,d\Ae\Big|
\leq
\eta\int_{\Se}u^2\,d\Ae.
\end{equation}
\end{theorem}

We first identify the cancellation of the terms of order $1/\eps^2$
in $Q_\varepsilon^{su}$. This is the content of the following theorem.
Here, $D_{X_i}\Nsu$ denotes the derivative defined in \eqref{VU}.

\begin{theorem}
\label{CANC}
For every $\varepsilon>0$ and $s\in[0,1]$, the following identity
holds on $\Se^{su}$:
\begin{equation}
\label{QUE}
\begin{split}
Q_\varepsilon^{su}
={}&
\sum_{i,j,\ell,k=1}^2
g^{ik}g^{j\ell}
\langle D_{X_i}\Nsu,X_j\rangle
\langle D_{X_\ell}\Nsu,X_k\rangle\\
&-
2\sigma
\frac{
\langle D_{X_1}\Nsu,\overline X_2\rangle
-\langle D_{X_2}\Nsu,\overline X_1\rangle
}
{\|X_1\wedge X_2\|},
\end{split}
\end{equation}
where $X_1,X_2$ form a local frame for $T\Se^{su}$,
$(g^{ij})$ is the inverse of the corresponding metric matrix, and
\[
\Nsu
 =\frac{X_1\wedge X_2}{\|X_1\wedge X_2\|}
\]
is the unit normal to $\Se^{su}$.
\end{theorem}

\begin{proof}
The Ricci curvature in the normal direction is
\begin{equation}
\label{RIO}
\operatorname{Ric}_\eps(\Nsu)
 =\sigma^2
 \left(
 4\langle\Nsu,\Te\rangle^2-2
 \right).
\end{equation}
We now compute the squared norm of the second fundamental form. Using
the summation convention for repeated indices and formula \eqref{LC},
we obtain
\begin{align*}
\|\hsu\|^2
&=
g^{ik}g^{j\ell}
\langle\nabla_{X_i}\Nsu,X_j\rangle
\langle\nabla_{X_\ell}\Nsu,X_k\rangle\\
&=
g^{ik}g^{j\ell}
\left\langle
D_{X_i}\Nsu-\sigma\overline X_i\wedge\Nsu,X_j
\right\rangle
\left\langle
D_{X_\ell}\Nsu-\sigma\overline X_\ell\wedge\Nsu,X_k
\right\rangle\\
&=
g^{ik}g^{j\ell}
\langle D_{X_i}\Nsu,X_j\rangle
\langle D_{X_\ell}\Nsu,X_k\rangle\\
&\quad
-2\sigma g^{ik}g^{j\ell}
\langle D_{X_i}\Nsu,X_j\rangle
\langle\overline X_\ell\wedge\Nsu,X_k\rangle\\
&\quad
+\sigma^2g^{ik}g^{j\ell}
\langle\overline X_i\wedge\Nsu,X_j\rangle
\langle\overline X_\ell\wedge\Nsu,X_k\rangle.
\end{align*}
The last term cancels the Ricci term in \eqref{RIO}. Indeed,
\begin{align*}
&g^{ik}g^{j\ell}
\langle\overline X_i\wedge\Nsu,X_j\rangle
\langle\overline X_\ell\wedge\Nsu,X_k\rangle\\
&\quad=
g^{ik}g^{j\ell}
\langle\overline X_i\wedge\Nsu,X_j\rangle
\langle\overline X_\ell,\Nsu\wedge X_k\rangle\\
&\quad=
-g^{ik}g^{j\ell}
\langle\overline X_i\wedge\Nsu,X_j\rangle
\langle X_\ell,\overline\Nsu\wedge\overline X_k\rangle\\
&\quad=
-g^{ik}
\langle\overline X_i\wedge\Nsu,
       \overline\Nsu\wedge\overline X_k\rangle\\
&\quad=
g^{ik}
\langle\overline X_i\wedge\Nsu,
       \overline X_k\wedge\overline\Nsu\rangle\\
&\quad=
g^{ik}g_{ik}\langle\Nsu,\overline\Nsu\rangle
-g^{ik}
\langle\overline X_i,\Nsu\rangle
\langle\overline X_k,\overline\Nsu\rangle\\
&\quad=
2\left\langle
\Nsu,\Nsu-2\langle\Nsu,\Te\rangle\Te
\right\rangle\\
&\quad=
2-4\langle\Nsu,\Te\rangle^2.
\end{align*}

We next consider the second term in the expression for $\|\hsu\|^2$.
By \eqref{VU2},
\[
\langle D_{X_i}\Nsu,\Nsu\rangle=0.
\]
Therefore,
\begin{align*}
&g^{ik}g^{j\ell}
\langle D_{X_i}\Nsu,X_j\rangle
\langle\overline X_\ell\wedge\Nsu,X_k\rangle\\
&\quad=
g^{ik}g^{j\ell}
\langle D_{X_i}\Nsu,X_j\rangle
\langle X_\ell,\overline{\Nsu\wedge X_k}\rangle\\
&\quad=
g^{ik}
\left\langle
D_{X_i}\Nsu,
\overline{\Nsu\wedge X_k}
-\langle\Nsu,\overline{\Nsu\wedge X_k}\rangle\Nsu
\right\rangle\\
&\quad=
g^{ik}
\langle D_{X_i}\Nsu,\overline{\Nsu\wedge X_k}\rangle.
\end{align*}
Since
\[
\begin{split}
\|X_1\wedge X_2\|\,\Nsu\wedge X_k
&=(X_1\wedge X_2)\wedge X_k\\
&=-\langle X_k,X_2\rangle X_1
  +\langle X_k,X_1\rangle X_2\\
&=g_{1k}X_2-g_{2k}X_1,
\end{split}
\]
we conclude that
\[
\begin{split}
&g^{ik}g^{j\ell}
\langle D_{X_i}\Nsu,X_j\rangle
\langle\overline X_\ell\wedge\Nsu,X_k\rangle\\
&\qquad=
\frac{
\langle D_{X_1}\Nsu,\overline X_2\rangle
-\langle D_{X_2}\Nsu,\overline X_1\rangle
}
{\|X_1\wedge X_2\|}.
\end{split}
\]
This proves \eqref{QUE}.
\end{proof}

In the next step, we compare
$Q_\eps^{su}(\Psi_\eps^{su}(p))$ with $Q_\eps(p)$ for $p\in\Se$.

\begin{theorem}
\label{estimQeps}
Under the assumptions of Theorem \ref{unifcontSecondVar}, for every
$0<\varepsilon\leq\eps_0$, $s\in[0,1]$, and $p\in\Se\setminus Z$,
\begin{equation}
\label{RioRosso}
\big|Q_\eps^{su}\bigl(\Psi_\eps^{su}(p)\bigr)
-
Q_\eps(p)\big |\leq  \frac{C_L}{r^2}|u(p)|_{2,w}.
\end{equation}
\end{theorem}

\begin{proof}
Set
\[
D_{ij}^{su}
 =\langle D_{X_i^{su}}\Nsu,X_j^{su}\rangle,
\qquad
\widehat D_{ij}^{su}
 =\langle D_{X_i^{su}}\widehat N_\eps,X_j^{su}\rangle,
\qquad
D_{ij}
 =\langle D_{X_i}\Ne,X_j\rangle.
\]
By \eqref{Da1}, \eqref{STELLA}, \eqref{NOTA}, \eqref{ZED}, and
\eqref{MORON}, together with $|u|_{1,w}\leq L$, we have
\begin{equation}
\label{pizzo}
\begin{split}
|D_{ij}^{su}-\widehat D_{ij}^{su}|
&\leq
\frac{C_L}{r}
\left(
\er^{i-1}|w_j+v_j|
+\er^i|c_j+e_j|
\right)|u|_{2,w}\\
&\leq
\frac{C_L}{r}
\left(
\er^{i+j-2}
+\er^i|c_j+e_j|
\right)|u|_{2,w}\\
&\leq
\frac{C_L}{r}\er^{i+j-2}|u|_{2,w}.
\end{split}
\end{equation}

The vector fields $D_{X_i^{su}}\widehat N_\eps$ and
$D_{X_i}\Ne$ have the same $[\ldots]_*$ coordinates, while
\[
X_j^{su}=[w_j+v_j,c_j+e_j]_*,
\qquad
X_j=[w_j,c_j]_*.
\]
It follows from \eqref{EL21}, \eqref{EL22}, \eqref{ZED}, and
\eqref{MORON} that
\begin{equation}
\label{pazzo}
\begin{split}
|\widehat D_{ij}^{su}-D_{ij}|
&\leq
\frac{C}{\eps}\er^i|v_j|
+C\eps^{i-2}\er^2|e_j|\\
&\leq
\frac{C_L}{\eps}\er^{i+j-1}|u|_{1,w}
+C_L\eps^{i-2}\er^j|u|_{1,w}\\
&\leq
\frac{C_L}{r}\er^{i+j-2}|u|_{1,w}.
\end{split}
\end{equation}
Combining \eqref{pizzo} and \eqref{pazzo}, we obtain
\begin{equation}
\label{pozzo}
|D_{ij}^{su}-D_{ij}|
\leq
\frac{C_L}{r}\er^{i+j-2}|u|_{2,w}.
\end{equation}

We first estimate the quadratic part of \eqref{QUE}. We use the
expansion
\begin{equation}
\label{pezzo}
\begin{split}
g^{ik}g^{j\ell}D_{ij}^{su}D_{\ell k}^{su}
={}&
g^{ik}g^{j\ell}
(D_{ij}^{su}-D_{ij})(D_{\ell k}^{su}-D_{\ell k})\\
&+
g^{ik}g^{j\ell}D_{ij}(D_{\ell k}^{su}-D_{\ell k})\\
&+
g^{ik}g^{j\ell}(D_{ij}^{su}-D_{ij})D_{\ell k}\\
&+
g^{ik}(g^{j\ell}-\delta^{j\ell})D_{ij}D_{\ell k}\\
&+
(g^{ik}-\delta^{ik})\delta^{j\ell}D_{ij}D_{\ell k}\\
&+
\delta^{ik}\delta^{j\ell}D_{ij}D_{\ell k}.
\end{split}
\end{equation}
It follows from \eqref{pozzo}, \eqref{mela}, and \eqref{EL23} that
\begin{equation}
\label{Q-quadratic}
\begin{split}
\big|g^{ik}g^{j\ell}D_{ij}^{su}D_{\ell k}^{su}
-\delta^{ik}\delta^{j\ell}D_{ij}D_{\ell k}\big| 
&\leq
\frac{C_L}{r^2}|u|_{2,w}
\sum_{i,j,\ell,k=1}^2
\er^{2-i-k}\er^{2-j-\ell}
\er^{i+k-2}\er^{j+\ell-2}\\
& \leq
\frac{C_L}{r^2}|u|_{2,w}.
\end{split}
\end{equation}

We next estimate the second term in \eqref{QUE}. We have
\[
\begin{split}
\langle D_{X_1^{su}}\Nsu,\overline X_2^{su}\rangle
-\langle D_{X_1}\Ne,\overline X_2\rangle
& =
\big\langle
D_{X_1^{su}}\Nsu-D_{X_1^{su}}\widehat N_\eps,
\overline X_2^{su}
\big\rangle\\
&\qquad+
\langle D_{X_1^{su}}\widehat N_\eps,\overline X_2^{su}\rangle
-\langle D_{X_1}\Ne,\overline X_2\rangle.
\end{split}
\]
By \eqref{Da1}, \eqref{NOTA1bis}, \eqref{ZED}, and \eqref{MORON},
\[
\begin{split}
\big|
\big\langle
D_{X_1^{su}}\Nsu-D_{X_1^{su}}\widehat N_\eps,
\overline X_2^{su}
\big\rangle
\big|
&\leq
\frac{C_L}{r}|u|_{2,w}
\left(
|w_2+v_2|
+\eps^{-1}\er|c_2+e_2|
\right)\\
&\leq
\eps\frac{C_L}{r^2}|u|_{2,w}.
\end{split}
\]
As in the proof of \eqref{EL23}, we also have
\[
\big|
\langle D_{X_1^{su}}\widehat N_\eps,\overline X_2^{su}\rangle
-\langle D_{X_1}\Ne,\overline X_2\rangle
\big|
\leq
\eps\frac{C_L}{r^2}|u|_{2,w}.
\]
The same estimates hold with the indices \(1\) and \(2\) interchanged. 
Using also \eqref{misura}, we obtain
\begin{equation}
\label{endlich}
\begin{split}
\Big|
&\sigma
\frac{
\langle D_{X_1^{su}}\Nsu,\overline X_2^{su}\rangle
-\langle D_{X_2^{su}}\Nsu,\overline X_1^{su}\rangle
}
{\|X_1^{su}\wedge X_2^{su}\|}\\
&\quad-
\sigma\left(
\langle D_{X_1}\Ne,\overline X_2\rangle
-\langle D_{X_2}\Ne,\overline X_1\rangle
\right)
\Big|
\leq
\frac{C_L}{r^2}|u|_{2,w}.
\end{split}
\end{equation}
Combining \eqref{Q-quadratic} and \eqref{endlich} proves
\eqref{RioRosso}.
\end{proof}

\begin{proof}[Proof of Theorem \ref{unifcontQeps}]
Making the change of variables
$
q=\Psi_\eps^{su}(p)$,
$p\in\Se$,
and using \eqref{EQUA1bis} and \eqref{mis2}, we obtain
\begin{equation}
\label{tonno}
\begin{split}
\int_{\Se^{su}}Q_\varepsilon^{su}\bar u^2\,d\Ae
&=
\int_{\Se}
Q_\varepsilon^{su}(\Psisu)
u^2
\langle\Ne,\Nsu\rangle^2(\Psisu)
\al\,d\Ae\\
&=
\int_{\Se}
Q_\varepsilon^{su}(\Psisu)u^2(1+\nu^{su})\,d\Ae,
\end{split}
\end{equation}
where
\begin{equation}\label{kikko}
1+\nu^{su}
 =
\langle\Ne,\Nsu\rangle^2(\Psisu)\al,
\qquad
|\nu^{su}|\leq C_L|u|_{1,w}.
\end{equation}
Setting
$
\chi^{su}
 =Q_\varepsilon^{su}(\Psisu)-Q_\varepsilon$, it follows from \eqref{tonno} that
\begin{equation}
\label{tanno}
\begin{split}
\int_{\Se^{su}}Q_\varepsilon^{su}\bar u^2\,d\Ae
&=
\int_{\Se}
(Q_\varepsilon+\chi^{su})u^2(1+\nu^{su})\,d\Ae\\
&=
\int_{\Se}Q_\varepsilon u^2\,d\Ae
+
\int_{\Se}u^2
\left[
Q_\varepsilon\nu^{su}
+(1+\nu^{su})\chi^{su}
\right]\,d\Ae.
\end{split}
\end{equation}

Formula \eqref{QUQQO}, \eqref{kikko}, and  \eqref{RioRosso}   imply that
\[
|Q_\varepsilon\nu^{su}
+(1+\nu^{su})\chi^{su}|
\leq
\frac{C_L}{r^2}|u|_{2,w}.
\]
Consequently,
\[
\Big|
\int_{\Se^{su}}Q_\varepsilon^{su}\bar u^2\,d\Ae
-\int_{\Se}Q_\varepsilon u^2\,d\Ae\Big|
\leq
C_L\int_{\Se}|u|_{2,w}\frac{u^2}{r^2}\,d\Ae.
\]
Splitting the last integral into the regions
$|z|\leq\varrho$ and $|z|>\varrho$, and arguing as in
\eqref{noce} and \eqref{mandorlo2}, we obtain
\[
C_L\int_{\Se}|u|_{2,w}\frac{u^2}{r^2}\,d\Ae
\leq
\eta\int_{\Se}u^2\,d\Ae.
\]
This proves \eqref{ortica}.
\end{proof}

  \subsection{Estimate for $\Hsu$ and $E_\eps^{su}$}

We estimate the terms in the second variation of area involving the mean
curvature, including the error $E_\eps^{su}$ in \eqref{pix} produced by the integration-by-parts as in \eqref{folle}. 
We use the following formula for the mean curvature of
$\Se^{su}$: 
\begin{equation}
\label{acca}
\begin{split}
\Hsu
 =\operatorname{div}_{\Se^{su}}(\Nsu)
&=
\sum_{i,j=1}^2
g^{ij}\langle\nabla_{X_i}\Nsu,X_j\rangle\\
&=
\sum_{i,j=1}^2
g^{ij}\langle D_{X_i}\Nsu,X_j\rangle,
\end{split}
\end{equation}
where $X_1,X_2$ form a local frame for $T\Se^{su}$. The last identity
in \eqref{acca} follows from \eqref{LC} and the identity
\[
\begin{split}
\sum_{i,j=1}^2
g^{ij}\langle\overline{\Nsu}\wedge X_i,X_j\rangle
={}&
g^{12}\langle\overline{\Nsu}\wedge X_1,X_2\rangle
+
g^{21}\langle\overline{\Nsu}\wedge X_2,X_1\rangle
=0,
\end{split}
\]
since the matrix $(g^{ij})$ is symmetric.

\begin{lemma}
\label{estimHeps}
Assume that $|u|_{1,w}\leq L$. Then there exists $C_L>0$ such that,
for every $p\in\Se\setminus Z$, $s\in[0,1]$, and
$0<\eps\leq\eps_0$,
\begin{equation}
\label{hacca}
\big|
\Hsu\bigl(\Psi_\eps^{su}(p)\bigr)-H
\big|
\leq
\frac{C_L}{r}|u(p)|_{2,w},
\end{equation}
where $H=2$ is the mean curvature of $\Se$.
\end{lemma}

\begin{proof}
By \eqref{pozzo}, \eqref{mela}, and \eqref{EL23}, using the summation
convention for repeated indices, we have
\[
\begin{split}
\big|
g^{ij}\langle D_{X_i^{su}}\Nsu,X_j^{su}\rangle
&-\delta^{ij}\langle D_{X_i}\Ne,X_j\rangle
\big|
=
\big|
g^{ij}\bigl(
\langle D_{X_i^{su}}\Nsu,X_j^{su}\rangle
-\langle D_{X_i}\Ne,X_j\rangle
\bigr)\\
& 
\hspace{5cm}+(g^{ij}-\delta^{ij})
\langle D_{X_i}\Ne,X_j\rangle
\big|
\\
& \leq
\frac{C_L}{r}
\er^{2-i-j}\er^{i+j-2}|u|_{2,w}
+\frac{C_L}{r}
\er^{2-i-j}\er^{i+j-2}|u|_{1,w}\\
& \leq
\frac{C_L}{r}|u|_{2,w}.
\end{split}
\]
By \eqref{acca}, the two terms on the first line are respectively
$\Hsu$ and $H$. This proves \eqref{hacca}.
\end{proof}

\begin{theorem}
\label{unifcontHeps}
Under the assumptions of Theorem \ref{unifcontSecondVar}, for every
$0<\varepsilon\leq\eps_0$ and $s\in[0,1]$,
\begin{equation}
\label{hecco}
\Big|
\int_{\Se^{su}}
\bar u^2\Hsu(\Hsu-H)\,d\Ae
\Big|
\leq
\eta\int_{\Se}u^2\,d\Ae.
\end{equation}
\end{theorem}

\begin{proof}
Using \eqref{hacca}, together with the change of variables
$q=\Psi_\eps^{su}(p)$, we obtain
\[
\begin{split}
\Big|
\int_{\Se^{su}}
\bar u^2\Hsu(\Hsu-H)\,d\Ae
\Big|
&=
\Big |
\int_{\Se}
u^2\langle\Ne,\Nsu\rangle^2(\Psisu)
\Hsu(\Psisu)\\
&\qquad\qquad{}\times
\bigl(\Hsu(\Psisu)-H\bigr)
\|X_1^{su}\wedge X_2^{su}\|\,d\Ae
\Big|\\
&\leq
C_L\int_{\Se}
u^2|\Hsu||\Hsu-H|\,d\Ae\\
&\leq
C_L\int_{\Se}
\frac{u^2}{r^2}|u|_{2,w}\,d\Ae.
\end{split}
\]
Here we have also used \eqref{EQUA1bis}, \eqref{mis2}, and
$|u|_{2,w}\leq L$. Estimating the last integral as in \eqref{noce} and
\eqref{mandorlo2}, we obtain  \eqref{hecco}.
\end{proof}

We now estimate $E_\eps^{su}$.
\begin{theorem}
\label{esso}
Under the assumptions of Theorem \ref{unifcontSecondVar}, for every
$0<\varepsilon\leq\eps_0$ and $s\in[0,1]$,
\begin{equation}
\label{hacco}
\big|E_\eps^{su}\big|
\leq
\eta\int_{\Se}u^2\,d\Ae.
\end{equation}
\end{theorem}

\begin{proof}
It is enough to prove the estimate for the term $E_{\eps,+}^{su}$ in $E_\eps^{su}$, defined by formula \eqref{folle}.
The set
\[
   T_\eps^{su} =\{ \Psi_\eps\big((z,t),su(z,t) \big)\in \hh^1: (z,t) \in\Se, |z|=\varrho\}
\] 
is the union of two parallel circles centred on the $Z$-axis, each with radius bounded above by $C\varrho$, where $C>0$ is an absolute constant. An elementary computation gives
\begin{equation}
\label{misso}
\mathscr{L}_\eps(T_\eps^{su})
\leq
\frac{4\pi C\varrho}{\eps}
\sqrt{\eps^2+\tau^2\varrho^2}
\leq
\frac{C\varrho}{\eps}.
\end{equation}
On each of the two components of $T_\eps^{su}$, the function $u$ is constant and its absolute value is comparable to $\varrho$. Therefore, by \eqref{misso}, \eqref{hacca}, the bound $|u|_{2,w}\leq C$, and \eqref{noce_due}, we obtain

 \begin{equation}
\label{picco}
\begin{split}
\big|E_{\eps,+} ^{su} \big| & =
\Big| \int_{T_{\eps} ^{su}}
u ^2 \langle \Ne,\mathscr{N}^{su}_{\eps,+}   \rangle 
\operatorname{div}_{\Se^{su}}(\mathscr{N}^{su}_{\eps,+})
\langle N_\varepsilon^\top, \mathscr{M}^{su}_{\eps,+}\rangle 
 d\mathscr{L}_\eps \Big|
 \\
 &
  \leq C\varrho^2 \max_{T_{\eps} ^{su}}  |H_{\eps,+} ^{su}| \mathscr{L}_\eps(T_{\eps} ^{su})
 \\
 &
  \leq \frac{C_L\varrho^2}{\eps} |u|_{2,w} 
  \\
  &
  \leq \frac{C_L\varrho^\gamma }{\eps}\int_{\So} u^2 d\Ao
  \\
  &
  \leq  2 C_L\varrho^\gamma  \int_{\Se} u^2 d\Ae,
 \end{split}
\end{equation}

The claim follows by choosing $\varrho$ so that
$2C_L\varrho^\gamma\leq\eta/2$.
\end{proof}

\subsection{Estimate for $E_4$ and $\widetilde E^{su}_\eps$}

It remains to estimate the following term arising from the second
variation of volume
\[
E_4(s)
 =
\int_{\Se^{su}}
 \bar u \Nsu(\bar u)\,d\Ae,
\]
and $\widetilde E^{su}_\eps$ as in \eqref{pox}.

\begin{theorem}
\label{unifcontE4}
Under the assumptions of Theorem \ref{unifcontSecondVar}, for every
$0<\varepsilon\leq\eps_0$ and $s\in[0,1]$,
\begin{equation}
\label{last}
\Big|
\int_{\Se^{su}}
 \bar u \Nsu(\bar u)\,d\Ae
\Big|
\leq
\eta\int_{\Se}u^2\,d\Ae.
\end{equation}
\end{theorem}

\begin{proof}
 Since $N_\eps=N_\eps^\top+\langle N_\eps,\Nsu\rangle\Nsu$, We have that 
$$
\bar u \Nsu(\bar u)= u N_\eps (\bar u)-u N_\eps^\top (\bar u).
$$
Let us study the first term. Since $\Ne u=0$ and $\nabla_{\Ne}\Ne=0$, we have $u N_\eps (\bar u)=u^2\langle\Ne,\nabla_{\Ne}\Nsu\rangle$. Using
\[
\nabla_{\Ne}\Nsu
 =D_{\Ne}\Nsu-\sigma\overline{\Ne}\wedge\Nsu
\]
and
$
\overline{\Ne}
 =\Ne-2\langle\Ne,\Te\rangle\Te,
$
we obtain
\[
\begin{split}
\langle\Ne,\nabla_{\Ne}\Nsu\rangle
&=
\langle\Ne,D_{\Ne}\Nsu\rangle
+2\sigma\langle\Ne,\Te\rangle
 \langle\Ne,\Te\wedge\Nsu\rangle\\
&=
\langle\Ne,D_{\Ne}\Nsu\rangle
+2\tau p_\eps\langle\Ne,\Te\wedge\Nsu\rangle.
\end{split}
\]

The second term can be estimated using Theorem \ref{Socrates}:
\begin{equation}
\label{ramino}
\begin{split}
\left|
2\tau p_\eps\langle\Ne,\Te\wedge\Nsu\rangle
\right|
&=
\left|
2\tau p_\eps
\langle\Ne-\Nsu,\Te\wedge\Nsu\rangle
\right|\\
&\leq
\frac{C_L}{r}|u|_{1,w}.
\end{split}
\end{equation}

To estimate the first term, we initially observe that \eqref{EQUA1}
gives
\[
\begin{split}
\left|
u\langle\Ne,D_{\Ne}\Nsu\rangle
\right|
&=
\left|
\langle\Ne-\Nsu,uD_{\Ne}\Nsu\rangle
\right|\\
&\leq
\|\Ne-\Nsu\|\,\|uD_{\Ne}\Nsu\|\\
&\leq
C_L\|uD_{\Ne}\Nsu\|\,|u|_{1,w}.
\end{split}
\]
The action of the vector field $u\Ne$ on $\Nsu$ corresponds to
differentiation with respect to $s$ of the $[\ldots]_*$ coordinates of
$\Nsu$. These coordinates are obtained from \eqref{STALLA}:
\[
X_1^{su}\wedge X_2^{su}
 =
[w_1+v_1,e_1]_*
\wedge[w_2+v_2,c_2+e_2]_*.
\]
Since $w_1,w_2$, and $c_2$ are independent of $s$,
\[
\begin{split}
uD_{\Ne}(X_1^{su}\wedge X_2^{su})
={}&
[v_1',e_1']_*
 \wedge[w_2+v_2,c_2+e_2]_*
+
[w_1+v_1,e_1]_*
 \wedge[v_2',e_2']_*,
\end{split}
\]
where the primes denote differentiation with respect to $s$.

Inspection of \eqref{LOP}, \eqref{STOP}, and \eqref{STAP} shows that
$v_1',v_2',e_1'$, and $e_2'$ satisfy the same estimates as
$v_1,v_2,e_1$, and $e_2$. Here we use
\[
\left|
\frac{\partial}{\partial s}\beta_j(s_\eps u)
\right|
+
\left|
\frac{\partial}{\partial s}\beta_j'(s_\eps u)
\right|
\leq C_L,
\]
which follows from $|\tau p_\eps u|\leq L$. Thus, as in \eqref{ZED}
and \eqref{MORON},
\[
|v_i'|
 \leq C_L\er^{i-1}|u|_{1,w},
\qquad
|e_i'|
 \leq C_L\er^{i-2}|u|_{1,w}.
\]
Using the structure of the wedge product in \eqref{WED}, we obtain
\[
\begin{split}
\|uD_{\Ne}(X_1^{su}\wedge X_2^{su})\|
\leq{}&
|v_1'||c_2+e_2|
+|e_1'||w_2+v_2|
+|v_1'||w_2+v_2|\\
&+
|w_1+v_1||e_2'|
+|e_1||v_2'|
+|w_1+v_1||v_2'|\\
\leq{}&
C_L|u|_{1,w}.
\end{split}
\]
Consequently,
\begin{equation}
\label{scala}
\|uD_{\Ne}\Nsu\|
\leq
\frac{C_L}{\|X_1^{su}\wedge X_2^{su}\|}
|u|_{1,w}.
\end{equation}

It follows from \eqref{ramino} and \eqref{scala} that
\begin{equation}\label{first}
\left|
u\langle\Ne,\nabla_{\Ne}\Nsu\rangle
\right|
\leq
\frac{C_L}{\|X_1^{su}\wedge X_2^{su}\|}
|u|_{1,w}^2.
\end{equation}
Let us study the second term $u\Ne^\top (\bar u)$. Since
$$
\Ne^\perp=\sum_{i,j=1}^2g^{ij}\langle\Ne^\top,X_i^{su}\rangle X_j^{su},
$$
we get
$$
|\Ne^\top (\bar u)|\leq \sum_{i,j=1}^2|g^{ij}||\langle\Ne^\top,X_i^{su}\rangle ||X_j^{su}u|+|u|\sum_{i,j=1}^2|g^{ij}||\langle\Ne^\top,X_i^{su}\rangle|(\Vert D_{X_j^{su}}\Ne\Vert+\Vert D_{X_j^{su}}\Nsu\Vert).
$$
Moreover, it holds
\begin{align*}
   |\langle \Ne^\top,X_i^{su}\rangle|&=|\langle N_\eps-\Nsu,X_i^{su} \rangle|
   \leq \sum_{k=1}^2|(\Ne-\Nsu)_k||(X_i^{su})_k|+|(\Ne-\Nsu)_3||(X_i^{su})_3|\\&\leq C_L |u|_{1,w}\er^{i-1} +C_L \er|u|_{1,w}\er^{i-2}
   \\&\leq C_L\er^{i-1}|u|_{1,w},
\end{align*}
where we used \eqref{Socrates}, \eqref{ZED} and \eqref{MORON}.
From this and the estimates  \eqref{STILO1}, \eqref{gigi}, \eqref{Da1}, \eqref{EL21}, \eqref{EL22}, we obtain
\begin{equation}\label{second}
\begin{aligned}
|\Ne^\top (\bar u)|&\leq C_L \sum_{i,j=1}^2 \er^{2-i-j}\er^{i-1} \er ^{j-1}|u|_{1,w} |\nabla u|
+\frac{C_L |u|}{r}\sum_{i,j=1}^2 \er^{2-i-j}\er^{i-1} \er ^{j-1}|u|_{1,w} \\&\leq C_L |u|_{1,w}^2,
\end{aligned}
\end{equation}
In particular, putting together \eqref{first} and \eqref{second}, we have
$$
|\bar u\Nsu(\bar u)|\leq C_L |u||u|_{1,w}^2,
$$
  and thus  
\[
\begin{split}
\left|
\int_{\Se^{su}}
 \bar u\Nsu(\bar u) \,d\Ae
\right|
&=
\left|
\int_{\Se}
 \bar u\Nsu(\bar u) (\Psisu)
\|X_1^{su}\wedge X_2^{su}\|\,d\Ae
\right|\\
&\leq
C_L\int_{\Se}|u||u|_{1,w}^2\,d\Ae\\
&\leq
C_L\int_{\Se}
\left(
\frac{u^2}{r}+|\nabla u|^2
\right)|u|_{1,w}\,d\Ae\\
&\leq
C_L\int_{\Se}
\left(
\frac{u^2}{r^2}+|\nabla u|^2
\right)|u|_{1,w}\,d\Ae.
\end{split}
\]
The last integral is estimated by splitting $\Se$ into the regions
$|z|\leq\varrho$ and $|z|>\varrho$ and arguing as in
\eqref{noce}, \eqref{mandorlo1}, and \eqref{mandorlo2}. This proves
\eqref{last}.
\end{proof}

\begin{remark}\label{remo}
The estimate for $\widetilde E_\eps^{su}$ is similar to the estimate for  $E_\eps^{su}$ and  in fact easier.
It reads
\begin{equation}
\label{hacco1}
\big|\widetilde E_\eps^{su}\big|
\leq
\eta\int_{\Se}u^2\,d\Ae.
\end{equation}
\end{remark}

Theorem \ref{unifcontSecondVar} now follows by combining Theorems
\ref{unifcontgradient}, \ref{unifcontQeps},
\ref{unifcontHeps},  \ref{esso}, \ref{unifcontE4}, and Remark \ref{remo} applied with suitably smaller
values of $\eta$.

\section{Fuglede theorem for $\So$}
\label{FUG}

In this section, we complete the proof of Theorem \ref{F0}. Fix
$\gamma>0$ and $L>0$, and let $u:\So\to\rr$ satisfy assumptions
(1)--(6). Let $0<\eps<\eps_0$, where $\eps_0>0$ is sufficiently small
that Theorem \ref{8.1}, Lemmas \ref{L5}, \ref{leox}, and \ref{Lex3}, 
Remark \ref{tris}, and Theorem \ref{unifcontSecondVar} apply. By
Theorem \ref{8.1},
\begin{equation}
\label{USeBIS}
\int_{\Se}
\left(
\|\nablaSe u\|^2-Q_\eps u^2
\right)\,d\Ae
\geq
\frac{m}{8}\int_{\Se}u^2\,d\Ae,
\end{equation}
where $m=m(\gamma,L)>0$ is independent of $\eps$.

Multiply both sides of \eqref{USeBIS} by $1-\lambda$, where
$\lambda\in(0,1)$ will be fixed below. Rearranging the terms, we obtain 
\[
\begin{split}
\int_{\Se}
\left(
\|\nablaSe u\|^2-Q_\eps u^2
\right)\,d\Ae
\geq{}&
\lambda\int_{\Se}\|\nablaSe u\|^2\,d\Ae\\
&+
\frac{m(1-\lambda)}{8}\int_{\Se}u^2\,d\Ae
-\lambda\int_{\Se}Q_\eps u^2\,d\Ae.
\end{split}
\]
By assumption (4), \eqref{lampo}, and the comparison of the surface
measures, for $\eps\leq1$ we have
\[
\begin{split}
\int_{\Se}Q_\eps u^2\,d\Ae
&\leq
L\left(\int_{\Se}r^\gamma Q_\eps\,d\Ae\right)
 \left(\int_{\So}u^2\,d\Ao\right)\\
&\leq
2\eps L
\left(\int_{\Se}Q_\eps\,d\Ae\right)
\left(\int_{\Se}u^2\,d\Ae\right)\\
&\leq
16\pi L\int_{\Se}u^2\,d\Ae.
\end{split}
\]
Choose $\lambda\in(0,1)$ such that
\[
\frac{m(1-\lambda)}{16}=16\pi\lambda L.
\]
It follows that there exists a constant $m_{\gamma,L}>0$ such that
\begin{equation}
\label{USeTER}
\int_{\Se}
\left(
\|\nablaSe u\|^2-Q_\eps u^2
\right)\,d\Ae
\geq
m_{\gamma,L}
\int_{\Se}
\left(
\|\nablaSe u\|^2+u^2
\right)\,d\Ae.
\end{equation}

Choose the parameter $\eta>0$ in Theorem
\ref{unifcontSecondVar} so that
\begin{equation}
\label{ETAETA}
\eta=\frac{m_{\gamma,L}}{2},
\end{equation}
and let $\delta=\delta(\eta;\gamma,L)>0$ be the corresponding constant.
Assume that $|u|_{2,w}\leq\delta$ when $|z|>\varrho$. By assumption
(1), we have $|u|_{2,w}\leq C$ when $|z|\leq\varrho$.

The Taylor expansion \eqref{G1} of the isoperimetric functional
$\Fe=\Ae-H\Ve$, together with \eqref{USeTER} and \eqref{FEP}, gives
\begin{equation}
\label{EXPA}
\begin{split}
\Fe(u)-\Fe(0)
={}&
\frac{1}{2}
\int_{\Se}
\left(
\|\nablaSe u\|^2-Q_\eps u^2
\right)\,d\Ae\\
&+
\int_0^1(1-s)
\bigl(\Fep''(s;u)-\Fep''(0;u)\bigr)\,ds\\
\geq{}&
\left(
\frac{m_{\gamma,L}}{2}-\frac{\eta}{2}
\right)
\int_{\Se}
\left(
\|\nablaSe u\|^2+u^2
\right)\,d\Ae\\
\geq{}&
\frac{m_{\gamma,L}}{4}
\int_{\Se}
\left(
\|\nablaSe u\|^2+u^2
\right)\,d\Ae.
\end{split}
\end{equation}

We next compare $\Ve(\Be^u\triangle\Be)$ with
$\|u\|_{L^2(\Se)}$, where $\Be^u$ is the bounded domain such that
$\partial\Be^u=\Se^u$. Let
\[
\de=\operatorname{dist}_\eps(\,\cdot\,;\Se)
\]
be the distance from $\Se$ with respect to $\gep$. In a sufficiently
small tubular neighbourhood of $\Se$, the function $\de$ satisfies the
eikonal equation
\[
\|\nabla_\eps\de\|=1.
\]
For $0<s<\max u$, set
\[
\Sigma_s
 =\{q\in\hh^1:\de(q)=s\}\cap(\Be^u\setminus\Be).
\]
By the coarea formula (see Theorem 1.3 in
\cite{Ritore2023}),
\[
\begin{split}
\Ve(\Be^u\setminus\Be)
&=
\int_{\Be^u\setminus\Be}d\Ve\\
&=
\int_{\Be^u\setminus\Be}\|\nabla_\eps\de\|\,d\Ve\\
&=
\int_0^{\max u}
\left(\int_{\Sigma_s}d\Ae\right)\,ds\\
&=
\int_0^{\max u}\Ae(\Sigma_s)\,ds.
\end{split}
\]

Let $D_s\subset\Se$ be the projection of $\Sigma_s$ onto $\Se$ along
the normal flow $\Psi_\varepsilon$. Equivalently,
\[
\Sigma_s=\Psi_\varepsilon^s(D_s),
\qquad
D_s=\{p\in\Se:u(p)>s\}.
\]
For $0<s<\max u\leq\delta_0$, with $\delta_0>0$ sufficiently small, 
\[
\Ae(\Sigma_s)\leq2\Ae(D_s).
\]
Consequently,
\begin{equation}
\label{symmdiff1}
\begin{split}
\Ve(\Be^u\setminus\Be)
&\leq
2\int_0^{\max u}
\left(\int_{\Se}\chi_{D_s}(p)\,d\Ae\right)\,ds\\
&=
2\int_{\Se}
\left(\int_0^{u_+(p)}ds\right)\,d\Ae\\
&=
2\int_{\Se}u_+(p)\,d\Ae,
\end{split}
\end{equation}
where $u_+=\max\{u,0\}$. Similarly,
\begin{equation}
\label{symmdiff2}
\Ve(\Be\setminus\Be^u)
\leq
2\int_{\Se}u_-(p)\,d\Ae,
\end{equation}
where $u_-=\max\{-u,0\}$. Adding \eqref{symmdiff1} and
\eqref{symmdiff2}, and applying Hölder's inequality, we obtain
\begin{equation}
\label{HO}
\Ve(\Be^u\triangle\Be)
\leq
2\int_{\Se}|u|\,d\Ae
\leq
2\Ae(\Se)^{1/2}\|u\|_{L^2(\Se)}.
\end{equation}

It follows from \eqref{EXPA} and \eqref{HO} that, for
$0<\eps<\eps_0$,
\[
\Ae(\Se^u)-\Ae(\Se)
-H\bigl(\Ve(\Be^u)-\Ve(\Be)\bigr)
\geq
\frac{m_{\gamma,L}}{16\Ae(\Se)}
\Ve(\Be^u\triangle\Be)^2.
\]
Multiplying both sides by $\eps>0$ and using the convergences
\[
\eps\Ae\to\Ao,
\qquad
\eps\Ve\to\mathscr L^3,
\]
we obtain 
\begin{equation}
\label{FINAL}
\Ao(\So^u)-\Ao(\So)
-H\bigl(\mathscr L^3(B_0^u)-\mathscr L^3(B_0)\bigr)
\geq
\frac{m_{\gamma,L}}{16\Ao(\So)}
\mathscr L^3(B_0^u\triangle B_0)^2.
\end{equation}
We have also used the fact that the area of the portion of $\Se^u$
parametrised as a normal graph over
$\Se\cap\{|z|\leq\varrho\}$ becomes negligible after multiplication by
$\eps$ as $\eps\to0$.

By assumption (6) in Theorem \ref{F0}, inequality \eqref{FINAL} yields \eqref{QI}.

\section{Appendix 1}

In this appendix, we prove Theorem \ref{INJ}. The map $\Psi_0$ is a
local diffeomorphism because its differential is invertible.
Surjectivity follows directly from the definitions of $t^-$ and
$t^+$. It therefore remains to prove global injectivity.

For $p\in\So$, the curve
\[
[t^-(r),t^+(r)]\ni s\mapsto
\gamma_p(s)=\Psi_0(p,s)
\]
is a Carnot--Carath\'eodory geodesic passing through $p$ and satisfying
\[
\gamma_p(t^-(r))\in Z,
\qquad
\gamma_p(t^+(r))\in Z.
\]
Its rotation about the $Z$-axis is therefore a Pansu sphere with centre
$(0,c(r))\in Z$ and radius $\lambda(r)>0$, for some $c(r)\in\rr$.
When $p$ lies on the equator of $\So$, this sphere degenerates into the
$xy$-plane.

We claim that the Pansu spheres corresponding to points $p\in\So$ on
distinct ``parallels'' in the same hemisphere, that is, to distinct
values of $r\in(0,1)$, are disjoint. In fact, they form a foliation of
$\hh^1\setminus Z$ whose leaves intersect $\So$ orthogonally in the
horizontal sense.

\medskip
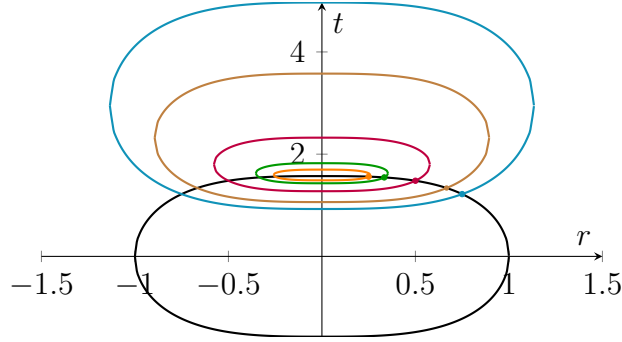
\begin{center}

\begin{tikzpicture}
\begin{axis}[
    width=9cm, height=6cm,
    axis lines=middle,
    xlabel={$r$}, ylabel={$t$},
    xmin=-1.5, xmax=1.5,
]

\addplot[domain=-1:1, samples=200, smooth, black, thick]
    {abs(x)*sqrt(1-x^2) + acos(abs(x))};

\addplot[domain=-1:1, samples=200, smooth, black, thick]
    {-(abs(x)*sqrt(1-x^2) + acos(abs(x)))};
    
% ---------- k = 1/4 ----------
\pgfmathsetmacro{\kA}{1/4}
\pgfmathsetmacro{\gA}{\kA/sqrt(1-\kA^2)}
\pgfmathsetmacro{\yA}{\kA*sqrt(1-\kA^2)+acos(\kA)}
\pgfmathsetmacro{\fA}{abs(\kA)*sqrt(1-\kA^2) + acos(abs(\kA))+\kA^2/(1-\kA^2)*atan(\gA)+\kA^3/sqrt(1-\kA^2)}
\pgfmathsetmacro{\aminA}{-\gA}
\pgfmathsetmacro{\amaxA}{+\gA}
\addplot[domain=\aminA:\amaxA, samples=120, smooth, orange, thick]
    {abs(x)*sqrt(\gA^2-(x)^2) + \gA^2*acos(abs(x)/\gA) + \fA};
\addplot[domain=\aminA:\amaxA, samples=120, smooth, orange, thick, forget plot]
    {-abs(x)*sqrt(\gA^2-(x)^2) - \gA^2*acos(abs(x)/\gA) + \fA};
%\addplot[only marks, mark=*, mark size=1pt, orange, forget plot] coordinates {(0,\fA)};
\addplot[only marks, mark=*, mark size=1pt, orange, forget plot] coordinates {(\kA,\yA)};
% ---------- k = 1/3 ----------
\pgfmathsetmacro{\kB}{1/3}
\pgfmathsetmacro{\gB}{\kB/sqrt(1-\kB^2)}
\pgfmathsetmacro{\yB}{\kB*sqrt(1-\kB^2)+acos(\kB)}
\pgfmathsetmacro{\fB}{abs(\kB)*sqrt(1-\kB^2) + acos(abs(\kB))+\kB^2/(1-\kB^2)*atan(\gB)+\kB^3/sqrt(1-\kB^2)}
\pgfmathsetmacro{\aminB}{-\gB}
\pgfmathsetmacro{\amaxB}{+\gB}
\addplot[domain=\aminB:\amaxB, samples=120, smooth, green!60!black, thick]
    {abs(x)*sqrt(\gB^2-(x)^2) + \gB^2*acos(abs(x)/\gB) + \fB};
\addplot[domain=\aminB:\amaxB, samples=120, smooth, green!60!black, thick, forget plot]
    {-abs(x)*sqrt(\gB^2-(x)^2) - \gB^2*acos(abs(x)/\gB) + \fB};
%\addplot[only marks, mark=*, mark size=1pt, green!60!black, forget plot] coordinates {(0,\fB)};
\addplot[only marks, mark=*, mark size=1pt, green!60!black, forget plot] coordinates {(\kB,\yB)};
% ---------- k = 1/2 ----------
\pgfmathsetmacro{\kC}{1/2}
\pgfmathsetmacro{\gC}{\kC/sqrt(1-\kC^2)}
\pgfmathsetmacro{\yC}{\kC*sqrt(1-\kC^2)+acos(\kC)}
\pgfmathsetmacro{\fC}{abs(\kC)*sqrt(1-\kC^2) + acos(abs(\kC))+\kC^2/(1-\kC^2)*atan(\gC)+\kC^3/sqrt(1-\kC^2)}
\pgfmathsetmacro{\aminC}{-\gC}
\pgfmathsetmacro{\amaxC}{+\gC}
\addplot[domain=\aminC:\amaxC, samples=120, smooth, purple, thick]
    {abs(x)*sqrt(\gC^2-(x)^2) + \gC^2*acos(abs(x)/\gC) + \fC};
\addplot[domain=\aminC:\amaxC, samples=120, smooth, purple, thick, forget plot]
    {-abs(x)*sqrt(\gC^2-(x)^2) - \gC^2*acos(abs(x)/\gC) + \fC};
%\addplot[only marks, mark=*, mark size=1pt, purple, forget plot] coordinates {(0,\fC)};
\addplot[only marks, mark=*, mark size=1pt, purple, forget plot] coordinates {(\kC,\yC)};
% ---------- k = 2/3 ----------
\pgfmathsetmacro{\kD}{2/3}
\pgfmathsetmacro{\gD}{\kD/sqrt(1-\kD^2)}
\pgfmathsetmacro{\yD}{\kD*sqrt(1-\kD^2)+acos(\kD)}
\pgfmathsetmacro{\fD}{abs(\kD)*sqrt(1-\kD^2) + acos(abs(\kD))+\kD^2/(1-\kD^2)*atan(\gD)+\kD^3/sqrt(1-\kD^2)}
\pgfmathsetmacro{\aminD}{-\gD}
\pgfmathsetmacro{\amaxD}{+\gD}
\addplot[domain=\aminD:\amaxD, samples=120, smooth, brown, thick]
    {abs(x)*sqrt(max(0,\gD^2-(x)^2)) + \gD^2*acos(min(1,abs(x)/\gD)) + \fD};
\addplot[domain=\aminD:\amaxD, samples=120, smooth, brown, thick, forget plot]
    {-abs(x)*sqrt(max(0,\gD^2-(x)^2)) - \gD^2*acos(min(1,abs(x)/\gD)) + \fD};
%\addplot[only marks, mark=*, mark size=0.8pt, brown, forget plot] coordinates {(0,\fD)};
\addplot[only marks, mark=*, mark size=0.8pt, brown, forget plot] coordinates {(\kD,\yD)};

% ---------- k = 3/4 ----------
\pgfmathsetmacro{\kE}{3/4}
\pgfmathsetmacro{\gE}{\kE/sqrt(1-\kE^2)}
\pgfmathsetmacro{\yE}{\kE*sqrt(1-\kE^2)+acos(\kE)}
\pgfmathsetmacro{\fE}{abs(\kE)*sqrt(1-\kE^2) + acos(abs(\kE))+\kE^2/(1-\kE^2)*atan(\gE)+\kE^3/sqrt(1-\kE^2)}
\pgfmathsetmacro{\aminE}{-\gE}
\pgfmathsetmacro{\amaxE}{+\gE}
\addplot[domain=\aminE:\amaxE, samples=120, smooth, cyan!70!black, thick]
    {abs(x)*sqrt(\gE^2-(x)^2) + \gE^2*acos(abs(x)/\gE) + \fE};
\addplot[domain=\aminE:\amaxE, samples=120, smooth, cyan!70!black, thick, forget plot]
    {-abs(x)*sqrt(\gE^2-(x)^2) - \gE^2*acos(abs(x)/\gE) + \fE};
%\addplot[only marks, mark=*, mark size=1pt, cyan!70!black, forget plot] coordinates {(0,\fE)};
\addplot[only marks, mark=*, mark size=1pt, cyan!70!black, forget plot] coordinates {(\kE,\yE)};

\end{axis}
\end{tikzpicture}

\end{center}

\medskip

The radius $\lambda(r)>0$ can be computed by maximising, over
$s\in[t^-(r),t^+(r)]$, the function
\[
|\Psi_{0,1}(p,s)|^2
 =
\frac{r^2}{2(1-r^2)}
\bigl(1+\cos(2(\theta_0-\tau p_0s))\bigr).
\]
The maximum is attained at $s=s_\lambda$, where
$
\theta_0-\tau p_0s_\lambda=0.
$
Thus,
\[
s_\lambda
 =
\frac{\theta_0}{\tau p_0}
 =
\frac{r}{\sqrt{1-r^2}}
\arctan\left(\frac{r}{\sqrt{1-r^2}}\right),
\]
and consequently
\[
\lambda(r)
 =
|\Psi_{0,1}(p,s_\lambda)|
 =
\frac{r}{\sqrt{1-r^2}}.
\]

The height coordinate of the centre is
$c(r)=\Psi_{0,2}(p,s_\lambda)$; namely,
\[
c(r)
 =
f_0(r)
+\frac{\tau}{2}\frac{r^2}{1-r^2}
 \arctan\left(\frac{r}{\sqrt{1-r^2}}\right)
+\frac{\tau}{2}\frac{r^3}{\sqrt{1-r^2}},
\]
where $f_0$ is the profile function of the Pansu sphere in
\eqref{PansuProfile} with $\varepsilon=0$.

The point $(0,c(r))\in\cc\times\rr^+$ is the centre of the Pansu sphere
of radius $\lambda(r)$ passing through $(z,f_0(r))\in\So$, where
$r=|z|$. We claim that two such spheres corresponding to distinct
values of $r$ do not intersect. It is enough to prove that, for each
fixed horizontal radius $\rho$, the functions
\[
r\mapsto c(r)+f_0(\rho,\lambda(r))
\quad\text{and}\quad
r\mapsto c(r)-f_0(\rho,\lambda(r))
\]
are increasing and decreasing, respectively. Here,
$f_0(\rho,\lambda(r))$ denotes the profile, as a function of $\rho$, of
the Pansu sphere centred at the origin and having radius $\lambda(r)$:
\[
f_0(\rho,\lambda(r))
 =
\frac{\tau}{2}
\left[
\rho\sqrt{\lambda^2(r)-\rho^2}
+\lambda^2(r)
 \arccos\left(\frac{\rho}{\lambda(r)}\right)
\right].
\]

Notice that
$
r\mapsto\lambda(r)
 =\frac{r}{\sqrt{1-r^2}}
$
is increasing. We therefore express $c(r)$ as a function of $\lambda$
and prove the required monotonicity with respect to $\lambda$. Indeed,
\begin{align*}
c(r)
&=
f_0(r)
+\frac{\tau}{2}\frac{r^2}{1-r^2}
 \arctan\left(\frac{r}{\sqrt{1-r^2}}\right)
+\frac{\tau}{2}\frac{r^3}{\sqrt{1-r^2}}\\
&=
\frac{\tau}{2}\lambda
+\frac{\tau}{2}
 \left(\frac{\pi}{2}-\arctan\lambda\right)
+\frac{\tau}{2}\lambda^2\arctan\lambda\\
&=
\frac{\tau}{2}
\left[
\frac{\pi}{2}+\lambda
+(\lambda^2-1)\arctan\lambda
\right]
=:C(\lambda).
\end{align*}
Here we have used the identities
\[
\arccos r
+\arctan\left(\frac{r}{\sqrt{1-r^2}}\right)
 =\frac{\pi}{2}
\]
and
\[
r\sqrt{1-r^2}
+\frac{r^3}{\sqrt{1-r^2}}
 =\lambda.
\]

We claim that, for each fixed $\rho>0$, the function
$
\lambda\mapsto C(\lambda)+f_0(\rho,\lambda)
$
is increasing when $\lambda>\rho$. Indeed,
\[
C'(\lambda)
 =
\frac{\tau}{2}
\left[
1+2\lambda\arctan\lambda
+\frac{\lambda^2-1}{\lambda^2+1}
\right]
>0,
\]
and
\[
\frac{\partial}{\partial\lambda}f_0(\rho,\lambda)
 =
\tau\left[
\frac{\rho\lambda}{\sqrt{\lambda^2-\rho^2}}
+\lambda\arccos\left(\frac{\rho}{\lambda}\right)
\right]
>0.
\]

We next claim that
$
\lambda\mapsto C(\lambda)-f_0(\rho,\lambda)
$
is decreasing. It is enough to prove that
\[
1+2\lambda\arctan\lambda
+\frac{\lambda^2-1}{\lambda^2+1}
<
\frac{2\lambda\rho}{\sqrt{\lambda^2-\rho^2}}
+2\lambda\arccos\left(\frac{\rho}{\lambda}\right).
\]
Using the identity
\[
\arccos\left(\frac{\rho}{\lambda}\right)
 =
\frac{\pi}{2}
-\arctan\left(
\frac{\rho}{\sqrt{\lambda^2-\rho^2}}
\right),
\]
this inequality is equivalent to
\[
1
+2\lambda\left[
\arctan\left(
\frac{\rho}{\sqrt{\lambda^2-\rho^2}}
\right)
-\arctan\left(\frac{1}{\lambda}\right)
\right]
+\frac{\lambda^2-1}{\lambda^2+1}
<
\frac{2\lambda\rho}{\sqrt{\lambda^2-\rho^2}}.
\]
This follows from
\[
\arctan\left(
\frac{\rho}{\sqrt{\lambda^2-\rho^2}}
\right)
<
\frac{\rho}{\sqrt{\lambda^2-\rho^2}}
\]
and
\[
1-2\lambda\arctan\left(\frac{1}{\lambda}\right)
+\frac{\lambda^2-1}{\lambda^2+1}
\leq0.
\]
This completes the proof of the foliation property.

%%%% FINO QUI. - INGLESE

\section{Appendix 2}

In this Appendix we give the details of the proof of Lemma \ref{derivX_i}.

\begin{proof}[Proof of \eqref{EL1} and \eqref{EL1bis}]

We use the following notation:
\[
\begin{split}
w_1 & =\frac{1}{r}\sqrt{\frac{\eps^2+\tau^2r^2}{\eps^2+\tau^2}}z(\tau p_\eps +i)=:f_1(r)zg_1(r),
\\
 w_2&=\frac{\eps}{r}\sqrt{\frac{1-r^2}{\eps^2+\tau^2}}z(1-\tau p_\eps i)=:f_2(r)zg_2(r).
 \end{split}
 \]
 By the definition in \eqref{NOTA1} for $X_1$ and $X_2$, and denoting by $Jw_k$ the $2\times 2$ Jacobian matrix of $w_k$,
 we have:
 \[
\begin{split}
X_1w_k &= \frac{1}{r}\sqrt{\frac{\eps^2+\tau^2r^2}{\eps^2+\tau^2}}(\tau p_\eps Jw_kz+Jw_kiz)
\\ 
X_2w_k &= \frac{\eps}{r}\sqrt{\frac{1-r^2}{\eps^2+\tau^2}}(Jw_kz-\tau p_\eps Jw_kiz),
\end{split}
\]
where
\[
\begin{split}
Jw_1 &=\frac{f_1'(r)}{r}z(\tau p_\eps+i)z^t+f_1(r)\begin{pmatrix}
   \tau p_\eps& -1\\1 &\tau p_\eps
\end{pmatrix}+\frac{f_1(r)}{r}\tau p_\eps'zz^t,
\\ 
Jw_2 &=\frac{f_2'(r)}{r}z(1-\tau p_\eps i)z^t+f_2(r)\begin{pmatrix}
   1& \tau p_\eps\\-\tau p_\eps &1
\end{pmatrix}-\frac{f_2(r)}{r}\tau p_\eps'izz^t.
\end{split}
\]
So we obtain
\[
\begin{split}
X_1w_1 &= \frac{1}{r}\sqrt{\frac{\eps^2+\tau^2r^2}{\eps^2+\tau^2}}(\tau p_\eps (f_1'rz(\tau p_\eps+i)+f_1z(\tau p_\eps+i)+f_1r\tau p_\eps'z)-f_1z(1-\tau p_\eps i)),
\\
X_1w_2&= \frac{1}{r}\sqrt{\frac{\eps^2+\tau^2r^2}{\eps^2+\tau^2}}(\tau p_\eps (f_2'rz(1-\tau p_\eps i)+f_2z(1-\tau p_\eps i)-f_2r\tau p_\eps'iz)+f_2z(\tau p_\eps+i)),
\\
X_2w_1&= \frac{\eps}{r}\sqrt{\frac{1-r^2}{\eps^2+\tau^2}}(f_1'rz(\tau p_\eps +i)+f_1z(\tau p_\eps+ i)+f_1r\tau p_\eps'z+\tau p_\eps f_1z(1-\tau p_\eps i)),
\\
X_2w_2&= \frac{\eps}{r}\sqrt{\frac{1-r^2}{\eps^2+\tau^2}}(f_2'rz(1-\tau p_\eps i)+f_2z(1-\tau p_\eps i)-f_2r\tau p_\eps'iz-\tau p_\eps f_2z(\tau p_\eps+i)).
\end{split}
\]
We now  use  
\[
f_1'=-\frac{\eps^2}{r^2\sqrt{(\eps^2+\tau^2)(\eps^2+\tau^2r^2)}},\quad f_2'=-\frac{\eps}{r^2\sqrt{(\eps^2+\tau^2)(1-r^2)}},
\]
\[
|g_1|=|g_2|=\sqrt{\frac{\eps^2+\tau^2}{\eps^2+\tau^2r^2}},\quad
 |p_\eps'|=\frac{(\eps^2+\tau^2)r}{\sqrt{1-r^2}(\sqrt{\eps^2+\tau^2r^2})^3},
\]
and, after elementary estimates, we conclude that for $i,k=1,2$ there holds
\[
|X_iw_k|\leq \frac{C}{r}\left(\frac{\eps}{\sqrt{\eps^2+\tau^2r^2}}\right)^{i+k-2}.
\]
This is \eqref{EL1}.

\end{proof}

For the proof of \eqref{EL1.2} and \eqref{EL1.3}, we 
need some preliminary estimates.

%
%Notice that, by definition of $X_i $, we have 
%\begin{align*}
%X_i(\Nsu)_j|_{\Psisu(p)}&=X_i|_p(\Nsu(\Psisu(p)))_j=X_i|_p\left(\frac{X_1(\Psisu(p))\wedge X_2(\Psisu(p))}{|X_1\wedge X_2|(\Psisu(p))}\right)_j\\&=\frac{X_i|_p(X_1(\Psisu(p))\wedge X_2(\Psisu(p)))_j-(\Nsu)_j\langle\Nsu,X_i|_p(X_1(\Psisu(p))\wedge X_2(\Psisu(p)))\rangle}{|X_1\wedge X_2|(\Psisu(p))},
%\end{align*}
%Hence we compute at  first the derivatives of the tangent vectors $X_1,X_2$. From \eqref{STELLA}, we have that
%$$
%X_k(\Psisu(p))=(\Psisu)_*X_k(p)=[w_k,c_k]_*+[v_k,e_k]_*,
%$$
%where Therefore we want to estimate the derivatives of $v_k$ and $e_k$.

\begin{lemma}\label{estimateslemma}
Let $u:\hh^1\to\rr$ be a function  such that $|u|\leq L|z|$ and  $|\nabla u|\leq L$, for some $L>0$. There exist constants $C>0 $ and 
 $C_L>0$, independent of  $\eps>0$, such that for $i,j=1,2$  we have:
    \begin{align}
    \label{EL4}
        |X_i(p_\eps u)|&\leq C\eps^{i-2}\er   |u|_{1,w};\\
    \label{EL5}
        |X_iX_j(p_\eps u)|&\leq   C\eps^{i+j-3}\er  \left\{\frac{|u|}{\eps^2+\tau^2r^2}+\frac{|\nabla u|}{\sqrt{\eps^2+\tau^2r^2}}+|\nabla^2 u|\right\};         \\
    \label{EL6}
        |X_i(u^2(1-\eps^2p_\eps^2))|&\leq C_L\eps^{i-1}|u|; \\
    \label{EL7}
        |X_iX_j(u^2(1-\eps^2p_\eps^2)|&\leq C_L \eps^{i+j-2}  |u|_{2,w}; 
        \\
    \label{EL8}
        |X_i(r^2u)|&\leq C_L r^2\eps^{i-1} |u|_{1,w}; \\
    \label{EL9}
        |X_iX_j(r^2u)|&\leq C_L  r \eps^{i+j-2}\ |u|_{2,w}.
    \end{align}
\end{lemma}

\begin{proof} 
From \eqref{EL2} and \eqref{STILO1}, we get \eqref{EL4}:
\begin{align*}
    |X_1(p_\eps u)|&\leq |X_1 p_\eps||u|+|p_\eps||X_1 u|\leq\frac{C}{(\eps^2+\tau^2r^2)}|u|+|p_\eps||\nabla u|\leq\frac{C}{\sqrt{\eps^2+\tau^2r^2}}
    |u|_{1,w}
    %    \left\{\frac{|u|}{r}+|\nabla u|\right\}
    \\
    |X_2(p_\eps u)|&\leq |X_2 p_\eps||u|+|p_\eps||X_2 u|\leq\frac{C\eps}{(\eps^2+\tau^2r^2)}|u|+ C \eps|p_\eps||\nabla u|\leq\frac{C\eps}{\sqrt{\eps^2+\tau^2r^2}}   |u|_{1,w}.
%    \left\{\frac{|u|}{r}+|\nabla u|\right\}.
\end{align*}

\noindent 
Now \eqref{EL5} follows from \eqref{EL2}:
\begin{align*}
|X_iX_j(p_\eps u)|&\leq|p_\eps||X_i X_ju|+|X_ip_\eps||X_ju|+|X_iu||X_jp_\eps|+|u||X_i X_jp_\eps|\\& \leq \frac{C\eps^{i+j-2}}{\sqrt{\eps^2+\tau^2r^2}}\left\{\frac{|u|}{\eps^2+\tau^2r^2}+\frac{|\nabla u|}{\sqrt{\eps^2+\tau^2r^2}}+|\nabla^2 u|\right\}.
\end{align*}

In the following, we start using the assumptions $|u|\leq L|z|$ and $|\nabla u|\leq L$. From  $X_k(u^2(1-\eps^2p_\eps^2))=2uX_ku-2\eps^2up_\eps X_k(up_\eps)$, we obtain \eqref{EL6}:
\[
\begin{split}
|X_1(u^2(1-\eps^2p_\eps^2))| & \leq C|\nabla u||u|+C\eps^2|u|p_\eps\left(|\nabla u|p_\eps+\frac{|u|}{\eps^2+\tau^2r^2}\right)\leq C_L|u|,
\\
|X_2(u^2(1-\eps^2p_\eps^2))|& \leq C\eps|\nabla u||u|+C\eps^2|u|p_\eps\left(\eps|\nabla u|p_\eps+\frac{\eps|u|}{\eps^2+\tau^2r^2}\right)\leq C_L\eps|u|.
\end{split}
\]

Analogously, from  $X_iX_k(u^2(1-\eps^2p_\eps^2))=2uX_iX_ku+2X_iuX_ku+2\eps^2X_i(up_\eps)X_k(up_\eps)+2\eps^2up_\eps X_iX_k(up_\eps)$, by \eqref{EL4} and \eqref{EL5}
we obtain \eqref{EL7}:
\begin{align*}
|X_iX_j(u^2(1-\eps^2p_\eps^2))|&\leq 
C|u|\eps^{i+j-2}|\nabla^2u|+C\eps^{i+j-2}|\nabla u|^2
+\frac{C\eps^{i+j}}{\eps^2+\tau^2r^2}
|u|_{1,w}^2  
\\&\quad+\eps^2|u|p_\eps\frac{C\eps^{i+j-2}}{\sqrt{\eps^2+\tau^2r^2}}\left\{\frac{|u|}{\eps^2+\tau^2r^2}+\frac{|\nabla u|}{r}+|\nabla^2u|\right\}    \\&\leq C_L\eps^{i+j-2} |u|_{2,w}.
\end{align*}

\noindent 
From $X_i(r^2u)=2urX_ir+r^2X_iu$, there follows \eqref{EL8}:
\[
  |X_i(r^2u)|\leq  Cr^2 \eps^{i-1} |u|_{1,w}.
  \]
 From $X_iX_j(r^2u)=2uX_irX_jr+2urX_iX_jr+2rX_jrX_iu+2rX_irX_ju+r^2X_iX_ju$, there follows \eqref{EL9}:
 \[
|X_iX_j(r^2u)|\leq  C_L\eps^{i+j-2} |u|_{2,w}.
\]
\end{proof}

\begin{proof}[Proof of \eqref{EL1.2} and \eqref{EL1.3}] Using the definition of $v_k$ in \eqref{LOP}, we compute,
for $i=1,2$:
\begin{align*}
    X_iv_k&=sX_iw_k u(\tau p_\eps+i)\beta_1+sw_kX_iu(\tau p_\eps+i)\beta_1+sw_ku\tau X_ip_\eps \beta_1+s w_k u(\tau p_\eps+i)\beta_1'2\tau s X_i(p_\eps u)\\&\quad+sX_iz\left( i (X_k u)\beta_1 + \tau X_k (u p_\varepsilon)\big(\beta_1 + 2 s u (\tau p_\varepsilon + i)\beta_1'\big) \right)\\&\quad+sz\left[iX_iX_ku \beta_1+iX_ku\beta_1'2\tau sX_i(p_\eps u)+\tau X_iX_k(up_\eps)\big(\beta_1 + 2 s u (\tau p_\varepsilon + i)\beta_1'\big)\right.\\&\quad \left.+\tau X_k(up_\eps)(\beta_1'2\tau sX_i(p_\eps u)+2s(\tau X_i(u p_\eps)+X_iu)\beta_1'+2su(\tau p_\eps+i)\beta_1''2\tau sX_i(p_\eps u))\right].
\end{align*}
Then, using the estimates of   Lemma \ref{estimateslemma} we obtain \eqref{EL1.2} (we omit the intermediate steps):
\begin{align}
    |X_i v_k| \leq \frac{C_L}{\sqrt{\eps^2+\tau^2r^2}}
  \er ^{i+k-2}
  |u|_{2,w}. 
\end{align}

Using the definition of $d_k$ in \eqref{LOP}, we compute,
for $i=1,2$: 
\begin{align*}
    X_id_k&=s\eps^2X_iX_k(up_\eps)+\tau s^2X_iX_k(u^2(1-\eps^2p_\eps^2))\beta_2+\tau s^2X_k(u^2(1-\eps^2p_\eps^2))\beta_2'2\tau sX_i(p_\eps u)\\&\quad+2\tau^2 s^3X_i(u^2(1-\eps^2p_\eps^2))\beta_2'X_k(up_\eps)+ 2 \tau^2 s^3 u^2 (1 - \varepsilon^2 p_\varepsilon^2)\beta_2''2\tau sX_i (up_\eps)X_k (u p_\varepsilon)\\ &\quad+  2 \tau^2 s^3 u^2 (1 - \varepsilon^2 p_\varepsilon^2)\beta_2' X_i X_k (u p_\varepsilon)+s \tau X_i X_k (r^2 u)(\beta_3 + \tau p_\varepsilon \beta_4)+s \tau X_k (r^2 u)\tau X_ip_\eps\beta_4\\&\quad+s \tau X_k (r^2 u)(\beta_3' + \tau p_\varepsilon \beta'_4)2\tau sX_i(p_\eps u)+ 2 \tau^2 s^2 X_i(r^2 u) (\beta_3' + s \tau p_\varepsilon \beta_4') X_k (u p_\varepsilon)\\&\quad + 2 \tau^2 s^2 r^2 u (\beta_3'' + s \tau p_\varepsilon \beta_4'') 2\tau s X_i(p_\eps u) X_k (u p_\varepsilon)+ 2 \tau^2 s^2 r^2 u s \tau X_i p_\varepsilon \beta_4'X_k (u p_\varepsilon)\\&\quad+ 2 \tau^2 s^2 r^2 u (\beta_3' + s \tau p_\varepsilon \beta_4') X_iX_k (u p_\varepsilon)+s \tau^2 X_i(r^2 u) X_k p_\varepsilon \beta_4+s \tau^2 r^2 u X_iX_k p_\varepsilon\beta_4\\&\quad +s \tau^2 r^2 u X_k p_\varepsilon \beta'_4 2\tau s X_i(p_\eps u).
\end{align*}
Then, using the estimates of   Lemma \ref{estimateslemma} we obtain (we omit the intermediate steps):
\begin{align*}
|X_id_k| \leq  C_L\eps^{i+k-2}
|u|_{2,w}
% \left\{\frac{|u|}{r}+|\nabla u|+r|\nabla^2 u|\right\}
 + C_L \eps^{i+k-1}|\nabla^2 u|.
\end{align*}

Finally, from the definition of $e_k$ in \eqref{STOP}, we have: 
\[
|X_ie_k|\leq  \frac{1}{\eps}\left\{|X_id_k|+|X_iz||v_k|+r|X_iv_k|+C_L|X_iu|r\frac{|w_k+v_k|}{\sqrt{\eps^2+\tau^2r^2}}
 +|u|X_i(\overline{z}(\tau p_\eps-i)\overline{\beta_1}(w_k+v_k))\right\},
\]
and, by   the estimates of   Lemma \ref{estimateslemma}, we obtain:
%\begin{align*}
%    \eps|X_ie_k|&\leq s C_L \eps^{i+k-2}\left\{\frac{|u|}{r}+|\nabla u|+r|\nabla^2 u|\right\}+s C_L \eps^{i+k-1}|\nabla^2 u|\\&+sC_L\eps^{i-1}\left(\frac{\eps}{\sqrt{\eps^2+\tau^2r^2}}\right)^{k-1}\left\{\frac{|u|}{r}+|\nabla u|\right\}\\&+r\frac{s C_L}{\sqrt{\eps^2+\tau^2r^2}}\left(\frac{\eps}{\sqrt{\eps^2+\tau^2r^2}}\right)^{i+k-2}\left\{\frac{|u|}{r}+|\nabla u|+r|\nabla^2 u|\right\}\\&+sC_L\eps^{i-1}|\nabla u|r\frac{C_L}{\sqrt{\eps^2+\tau^2r^2}}\left(\frac{\eps}{\sqrt{\eps^2+\tau^2r^2}}\right)^{k-1}+s|u|\eps^{i-1}\frac{C_L}{\sqrt{\eps^2+\tau^2r^2}}\left(\frac{\eps}{\sqrt{\eps^2+\tau^2r^2}}\right)^{k-1}\\&+s|u|r\frac{C_L\eps^{i-1}}{\eps^2+\tau^2r^2}\left(\frac{\eps}{\sqrt{\eps^2+\tau^2r^2}}\right)^{k-1}+s|u|r\frac{C_L}{\sqrt{\eps^2+\tau^2r^2}}\frac{C\eps^{i-1}}{\sqrt{\eps^2+\tau^2r^2}}\left\{\frac{|u|}{r}+|\nabla u|\right\}\left(\frac{\eps}{\sqrt{\eps^2+\tau^2r^2}}\right)^{k-1}\\&+s|u|r\frac{C_L}{\sqrt{\eps^2+\tau^2r^2}}\left(\frac{C}{r}\left(\frac{\eps}{\sqrt{\eps^2+\tau^2r^2}}\right)^{i+k-2}+\frac{s C_L}{\sqrt{\eps^2+\tau^2r^2}}\left(\frac{\eps}{\sqrt{\eps^2+\tau^2r^2}}\right)^{i+k-2}\left\{\frac{|u|}{r}+|\nabla u|+r|\nabla^2 u|\right\}\right),
%\end{align*}
%
\[
|X_ie_k|\leq \frac{C_L}{\eps }
\er ^{i+k-2} 
|u|_{2,w} +C_L\eps^{i+k-2}|\nabla^2 u|.
\]
This is \eqref{EL1.3}.
\end{proof}

\bibliographystyle{plain}

\end{document}